\documentclass{svjour3}       
\journalname{}
\usepackage[
    backend=biber,
    style=alphabetic,
    ]{biblatex}
\smartqed  
\usepackage{graphicx}
\usepackage{lipsum}
\usepackage{amsfonts,amsmath,amssymb}
\usepackage{graphicx}
\usepackage{epstopdf}
\usepackage{algorithm}
\usepackage{algorithmic}

\usepackage{commath,amsopn}
\usepackage{multirow}
\usepackage{hyperref}
\hypersetup{colorlinks=true,linkcolor=blue,citecolor=blue}
\usepackage{stmaryrd}
\usepackage{xifthen}
\usepackage{booktabs}
\usepackage{graphicx}
\usepackage{subfigure}
\usepackage{mathtools}
\usepackage{epstopdf}

\usepackage{tablefootnote}
\usepackage{extarrows}
\usepackage{colortbl}
\usepackage{adjustbox}

\newcommand{\revise}[1]{{#1}}

\newcommand{\vece}{\mathbf{e}}
\newcommand{\vecl}{\boldsymbol{\ell}}
\newcommand{\vecx}{\mathbf{x}}
\newcommand{\vecr}{\mathbf{r}}
\newcommand{\vecu}{\mathbf{u}}
\newcommand{\vecv}{\mathbf{v}}
\newcommand{\vecy}{\mathbf{y}}

\newcommand{\mata}{\mathbf{A}}
\newcommand{\matb}{\mathbf{B}}
\newcommand{\matc}{\mathbf{C}}
\newcommand{\matd}{\mathbf{D}}
\newcommand{\mate}{\mathbf{E}}
\newcommand{\matf}{\mathbf{F}}
\newcommand{\matx}{\mathbf{X}}
\newcommand{\matu}{\mathbf{U}}
\newcommand{\matv}{\mathbf{V}}
\newcommand{\matW}{\mathbf{W}}
\newcommand{\maty}{\mathbf{Y}}

\newcommand{\matC}{\mathbf{C}}

\newcommand{\matE}{\mathbf{E}}
\newcommand{\matG}{\mathbf{G}}
\newcommand{\matI}{\mathbf{I}}

\newcommand{\matR}{\mathbf{R}}
\newcommand{\matM}{\mathbf{M}}
\newcommand{\matN}{\mathbf{N}}

\newcommand{\matQ}{\mathbf{Q}}
\newcommand{\matS}{\mathbf{S}}
\newcommand{\matX}{\mathbf{X}}
\newcommand{\matU}{\mathbf{U}}
\newcommand{\matV}{\mathbf{V}}
\newcommand{\matY}{\mathbf{Y}}

\newcommand{\tensA}{\mathcal{A}}

\newcommand{\tensC}{\mathcal{C}}

\newcommand{\tensG}{\mathcal{G}}

\newcommand{\tensM}{\mathcal{M}}
\newcommand{\tensN}{\mathcal{N}}

\newcommand{\tensT}{\mathcal{T}}

\newcommand{\tensV}{\mathcal{V}}
\newcommand{\tensW}{\mathcal{W}}
\newcommand{\tensY}{\mathcal{Y}}
\newcommand{\tensX}{\mathcal{X}}

\newcommand{\rank}{\mathrm{rank}}

\newcommand{\ranktc}{\mathrm{rank}_{\mathrm{tc}}}

\newcommand{\St}{\mathrm{St}}
\newcommand{\Span}{\mathrm{span}}
\newcommand{\grad}{\mathrm{grad}}

\newcommand{\subjectto}{\mathrm{s.\,t.}}
\newcommand{\T}{\mathsf{T}}

\newcommand{\frob}{\mathrm{F}}

\newcommand{\PGD}{$\mathrm{P^2GD}$}

\usepackage{tikz}
\usepackage{tikz-3dplot}
\usetikzlibrary{patterns}
\usetikzlibrary{shapes,calc,shapes,arrows}

\makeatletter
\@addtoreset{equation}{section}
\makeatother

\DeclareMathOperator{\diag}{diag}

\DeclareMathOperator{\ten}{ten}
\DeclareMathOperator{\tangent}{T}
\DeclareMathOperator{\normal}{N}
\DeclareMathOperator{\retr}{R}

\DeclareMathOperator{\proj}{P}
\DeclareMathOperator{\approj}{\tilde{P}}
\DeclareMathOperator*{\argmin}{arg\,min}
\DeclareMathOperator*{\argmax}{arg\,max}

\usepackage{array}
\usepackage{longtable} 
\usepackage{colortab} 
\usepackage{colortbl}
\usepackage{arydshln}

\begin{document}

\title{Low-rank optimization on Tucker tensor varieties\thanks{BG was supported by the Young Elite Scientist Sponsorship Program by CAST. YY was supported by the National Natural Science Foundation of China (grant No. 12288201).}
}

\titlerunning{Tucker tensor variety}        

\author{Bin Gao \and Renfeng Peng \and Ya-xiang Yuan}


\institute{Bin Gao \and Ya-xiang Yuan \at
    State Key Laboratory of Scientific and Engineering Computing, Academy of Mathematics and Systems Science, Chinese Academy of Sciences, Beijing, China \\
              \email{\{gaobin,yyx\}@lsec.cc.ac.cn; }
           \and
           Renfeng Peng \at
           State Key Laboratory of Scientific and Engineering Computing, Academy of Mathematics and Systems Science, Chinese Academy of Sciences, and University of Chinese Academy of Sciences, Beijing, China\\
           \email{pengrenfeng@lsec.cc.ac.cn} 
}

\date{Received: date / Accepted: date}

\maketitle

\begin{abstract}
    In the realm of tensor optimization, the low-rank Tucker decomposition is crucial for reducing the number of parameters and for saving storage. We explore the geometry of Tucker tensor varieties---the set of tensors with bounded Tucker rank---which is notably more intricate than the well-explored matrix varieties. We give an explicit parametrization of the tangent cone of Tucker tensor varieties and leverage its geometry to develop provable gradient-related line-search methods for optimization on Tucker tensor varieties. To the best of our knowledge, this is the first work concerning geometry and optimization on Tucker tensor varieties. In practice, low-rank tensor optimization suffers from the difficulty of choosing a reliable rank parameter. To this end, we incorporate the established geometry and propose a Tucker rank-adaptive method that aims to identify an appropriate rank with guaranteed convergence. Numerical experiments on tensor completion reveal that the proposed methods are in favor of recovering performance over other state-of-the-art methods. The rank-adaptive method performs the best across various rank parameter selections and is indeed able to find an appropriate rank. 
\end{abstract}

\section{Introduction}
Tensors are powerful tools for representing multi-dimensional data, yet they are often encumbered by high storage and computational costs. Adopting a low-rank assumption mitigates these challenges by extracting the most significant information from tensors, thereby substantially reducing storage requirements. Low-rank optimization has demonstrated its effectiveness across various applications, including image processing~\cite{vasilescu2003multilinear,6909773}, matrix and tensor completion~\cite{kressner2014low,kasai2016low,gao2023optimization,gao2024riemannian}, tensor equations~\cite{kressner2016preconditioned}, mathematical finance~\cite{glau2020low}, and high-dimensional partial differential equations~\cite{bachmayr2023dynamical,wang2023solving}; see~\cite{grasedyck2013literature,uschmajew2020geometric} for an overview. 

In contrast with the matrix rank, different tensor decomposition formats can lead to various definitions of the rank of a tensor. The canonical polyadic decomposition~\cite{hitchcock1928multiple}, Tucker decomposition~\cite{tucker1964extension}, tensor train decomposition~\cite{oseledets2011tensor}, and tensor ring decomposition~\cite{zhao2016tensor} are among the most typical decomposition formats; see~\cite{kolda2009tensor} for an overview. Specifically, Tucker decomposition, also referred to as higher-order principal component analysis~\cite{kapteyn1986approach} or multilinear singular value decomposition~\cite{de2000multilinear}, allows us to explore the low-rank structure along each mode of a tensor. Moreover, Tucker decomposition boils down to the ubiquitous singular value decomposition (SVD) in the matrix case. In this paper, we are concerned with the following low-rank Tucker tensor optimization problem in which the search space consists of tensors with bounded Tucker rank, i.e.,
\begin{equation}
    \begin{aligned}
        \min_{\tensX}\ &\ \ \ \ \ \ f(\tensX) \\
        \subjectto\ &\quad \tensX\in\tensM_{\leq\vecr}:=\{\tensX\in\mathbb{R}^{n_1\times n_2\times\cdots\times n_d}:\ranktc(\tensX)\leq\vecr\},
    \end{aligned}
    \label{eq: problem (P)}
\end{equation}
where $f:\mathbb{R}^{n_1\times n_2\times\cdots\times n_d}\to\mathbb{R}$ is a smooth function, $\vecr=(r_1,r_2,\dots,r_d)$ is an array of $d$ positive integers, and $\ranktc(\tensX)$ denotes the Tucker rank of $\tensX$. The set $\tensM_{\leq\vecr}$ can be constructed through the determinants of minors with size $(r_k+1)$ being zero from the mode-$k$ unfolding matrix of a tensor for $k=1,2,\dots,d$, which renders $\tensM_{\leq\vecr}$ a~real-algebraic variety. Hence, we refer to~$\tensM_{\leq\vecr}$ as the \emph{Tucker tensor varieties}. Note that there is another type of tensor varieties in tensor train format. Kutschan~\cite{kutschan2018tangent} provided the parametrization of the tangent cone for tensor train varieties. More recently, Vermeylen et~al.~\cite{vermeylen2023rank} adopted the geometric tools in tensor train varieties and proposed a rank-estimation method for third-order tensor train completion.

\paragraph{Related work and motivation}
Since low-rank Tucker tensor optimization is closely related to low-rank matrix optimization, we start with an overview of the existing research in the field of low-rank matrix optimization. Recall that $\mathbb{R}_{r}^{m\times n}:=\{\matx\in\mathbb{R}^{m\times n}:\rank(\matx)=r\}$ is a smooth manifold (see, e.g.,~\cite{helmke1995critical,bruns2006determinantal}), one can benefit from the geometric tools and design geometric methods for minimizing~$f$ over $\mathbb{R}_{r}^{m\times n}$. For instance, Shalit et~al.~\cite{shalit2012online} proposed an online-learning procedure on $\mathbb{R}_{r}^{m\times n}$ and applied the procedure to a multi-label image classification task. Vandereycken~\cite{vandereycken2013low} derived geometric tools on the manifold $\mathbb{R}_{r}^{m\times n}$ and proposed the Riemannian conjugate gradient method for low-rank matrix completion. Mishra et~al.~\cite{mishra2014fixed} studied the quotient geometry of product manifolds generated by fixed-rank matrix factorizations and proposed Riemannian methods for low-rank matrix optimization. Based on the fact that~$\mathbb{R}_{r}^{m\times n}$ is not closed, Schneider and Uschmajew~\cite{schneider2015convergence} considered minimizing $f$ over the closure of $\mathbb{R}_{r}^{m\times n}$, i.e., the matrix varieties $\mathbb{R}_{\leq r}^{m\times n}:=\{\matx\in\mathbb{R}^{m\times n}:\rank(\matx)\leq r\}$, and proposed the projected gradient descent method (\PGD) in which the $t$-th iterate is updated with stepsize $s^{(t)}$ by
\[\matx^{(t+1)}=\proj_{\leq r}\!\left(\matx^{(t)}+s^{(t)}\proj_{\tangent_{\matx^{(t)}}\!\mathbb{R}_{\leq r}^{m\times n}}(-\nabla f(\matx^{(t)}))\right)\] 
that involves two metric projections onto the varieties and the tangent cone. The convergence of \PGD\ was proved by means of the assumption on $f$ satisfying \L{}ojasiewicz inequality. Zhou et~al.~\cite{zhou2016riemannian} designed a Riemannian rank-adaptive method on~$\mathbb{R}_{\leq r}^{m\times n}$ where the convergence is guaranteed. Recently, Hosseini and Uschmajew~\cite{hosseini2019gradient} proposed a gradient sampling method for optimization on general real algebraic varieties. Gao and Absil~\cite{gao2022riemannian} employed the geometric illustration of tangent cone to develop a Riemannian rank-adaptive method for low-rank matrix completion, which also appears to be favorable in low-rank semidefinite programming~\cite{tang2023solving}. More recently, Olikier and Absil~\cite{olikier2023apocalypse} proposed a first-order algorithm by equipping \PGD\ with a number of rank decrease attempts and proved that every accumulation point is stationary. Furthermore, Olikier et~al.~\cite{olikier2023first} developed a framework for first-order optimization on general stratified sets of matrices. Levin et~al.~\cite{levin2023finding} proposed a formulation in which the feasible set $\mathbb{R}_{\leq r}^{m\times n}$ was parametrized by a (product) manifold $\tensM$ (e.g., $\tensM=\mathbb{R}^{m\times r}\times\mathbb{R}^{n\times r}$) along with a lift $\varphi:\tensM\to\mathbb{R}_{\leq r}^{m\times n}$ satisfying $\varphi(\tensM)=\mathbb{R}_{\leq r}^{m\times n}$. Subsequently, the low-rank matrix optimization problem was reformulated as minimizing $f\circ\varphi$ on the manifold~$\tensM$, which can be solved by Riemannian optimization methods (see, e.g.,~\cite{absil2009optimization,boumal2023intromanifolds} for an overview). Olikier et~al.~\cite{olikier2023gauss} proposed Gauss--Southwell type descent methods on matrix varieties.

\revise{The low-rank Tucker tensor optimization problems are much more challenging than the matrix case, due to the intricate geometry of Tucker tensors.} In the light of low-rank matrix problems, the low-rank Tucker tensor optimization has several different formulations. One type is minimizing $f$ on a smooth manifold $\tensM_{\vecr}$ of tensors with fixed Tucker rank, i.e., 
\begin{equation}
  \label{eq: fixed-rank problem}
  \begin{aligned}
      \min_\tensX\ &\ \ \ \ \ \ f(\tensX)\\ 
      \subjectto\ &\quad \tensX\in\tensM_{\vecr}:=\{\tensX\in\mathbb{R}^{n_1\times n_2\times\cdots\times n_d}:\ranktc(\tensX)=\vecr\}.
  \end{aligned}
\end{equation}
Uschmajew and Vandereycken~\cite{uschmajew2013geometry} showed that the set of tensors with fixed Tucker rank is a submanifold of $\mathbb{R}^{n_1\times\cdots\times n_d}$ and provided explicit characterizations of the tangent space. Kressner et~al.~\cite{kressner2014low} proposed the Riemannian conjugate gradient method for low-rank Tucker tensor completion. Kasai and Mishra~\cite{kasai2016low} considered the quotient geometry of the product manifold from Tucker decomposition and proposed the Riemannian conjugate gradient method for tensor completion under a preconditioned metric. Interested readers are referred to~\cite{uschmajew2020geometric} for an overview of geometric methods on fixed-rank matrix and tensor manifolds. Nevertheless, it is worth noting that the fixed-rank Tucker manifold $\tensM_{\vecr}$ is not a closed set in $\mathbb{R}^{n_1\times n_2\times\cdots\times n_d}$. As a consequence, the classical convergence results established in Riemannian optimization (e.g.,~\cite{boumal2019global}) do not hold for accumulation points located on the boundary~$\tensM_{\leq\vecr}\setminus\tensM_{\vecr}$.

Instead of solving the fixed-rank optimization problem~\eqref{eq: fixed-rank problem}, we consider minimizing $f$ on the closure of $\tensM_{\vecr}$, i.e., solving the optimization problem~\eqref{eq: problem (P)}. On the one hand, unlike the well-explored geometric properties of matrix varieties~\cite{schneider2015convergence} or fixed-rank Tucker tensor manifold~\cite{uschmajew2013geometry}, the geometric counterpart for Tucker tensor varieties (e.g., the tangent cone at rank-deficient points in $\tensM_{\leq\vecr}\setminus\tensM_\vecr$) is unknown. On the other hand, $\tensM_{\leq\vecr}$ can be constructed by the intersection of $d$ tensorized matrix varieties of all unfolding matrices along different modes. However, the relationship between the tangent cone of $\tensM_{\leq\vecr}$ and the tangent cones of matrix varieties along different modes is unclear. Therefore, the geometry of Tucker tensor varieties can not be easily generalized from matrix varieties. The existing work in geometry and optimization on Tucker tensor varieties is quite sparse. Luo and Qi~\cite{luo2023optimality} studied the optimality conditions of~\eqref{eq: problem (P)} by exploiting a subset of the normal cone while the formulation of tangent cone remains unclear. The unknown geometry of Tucker tensor varieties hampers one from designing optimization methods on Tucker tensor varieties. In summary, we are motivated to seek an explicit parametrization of the tangent cone and to employ the established results to design geometric methods for~\eqref{eq: problem (P)}. 

In addition, the question ``how to choose an appropriate rank parameter $\vecr$ in low-rank optimization'' is appealing in practice. We observe that the numerical performance of optimization methods in low-rank optimization can be sensitive to the choice of rank parameter $\vecr$; see, e.g.,~\cite{zhou2016riemannian,gao2022riemannian,dong2022new}. Opting for a  larger rank parameter $\vecr$ is able to enlarge the search space and potentially leads to a better optimum. However, if $\vecr$ is too large, it triggers more storage and computational costs. Moreover, the optimization methods may converge to rank-deficient points due to the non-closed nature of~$\tensM_{\vecr}$. In view of these obstacles, we are motivated to design rank-adaptive strategies, tailored for optimization on Tucker tensor varieties, that are able to find an appropriate rank parameter during iterations.

\paragraph{Contributions}
In this paper, we delve into the geometry of Tucker tensor varieties $\tensM_{\leq\vecr}$ and propose geometric and rank-adaptive methods for optimization on Tucker tensor varieties. Specifically, we first provide new equivalent reformulations for both the tangent cone of matrix varieties and the tangent space of Tucker tensor manifold. Then, an explicit characterization of the tangent cone of $\tensM_{\leq\vecr}$ is constructed, paving the way for developing optimization methods on~$\tensM_{\leq\vecr}$. In order to bypass the computationally intractable metric projection onto the tangent cone, we propose an approximate projection by choosing bases from $d$ orthogonal complements of the corresponding $d$ factor matrices of a Tucker tensor. To the best of our knowledge, this is the first work investigating the geometry of Tucker tensor varieties.

Taking advantage of the derived geometric properties, we propose the gradient-related approximate projection (GRAP) method in which the iterate is given by
\[\tensX^{(t+1)}=\proj_{\leq\vecr}^{\mathrm{HO}}\left(\tensX^{(t)}+s^{(t)}\approj_{\tangent_{\tensX^{(t)}}\!\tensM_{\leq\vecr}}(-\nabla f(\tensX^{(t)}))\right),\] 
where $\proj_{\leq\vecr}^{\mathrm{HO}}$ projects a tensor onto the Tucker tensor varieties $\tensM_{\leq\vecr}$ by the higher-order singular value decomposition and $\approj$ is the proposed approximate projection. The GRAP method can be viewed as a generalization of the Riemannian gradient method for optimization on $\tensM_\vecr$, while it is capable of dealing with rank-deficient points, and the convergence can still be guaranteed via \L{}ojasiewicz inequality. Additionally, we propose a new approximate projection operator $\hat{\proj}$ by leveraging partial information of the tangent cone. Surprisingly, the operator $\hat{\proj}$ is apt to develop a method without projection onto $\tensM_{\leq\vecr}$, namely, retraction-free, which iterates as 
\[\tensX^{(t+1)}=\tensX^{(t)}+s^{(t)}\hat{\proj}_{\tangent_{\tensX^{(t)}}\!\tensM_{\leq\vecr}}(-\nabla f(\tensX^{(t)}))\]
while preserving feasibility, i.e., $\tensX^{(t+1)}\in\tensM_{\leq\vecr}$, and thus saves computational costs. This method is called the retraction-free gradient-related approximate projection (rfGRAP) method and its convergence is also proved.

In order to identify an appropriate Tucker rank during iterations, we resort to the geometric illustration of the tangent cone and propose a provable Tucker rank-adaptive method (TRAM), which consists of line search on a fixed-rank manifold, rank-decreasing, and rank-increasing procedures. Specifically, a rank-decreasing procedure is aimed to save storage and eliminate singularity if an iterate is nearly rank-deficient. If an iterate appears to be nearly stationary on the fixed-rank manifold, we increase the rank in pursuit of a better optimum.

We compare the proposed GRAP, rfGRAP, and TRAM methods with existing methods in tensor completion on synthetic datasets, hyperspectral images, and movie ratings. The numerical results suggest that the proposed methods perform better than the state-of-the-art methods under different rank parameter selections. Specifically, the proposed TRAM method demonstrates its capability to find an appropriate rank in practice.

\paragraph{Organization}
We introduce the geometric tools of matrix varieties and notation for tensor operations in section~\ref{sec: Preliminaries}. In section~\ref{sec: geom of Tucker tensor varieties}, an explicit parametrization of the tangent cone of Tucker tensor varieties is provided along with an approximate projection onto the tangent cone. We propose the geometric methods, GRAP and rfGRAP, in sections~\ref{sec: GRAP} and~\ref{sec: rfGRAP}, and we design the Tucker rank-adaptive method in section~\ref{sec: TRAM}. Section~\ref{sec: experiments} reports the numerical performance of proposed methods in tensor completion. Finally, we draw the conclusion in section~\ref{sec: conclusion}.

\section{Low-rank manifolds and matrix varieties}\label{sec: Preliminaries}
In this section, the required geometry of matrix manifold and matrix varieties, which is closely related to the Tucker tensor varieties, is introduced first. Then, we describe the \revise{notation for tensor operations}, the definition of Tucker decomposition, and the geometry of the fixed-rank Tucker manifold.

Given a nonempty subset $C$ of $\mathbb{R}^{n_1\times n_2\times\cdots\times n_d}$, the (Bouligand) \emph{tangent cone} of $C$ at $\tensX\in C$ is defined by
\begin{equation*}
        \tangent_\tensX\!C:=\{\tensV\in\mathbb{R}^{n_1\times n_2\times\cdots\times n_d}:\exists t^{(i)}\to 0,\ \tensX^{(i)}\to \tensX\text{ in }C,\ \subjectto\ \frac{\tensX^{(i)}-\tensX}{t^{(i)}}\to\tensV\}.
\end{equation*}
The set $\normal_\tensX\! C=(\tangent_\tensX\! C)^\circ:=\{\tensN\in\mathbb{R}^{n_1\times n_2\times\cdots\times n_d}:\langle\tensN,\tensV\rangle\leq 0\ \text{for all}\ \tensV\in\tangent_\tensX\! C\}$ is called the \emph{normal cone} of $C$ at $\tensX$. Note that if $C$ is a manifold, the tangent cone $\tangent_\tensX\!C$ (normal cone $\normal_\tensX\!C$) is referred to as the tangent space (normal space) of $C$ at $\tensX$. A mapping $\retr:\bigcup_{\tensX\in C}\{\tensX\}\times\tangent_\tensX\!C\to C$ is called a \revise{\emph{retraction}~\cite[\S 3.1.2]{hosseini2019gradient} if for all $\tensX\in C$ and $\tensV\in\tangent_\tensX\!C$ it holds that $\lim_{t\to 0^+}(\retr_\tensX(t\tensV)-\tensX-t\tensV)/t=0$}.

\begin{definition}\label{def: stationary}
    A point $\tensX\in C$ is called \emph{stationary} for the optimization problem~\eqref{eq: problem (P)} if $\langle\nabla f(\tensX),\tensV\rangle\geq 0$ holds for all~$\tensV\in\tangent_\tensX\!C$, which is equivalent to $-\nabla f(\tensX)\in\normal_\tensX\!C$ or \revise{the projected anti-gradient satisfies} $\|\proj_{\tangent_\tensX\!C}(-\nabla f(\tensX))\|_\frob=0$. 
\end{definition}

\subsection{Low-rank matrix manifold and varieties}\label{subsec: low-rank matrix}
Let $m,n,r$ be positive integers satisfying $r\leq\min\{m,n\}$. Given a matrix $\matx\in\mathbb{R}^{m\times n}$, the image of $\matx$ and its orthogonal complement are defined by $\Span(\matx):=\{\matx\vecy:\vecy\in\mathbb{R}^n\}\subseteq\mathbb{R}^m$ and $\Span(\matx)^\perp:=\{\vecy\in\mathbb{R}^m:\langle\vecx,\vecy\rangle=0\ \text{for all}\ \vecx\in\Span(\matx)\}$ respectively. The fixed-rank matrix manifold and matrix varieties are denoted by $\mathbb{R}^{m\times n}_{r}$ and $\mathbb{R}^{m\times n}_{\leq r}$ respectively. The set $\St(r,n):=\{\matx\in\mathbb{R}^{n\times r}:\matx^\T\matx=\matI_r\}$ is the \emph{Stiefel manifold}, where $\matI_r$ is the $r\times r$ identity matrix. The orthogonal group is denoted by $\mathcal{O}(n):=\{\matQ\in\mathbb{R}^{n\times n}:\matQ^\T\matQ=\matQ\matQ^\T=\matI_n\}$. 

\paragraph{Geometry of matrix manifold and varieties} The tangent and normal cones play important roles in optimization on matrix varieties $\mathbb{R}^{m\times n}_{\leq r}$. Therefore, we introduce the explicit formulae of tangent and normal cones of matrix varieties (see~\cite[Proposition 2.1]{vandereycken2013low} and~\cite[Theorem 3.2]{schneider2015convergence}) as follows. 
\begin{proposition}\label{prop: geom of matrix varieties}
    Given $\matx\in\mathbb{R}_{\underline{r}}^{m\times n}$ with $\underline{r}\leq r$. A thin SVD of $\matx$ is $\matx=\matu\Sigma\matv^\T$, where $\matu\in\St(\underline{r},m)$, $\matv\in\St(\underline{r},n)$ and $\Sigma=\diag(\sigma_1,\sigma_2,\dots,\sigma_{\underline{r}})$ with $\sigma_1\geq\sigma_2\geq\cdots\geq\sigma_{\underline{r}}>0$. It holds that
    \begin{equation*}
        \begin{aligned}
            \tangent_\matx\!\mathbb{R}^{m\times n}_{\underline{r}}&=\left\{\begin{bmatrix}
                \matu & \matu^\perp
            \end{bmatrix}
            \begin{bmatrix}
                \mathbb{R}^{\underline{r}\times \underline{r}} & \mathbb{R}^{\underline{r}\times (n-\underline{r})}\\
                \mathbb{R}^{(m-\underline{r})\times \underline{r}} & 0
            \end{bmatrix}
            \begin{bmatrix}
                \matv & \matv^\perp
            \end{bmatrix}^\T\right\},\\
            \normal_\matx\!\mathbb{R}^{m\times n}_{\underline{r}}&=\left\{\begin{bmatrix}
                \matu & \matu^\perp
            \end{bmatrix}
            \begin{bmatrix}
                0 & 0\\
                0 & \mathbb{R}^{(m-\underline{r})\times (n-\underline{r})}
            \end{bmatrix}
            \begin{bmatrix}
                \matv & \matv^\perp
            \end{bmatrix}^\T\right\},\\
            \tangent_\matx\!\mathbb{R}^{m\times n}_{\leq r}&=\left\{\begin{bmatrix}
                \matu & \matu^\perp
            \end{bmatrix}
            \begin{bmatrix}
                \mathbb{R}^{\underline{r}\times \underline{r}} & \mathbb{R}^{\underline{r}\times (n-\underline{r})}\\
                \mathbb{R}^{(m-\underline{r})\times \underline{r}} & \mathbb{R}^{(m-\underline{r})\times (n-\underline{r})}_{\leq (r-\underline{r})}
            \end{bmatrix}
            \begin{bmatrix}
                \matv & \matv^\perp
            \end{bmatrix}^\T\right\},\\
            \normal_\matx\!\mathbb{R}^{m\times n}_{\leq r}&=
            \left\{
                \begin{array}{ll}
                    \normal_\matx\!\mathbb{R}^{m\times n}_{r},&\ \text{if}\ \underline{r}=r;\\
                    \{0\},&\ \text{if}\ \underline{r}<r,
                \end{array}
            \right.
        \end{aligned}
    \end{equation*}
    for any $\matu^\perp\in\St(m-\underline{r},m)$ with $\Span(\matu^\perp)=\Span(\matu)^\perp$ and $\matv^\perp\in\St(n-\underline{r},n)$ with $\Span(\matv^\perp)=\Span(\matv)^\perp$. 
\end{proposition}

Note that $(\matu^\perp)^\T\matu=0$ and $(\matv^\perp)^\T\matv=0$. The choice of $\matu^\perp$ and $\matv^\perp$ is not unique, but the results in Proposition~\ref{prop: geom of matrix varieties} are independent of a specific choice of $\matu^\perp$ and $\matv^\perp$. Actually, the tangent space and normal space can be uniquely represented in the sense of tensor product by
    \begin{equation*}
        \begin{aligned}
            \tangent_\matx\!\mathbb{R}^{m\times n}_{\underline{r}}&=\Span(\matu)\otimes\Span(\matv)+\Span(\matu)^\perp\otimes\Span(\matv)+\Span(\matu)\otimes\Span(\matv)^\perp,\\
            \normal_\matx\!\mathbb{R}^{m\times n}_{\underline{r}}&=\Span(\matu)^\perp\otimes\Span(\matv)^\perp.
        \end{aligned}
    \end{equation*}
    The direct sum of tangent and normal spaces forms the Euclidean space $\mathbb{R}^{m\times n}$, i.e., 
    \begin{equation*}
        \begin{aligned}
            \mathbb{R}^{m\times n}&=\tangent_\matx\!\mathbb{R}^{m\times n}_{\underline{r}}+\normal_\matx\!\mathbb{R}^{m\times n}_{\underline{r}}\\
            &=\left\{\begin{bmatrix}
                \matu & \matu^\perp
            \end{bmatrix}
            \vcenter{\hbox{\includegraphics[scale=1.2]{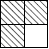}}}
            \begin{bmatrix}
                \matv & \matv^\perp
            \end{bmatrix}^\T\right\}+\left\{\begin{bmatrix}
                \matu & \matu^\perp
            \end{bmatrix}
            \vcenter{\hbox{\includegraphics[scale=1.2]{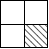}}}
            \begin{bmatrix}
                \matv & \matv^\perp
            \end{bmatrix}^\T\right\},
        \end{aligned}
    \end{equation*}
    where a shaded square represents an arbitrary matrix and the blank represents the matrix with zero elements.

\paragraph{\revise{New parametrization of the tangent cone}}
We propose a new reduced \revise{parametrization} of the tangent cone of matrix varieties. For any element $\Xi$ in the tangent cone $\tangent_\matx\!\mathbb{R}^{m\times n}_{\leq r}$, there exists $\matC\in\mathbb{R}^{\underline{r}\times \underline{r}}$, $\matd\in\mathbb{R}^{(m-\underline{r})\times \underline{r}}$, $\matE\in\mathbb{R}^{\underline{r}\times (n-\underline{r})}$ and $\matf\in\mathbb{R}^{(m-\underline{r})\times (n-\underline{r})}_{\leq (r-\underline{r})}$, such that
\begin{equation}
    \label{eq: vec in tangent cone}
    \Xi=\matu\matc\matv^\T+\matu^\perp\matd\matv^\T+\matu\mate(\matv^\perp)^\T+\matu^\perp\matf(\matv^\perp)^\T.
\end{equation}
\revise{Subsequently, we decompose $\Xi$ by using the decomposition $\matf=\tilde{\matu}\matS\tilde{\matv}^\T$, where $\tilde{\matu}\in\St(r-\underline{r},m-\underline{r})$, $\tilde{\matv}\in\St(r-\underline{r},n-\underline{r})$, and $\matS\in\mathbb{R}^{(r-\underline{r})\times (r-\underline{r})}$. Note that $\matS$ is not necessarily of full rank. Since $r-\underline{r}\leq\min\{m-\underline{r},n-\underline{r}\}$, there exist matrices $\tilde{\matu}_2\in\St(m-r,m-\underline{r})$ and $\tilde{\matv}_2\in\St(n-r,n-\underline{r})$,} such that $[\tilde{\matu}\ \tilde{\matu}_2]\in\mathcal{O}(m-\underline{r})$ and $[\tilde{\matv}\ \tilde{\matv}_2]\in\mathcal{O}(n-\underline{r})$. In fact, it holds that $\Span(\tilde{\matu}_2)=\Span(\tilde{\matu})^\perp$ and $\Span(\tilde{\matv}_2)=\Span(\tilde{\matv})^\perp$. As a result, we yield a new equivalent (reduced) parametrization of $\Xi$ as follows
\begin{align}
    \Xi&=\matu\matc\matv^\T+\matu^\perp\matd\matv^\T+\matu\matE(\matv^\perp)^\T+\matu^\perp\tilde{\matu}\matS\tilde{\matv}^\T(\matv^\perp)^\T\nonumber\\
    &=\matu\matc\matv^\T+\matu^\perp\matd\matv^\T+\matu\mate(\matv^\perp)^\T+\matu_1\matS\matv_1^\T\nonumber\\
    &=\matu\matc\matv^\T+\matu^\perp[\tilde{\matu}\ \tilde{\matu}_2][\tilde{\matu}\ \tilde{\matu}_2]^\T\matd\matv^\T+\matu\mate[\tilde{\matv}\ \tilde{\matv}_2][\tilde{\matv}\ \tilde{\matv}_2]^\T(\matv^\perp)^\T+\revise{\matu_1\matS\matv_1^\T}\nonumber\\
    &=\begin{bmatrix}
        \matu & \matu_1 & \matu_2
    \end{bmatrix}
    \left[
        \begin{array}{cc:c}
            \matc & ~\mate\tilde{\matv} & ~\mate\tilde{\matv}_2\\
            \tilde{\matu}^\T\matd & ~\revise{\matS} & 0\\
            \cdashline{1-2}[4pt/2pt]
            \multicolumn{3}{c}{}\\[-3mm]
            \tilde{\matu}_2^\T\matd & \multicolumn{2}{l}{\ \!~~0\ \ \ \ \ ~\!0}
        \end{array}
    \right]
    \begin{bmatrix}
        \matv & \matv_1 & \matv_2
    \end{bmatrix}^\T,\nonumber
\end{align}
\revise{where $\matu_1:=\matu^\perp\tilde{\matu}$, $\matu_2:=\matu^\perp\tilde{\matu}_2$, $\matv_1:=\matv^\perp\tilde{\matv}$, and $\matv_2:=\matv^\perp\tilde{\matv}_2$. Figure~\ref{fig: new representation of matrix tangent cone} provides an equivalent illustration of an element in $\tangent_\matx\!\mathbb{R}^{m\times n}_{\leq r}$, where a dashed square represents an arbitrary matrix in $\mathbb{R}^{r\times r}$. The illustration is vital to the development of the tangent cone of Tucker tensor varieties; see Theorem~\ref{thm: tangent cone of Tucker} for details.

\begin{figure}[htbp]
    \centering
    \begin{tikzpicture}[scale=1.2]
        \node at (0,0) {\includegraphics[scale=1.4]{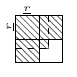}};
        \node at (-1.4,0) {$
                \begin{bmatrix}
                    \matu & \matu_1 & \matu_2
                \end{bmatrix}$};
        \node at (1.4,0.03) {$
                \begin{bmatrix}
                    \matv & \matv_1 & \matv_2
                \end{bmatrix}^\T$};
    \end{tikzpicture}
    \caption{Illustration of an element in $\tangent_{\matx}\!\mathbb{R}^{m\times n}_{\leq r}$ with parameters $\matu_1\in\St(r-\underline{r},m),\matv_1\in\St(r-\underline{r},n),\matu_2\in\St(m-r,m),\matv_2\in\St(n-r,n)$ satisfying $[\matu\ \matu_1\ \matu_2]\in\mathcal{O}(m)$ and $[\matv\ \matv_1\ \matv_2]\in\mathcal{O}(n)$}
    \label{fig: new representation of matrix tangent cone}
\end{figure}
}

It is worth noting that we can further develop a compact parametrization in the way of full rank decomposition for $\matf$ in~\eqref{eq: vec in tangent cone}. Specifically, we consider the decomposition $\matf=\hat{\matu}\hat{\matS}\hat{\matv}^\T$, where $\hat{\matu}\in\St(\ell,m-\underline{r})$, $\hat{\matv}\in\St(\ell,n-\underline{r})$, $\hat{\matS}\in\mathbb{R}_\ell^{\ell\times \ell}$ is a full rank matrix, and $\ell=\rank(\matf)$. Then, we obtain the compact parametrization
\begin{align}
    \Xi&=\matu\matc\matv^\T+\matu^\perp\matd\matv^\T+\matu\mate(\matv^\perp)^\T+\hat{\matu}_1\hat{\matS}\hat{\matv}_1^\T\label{eq: new representation of matrix 01},
\end{align}
where $\hat{\matu}_1:=\matu^\perp\hat{\matu}$ and $\hat{\matv}_1:=\matv^\perp\hat{\matv}$. Since $\hat{\matS}$ is of full rank, the representation~\eqref{eq: new representation of matrix 01} is unique in the sense of right orthogonal group actions on $\hat{\matu}_1$ and $\hat{\matv}_1$. In fact, we observe that $\proj_{\matu}^\perp\Xi\proj_{\matv}^\perp=\hat{\matu}_1\hat{\matS}\hat{\matv}_1^\T$,
where $\proj_\matu:=\matu\matu^\T$, $\proj_\matu^\perp:=\matI_{m}-\proj_\matu$, $\proj_\matv:=\matv\matv^\T$, and $\proj_\matv^\perp:=\matI_{n}-\proj_\matv$. Therefore, it follows that
\[\Span(\hat{\matu}_1)=\Span(\proj_{\matu}^\perp\Xi\proj_{\matv}^\perp)\quad\text{and}\quad\Span(\hat{\matv}_1)=\Span((\proj_{\matu}^\perp\Xi\proj_{\matv}^\perp)^\T),\]
which implies that the spaces are unique.

\begin{remark}\label{rem: 1}
    We observe that the tangent cone can be alternatively decomposed as 
    \[\tangent_\matx\!\mathbb{R}^{m\times n}_{\leq r}=\tangent_\matx\!\mathbb{R}^{m\times n}_{\underline{r}}+\normal_{\leq (r-\underline{r})}(\matx),\]
    where $\normal_{\leq (r-\underline{r})}(\matx):=\left\{\matN\in\normal_\matx\!\mathbb{R}^{m\times n}_{\underline{r}}: \rank(\matN)\leq(r-\underline{r})\right\}$. As suggested in~\cite{gao2022riemannian}, given a matrix $\matx\in\mathbb{R}^{m\times n}_{\underline{r}}$ with $\underline{r}<r$ and a vector in the cone $\matV\in \normal_{\leq (r-\underline{r})}(\matx)\setminus\{0\}$, it holds that \[\rank(\matx+s\matV)\in(\underline{r},r]\]
    for $s>0$. The principle can be applied to increase the rank of $\matx$. A similar observation can be found in tensor cases; see section~\ref{subsec: rank increase} for details.
\end{remark}

\paragraph{Metric projections}
Given a matrix $\mata=\sum_{k=1}^I\sigma_k^{}\vecu_k^{}\vecv_k^\T\in\mathbb{R}^{m\times n}_I$, where $\sigma_1\geq\cdots\geq\sigma_I>0$, $\vecu_k\in\mathbb{R}^m$ and $\vecv_k\in\mathbb{R}^n$ are singular vectors of $\mata$. The \emph{metric projection} of $\mata$ onto the matrix varieties $\mathbb{R}^{m\times n}_{\leq r}$ is defined by
\[\proj_{\leq r}(\mata):=\argmin_{\matx\in\mathbb{R}^{m\times n}_{\leq r}}\|\mata-\matx\|_\frob^2.\]
By using the Eckart--Young theorem, the metric projection exists and \revise{is given by}
\[\proj_{\leq r}(\mata)=
\left\{
    \begin{array}{ll}
        \mata,&\ \text{if}\ I\leq r,\\
    \sum_{k=1}^r\sigma_k^{}\vecu_k^{}\vecv_k^\T,&\ \text{if}\ I>r.
    \end{array}
\right.
\]
Note that $\proj_{\leq r}(\mata)$ is not unique when the singular value $\sigma_{r}$ is equal to $\sigma_{r+1}$, but we can always choose $\sum_{k=1}^r\sigma_k^{}\vecu_k^{}\vecv_k^\T$ in practice. \revise{When $\rank(\mata)\geq r$}, the metric projection $\proj_r(\mata)$ onto the fixed-rank manifold $\mathbb{R}^{m\times n}_{r}$ \revise{equals to $\proj_{\leq r}(\mata)$. However, when $\rank(\mata)<r$, the matrix $\mata$ can be approximated arbitrarily closely by a rank-$r$ matrix and thus $\proj_r(\mata)$ does not exist.}

Moreover, in view of Proposition~\ref{prop: geom of matrix varieties}, given a matrix $\mata\in\mathbb{R}^{m\times n}$, the orthogonal projections onto the tangent space and tangent cone are given by
    \begin{equation}\label{eq: projection onto tangents matrix}
        \begin{aligned}
            \proj_{\tangent_\matx\!\mathbb{R}^{m\times n}_{\underline{r}}}\!\mata&=\proj_\matu\!\mata\!\proj_\matv+\proj_\matu^\perp\!\mata\!\proj_\matv+\proj_\matu\!\mata\!\proj_\matv^\perp,\\
            \proj_{\tangent_\matx\!\mathbb{R}^{m\times n}_{\leq r}}\!\mata&=\proj_{\tangent_\matx\!\mathbb{R}^{m\times n}_{\underline{r}}}\!\mata+\proj_{\leq (r-\underline{r})}\!\left(\proj_\matu^\perp\!\mata\!\proj_\matv^\perp\right).
        \end{aligned}
    \end{equation}
    \revise{In practice, $\proj_\matu^\perp\!\mata\!\proj_\matv^\perp$ can be efficiently computed by $\proj_\matu^\perp\!\mata\!\proj_\matv^\perp=\mata-\matu(\matu^\T\mata)-(\mata\matv)\matv^\T+\matu(\matu^\T\mata\matv)\matv^\T$.}

\subsection{Tucker decomposition: definition and geometry}
We introduce \revise{notation for tensor operations}. Denote the index set $\{1,2,\dots,n\}$ by~$[n]$. The inner product between two tensors $\tensX,\tensY\in\mathbb{R}^{n_1\times n_2\times\cdots\times n_d}$ is defined by $\langle\tensX,\tensY\rangle := \sum_{i_1=1}^{n_1} \cdots \sum_{i_d=1}^{n_d} \tensX({i_1,\dots,i_d})\tensY({i_1,\dots,i_d})$. The Frobenius norm of a tensor $\tensX$ is defined by $\|\tensX\|_\mathrm{F}:=\sqrt{\langle\tensX,\tensX\rangle}$. The mode-$k$ unfolding of a tensor $\tensX \in \mathbb{R}^{n_1 \times\cdots\times n_d}$ is denoted by a matrix $\matx_{(k)}\in\mathbb{R}^{n_k\times n_{-k}} $ for $k=1,\dots,d$, where $n_{-k}:=\prod_{i\neq k}n_i$. The $ (i_1,i_2,\dots,i_d)$-th entry of $\tensX$ corresponds to the $(i_k,j)$-th entry of $\matx_{(k)}$, where
$ j = 1 + \sum_{\ell \neq k, \ell = 1}^d(i_\ell-1)J_\ell$ with $J_\ell = \prod_{m = 1,m \neq k}^{\ell-1} n_m$. The tensorization operator maps a matrix $\matx_k\in\mathbb{R}^{n_k\times n_{-k}}$ to a tensor $\ten_{(k)}(\matx_k)\in\mathbb{R}^{n_1\times\cdots\times n_d}$ defined by $\ten_{(k)}(\matx_k)(i_1,\dots,i_d)=\matx_k(i_k,1 + \sum_{\ell \neq k, \ell = 1}^d(i_\ell-1)J_\ell)$ for $(i_1,\dots,i_d)\in[n_1]\times\cdots\times[n_d]$. Note that $\ten_{(k)}(\matx_{(k)})=\tensX$ holds for fixed $n_1,\dots,n_d$. Therefore, the tensorization operator is invertible. The $k$-mode product of a tensor $\tensX$ and a matrix $\mata\in\mathbb{R}^{n_k\times M}$ is denoted by $\tensX\times_k\mata\in\mathbb{R}^{n_1\times\cdots\times M\times\cdots\times n_d}$, where the $ (i_1,\dots,i_{k-1},j,i_{k+1},\dots,i_d)$-th entry of $\tensX\times_k\mata$ is $\sum_{i_k=1}^{n_k}x_{i_1\dots i_d}a_{ji_k}$. It holds that $(\tensX\times_k\mata)_{(k)}=\mata\matx_{(k)}$. Given $\vecu_1\in\mathbb{R}^{n_1}\setminus\{0\},\dots,\vecu_d\in\mathbb{R}^{n_d}\setminus\{0\}$, a rank-$1$ tensor of size $n_1\times\cdots\times n_d$ is defined by the outer product $\tensV:=\vecu_1\circ\cdots\circ\vecu_d$, or $v_{i_1,\dots,i_d}:=u_{1,i_1}\cdots u_{d,i_d}$ equivalently. The Kronecker product of two matrices $\mata\in\mathbb{R}^{m_1\times n_1}$ and $\matb\in\mathbb{R}^{m_2\times n_2}$ is an $(m_1m_2)$-by-$(n_1n_2)$ matrix defined by $\mata\otimes\matb:=(a_{ij}\matb)_{ij}$. The vector $\vece_i\in\mathbb{R}^{n}$ is defined by the $i$-th column of $n$-by-$n$ identity matrix $\matI_n$. \revise{Given two vectors $\vecx,\vecy\in\mathbb{R}^d$, we denote $\vecx\leq\vecy$ ($\vecx<\vecy$) if $x_i\leq y_i$ ($x_i<y_i$) for all $i\in[d]$.}

\begin{definition}[Tucker decomposition]
    Given a tensor $\tensX \in \mathbb{R}^{n_1 \times n_2\times\cdots\times n_d}$, 
    the Tucker decomposition is \[\tensX =\tensG\times_1\matu_1\times_2\matu_2\cdots\times_d\matu_d=\revise{\tensG\times_{k=1}^d\matu_k},\]
    where $\tensG\in\mathbb{R}^{r_1 \times r_2\times\cdots\times r_d}$ is a core tensor, $\matu_k\in\St(r_k,n_k)$ are factor matrices with orthogonal columns and $r_k=\rank(\matx_{(k)})$. 
\end{definition}

The Tucker rank of a tensor $\tensX$ is defined by $$\ranktc(\tensX):=\vecr=(r_1,r_2,\dots,r_d)=(\rank(\matx_{(1)}),\rank(\matx_{(2)}),\dots,\rank(\matx_{(d)})).$$ 
Figure~\ref{fig: 3D Tucker} depicts the Tucker decomposition of a third-order tensor. Note that the mode-$k$ unfolding of a tensor $\tensX =\tensG\times_1\matu_1\times_2\matu_2\cdots\times_d\matu_d$ satisfies 
\[\matx_{(k)}=\matu_k\matG_{(k)}\left(\matu_d\otimes\cdots\otimes\matu_{k+1}\otimes\matu_{k-1}\otimes\cdots\otimes\matu_{1}\right)^\T=\matu_k\matG_{(k)}((\matu_j)^{\otimes j\neq k})^\T,\]
where $(\matu_j)^{\otimes j\neq k}:=\matu_d\otimes\cdots\otimes\matu_{k+1}\otimes\matu_{k-1}\otimes\cdots\otimes\matu_{1}$ for $k\in[d]$. Notably, for a $d$-th order tensor $\tensA$, it holds that 
\begin{equation}
    \label{eq: tensor product}
    \tensA\in\bigotimes_{k=1}^d\Span(\matu_k)\iff\tensA=\tensC\times_{k=1}^d\matu_k
\end{equation}
with $\tensC\in\mathbb{R}^{r_1\times r_2\times\cdots\times r_d}$.
\begin{figure}[htbp]
    \centering
    \includegraphics[scale=0.833]{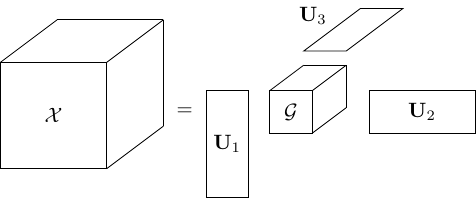}
    \caption{Tucker decomposition of a third-order tensor}
    \label{fig: 3D Tucker}
\end{figure}

\paragraph{Fixed-rank Tucker manifold}
Since $\tensM_\vecr=\{\tensX\in\mathbb{R}^{n_1\times n_2\times\cdots\times n_d}:\ranktc(\tensX)=\vecr\}$ forms a smooth manifold with dimension $\dim(\tensM_\vecr)=r_1r_2\cdots r_d+\sum_{k=1}^{d}(n_kr_k-r_k^2)$, Koch and Lubich~\cite{koch2010dynamical} provided that the tangent space of $\tensM_\vecr$ at $\tensX$ is characterized by
\begin{equation}
    \label{eq: old tangent space}
      \tangent_\tensX\!\tensM_\vecr=
      \left\{
      \begin{array}{l}
           \dot{\tensG}\times_1\matu_1\cdots\times_d\matu_d+\sum_{k=1}^d\tensG\times_k\dot{\matu}_k\times_{j\neq k}\matu_j:\vspace{4pt}\\
           \dot{\tensG}\in\mathbb{R}^{r_1\times r_2\times\cdots\times r_d},\dot{\matu}_k\in\mathbb{R}^{n_k\times r_k},\dot{\matu}_k^\T\matu_k^{}=0
      \end{array}\right\},
\end{equation}
\revise{where $\tensG\times_k\dot{\matu}_k\times_{j\neq k}\matu_j=\tensG\times_1\matu_1\cdots\times_{k-1}\matu_{k-1}\times_k\dot{\matu}_k\times_{k+1}\matu_{k+1}\cdots\times_d\matu_d$.}

Though the Tucker decomposition of a tensor is not unique~\revise{\cite[\S 4.3]{kolda2009tensor}}, the \revise{parametrization} of $\tangent_\tensX\!\tensM_\vecr$ does not \revise{depend} on a specific Tucker decomposition. \revise{Specifically, consider another Tucker decomposition of $\tensX=\breve{\tensG}\times_{k=1}^d\breve{\matu}_k$ and the associated tangent space $\breve{\tangent}_\tensX\!\tensM_\vecr$ via~\eqref{eq: old tangent space} with parameters $\breve{\tensG}\in\mathbb{R}^{r_1\times r_2\times\cdots\times r_d}$ and $\breve{\matu}_k\in\St(r_k,n_k)$ for $k\in[d]$. It suffices to show that $\tangent_\tensX\!\tensM_\vecr=\breve{\tangent}_\tensX\!\tensM_\vecr$. In view of the fact that $\Span(\matu_k)=\Span(\matx_{(k)})=\Span(\breve{\matu}_k)$, there exists $\matQ_k\in{\cal O}(r_k)$ such that $\breve{\matu}_k=\matu_k\matQ_k$ and $\breve{\tensG}=\tensG\times_{i=1}^d\matQ_i^\T$ for $k\in[d]$. For all $\tensV\in\tangent_\tensX\!\tensM_\vecr$, it holds that
\begin{equation*}
    \begin{aligned}
        \tensV&= \dot{\tensG}\times_{k=1}^d\matu_k+\sum_{k=1}^d\tensG\times_k\dot{\matu}_k\times_{j\neq k}\matu_j\\
        &= (\dot{\tensG}\times_{k=1}^d\matQ_k^\T)\times_{k=1}^d(\matu_k\matQ_k)+\sum_{k=1}^d(\tensG\times_{i=1}^d\matQ_i^\T)\times_k(\dot{\matu}_k\matQ_k)\times_{j\neq k}(\matu_j\matQ_j)\\
        &=(\dot{\tensG}\times_{k=1}^d\matQ_k^\T)\times_{k=1}^d\breve{\matu}_k+\sum_{k=1}^d\breve{\tensG}\times_k(\dot{\matu}_k\matQ_k)\times_{j\neq k}\breve{\matu}_j.
    \end{aligned}
\end{equation*}
Since $\dot{\tensG}\times_{k=1}^d\matQ_k^\T\in\mathbb{R}^{r_1\times r_2\times\cdots\times r_d}$, $\dot{\matu}_k\matQ_k\in\mathbb{R}^{n_k\times r_k}$, and 
\[(\dot{\matu}_k\matQ_k)^\T\breve{\matu}_k^{}=\matQ_k^\T\dot{\matu}_k^\T\matu_k^{}\matQ_k^{}=0,\]
we conclude that $\tensV\in\breve{\tangent}_\tensX\!\tensM_\vecr$ as in~\eqref{eq: old tangent space}. Therefore, $\tangent_\tensX\!\tensM_\vecr\subseteq \breve{\tangent}_\tensX\!\tensM_\vecr$ and the converse is also true. 
}

\paragraph{Metric projections}
Given a tensor $\tensA\in\mathbb{R}^{n_1\times n_2\times\cdots\times n_d}$, the metric projection of $\tensA$ onto the Tucker tensor varieties $\tensM_{\leq\vecr}=\{\tensX\in\mathbb{R}^{n_1\times n_2\times\cdots\times n_d}:\ranktc(\tensX)\leq\vecr\}$ is defined by
\begin{equation}
    \label{eq: metric projection onto M}
    \proj_{\leq\vecr}(\tensA):=\argmin_{\tensX\in\tensM_{\leq\vecr}}\|\tensA-\tensX\|_\frob^2.
\end{equation}
In contrast with the matrix case in section~\ref{subsec: low-rank matrix}, $\proj_{\leq\vecr}(\tensA)$ does not have a closed-form expression in general~\cite{de2000multilinear}. Nevertheless, one can apply \emph{higher-order singular value decomposition} (HOSVD) to yield a quasi-optimal solution. Specifically, the HOSVD procedure sequentially applies the best rank-$r_k$ approximation operator $\proj^k_{\leq r_k}$ to each mode of $\tensA$ for $k=1,2,\dots,d$, i.e,
\begin{equation}
    \label{eq: HOSVD}
    \proj_{\leq\vecr}^\mathrm{HO}(\tensA):=\proj_{\leq r_d}^d(\proj_{\leq r_{d-1}}^{d-1}\cdots(\proj_{\leq r_1}^1(\tensA))),
\end{equation}
\revise{where $\proj_{\leq r_k}^k(\tensA):=\ten_{(k)}(\bar{\matu}_k^{}\bar{\matu}_k^\T\mata_{(k)})$, $\bar{\matu}_k$ is the leading $r_k$ singular vectors of~$\mata_{(k)}$, and $\proj_{\leq\vecr}^\mathrm{HO}$ does not depend on the order of $\{\proj_{\leq r_k}^k\}_{k=1}^d$~\cite[\S 3]{vannieuwenhoven2012new}.} Since the quasi-optimality
\begin{equation}
    \label{eq: quasi-optimal HOSVD}
    \|\tensA-\proj_{\leq\vecr}^\mathrm{HO}(\tensA)\|_\frob\leq\sqrt{d}\|\tensA-\proj_{\leq\vecr}(\tensA)\|_\frob
\end{equation}
\revise{holds~\cite[Lemma 2.6]{grasedyck2010hierarchical},} HOSVD can be served as an approximate projection onto $\tensM_{\leq\vecr}$. Moreover, we can prove that HOSVD is also a retraction on $\tensM_{\leq\vecr}$ around $\tensX$. 
\begin{proposition}\label{prop: retraction}
    The mapping $\retr_\tensX^\mathrm{HO}:\tangent_{\tensX}\!\tensM_{\leq\vecr}\to\tensM_{\leq\vecr}:\tensV\mapsto\proj_{\leq\vecr}^\mathrm{HO}(\tensX+\tensV)$ is a retraction on $\tensM_{\leq\vecr}$.
\end{proposition}
\begin{proof}
    See Appendix~\ref{app: retraction}. \qed
\end{proof}

In addition, the metric projection of a tensor $\tensA\in\mathbb{R}^{n_1\times n_2\times\cdots\times n_d}$ onto $\tangent_{\tensX}\!\tensM_\vecr$ is given by~\cite[\S 2.3]{koch2010dynamical}
\begin{equation}\label{eq: orth proj onto tangent space Tucker}
    \proj_{\tangent_{\tensX}\!\tensM_\vecr}\!\tensA=\tensA\times_{k=1}^d\proj_{\matu_k}+\sum_{k=1}^d \tensG\times_k\left(\proj_{\matu_k}^\perp\!\left(\tensA\times_{j\neq k}\matu_j^\T\right)_{(k)}\matG_{(k)}^\dagger\right)\times_{j\neq k}\matu_j,
\end{equation}
where $\matG_{(k)}^\dagger=\matG_{(k)}^\T(\matG_{(k)}^{}\matG_{(k)}^\T)^{-1}$ is the Moore--Penrose pseudoinverse of $\matG_{(k)}$, the mode-$k$ unfolding matrix of $\tensG$.

\section{Geometry of Tucker tensor varieties}\label{sec: geom of Tucker tensor varieties}
We first revisit the tangent space of the fixed-rank Tucker manifold from a new geometric perspective. Then, an explicit \revise{parametrization} of the tangent cone of Tucker varieties is developed. In the end, we propose an approximate projection onto the tangent cone.
\subsection{A new formulation of the tangent space of fixed-rank Tucker manifold}
In this subsection, we give an equivalent formulation for the tangent space to fixed-rank Tucker manifold $\tensM_{\underline{\vecr}}$. Given a tensor $\tensX\in\mathbb{R}^{n_1\times n_2\times \cdots\times n_d}$ with $\ranktc(\tensX)=\underline{\vecr}$ and Tucker decomposition $\tensX=\tensG\times_{k=1}^d\matu_k$, it follows from~\eqref{eq: old tangent space} that $\Span(\dot{\matu}_k)\subseteq\Span(\matu_k)^\perp$, i.e., the set $\{\dot{\matu}_k\in\mathbb{R}^{n_k\times \underline{r}_k}: \dot{\matu}_k^\T\matu_k=0\}$ can be presented as $\{\dot{\matu}_k=\matu_k^\perp\dot{\matR}_k^{}: \dot{\matR}_k^{}\in\mathbb{R}^{(n_k-\underline{r}_k)\times \underline{r}_k}\}$, where $\matu_k^\perp$ is defined in a same fashion as Proposition~\ref{prop: geom of matrix varieties} for $k\in[d]$. Subsequently, for any tangent vector $\tensV\in\tangent_\tensX\!\tensM_{\underline{\vecr}}$, we have
\begin{equation*}
    \begin{aligned}
        \tensV&=\dot{\tensG}\times_1\matu_1\cdots\times_d\matu_d+\sum_{k=1}^d\tensG\times_k\dot{\matu}_k\times_{j\neq k}\matu_j\\
        &=\dot{\tensG}\times_1\matu_1\cdots\times_d\matu_d+\sum_{k=1}^d\tensG\times_k(\matu_k^\perp\dot{\matR}_k^{})\times_{j\neq k}\matu_j\\
        &=\dot{\tensG}\times_1\matu_1\cdots\times_d\matu_d+\sum_{k=1}^d(\tensG\times_k\dot{\matR}_k^{})\times_k\matu_k^\perp\times_{j\neq k}\matu_j.
    \end{aligned}
\end{equation*}
\revise{Therefore, the tangent space $\tangent_\tensX\!\tensM_{\underline{\vecr}}$ can also be parametrized by}
\begin{equation}
    \label{eq: new tangent space}
        \tangent_\tensX\!\tensM_{\underline{\vecr}}=\left\{\begin{array}{l}
             \dot{\tensG}\times_1\matu_1\cdots\times_d\matu_d+\sum_{k=1}^d(\tensG\times_k\dot{\matR}_k)\times_k\matu_k^\perp\times_{j\neq k}\matu_j:\vspace{4pt}\\
             \dot{\tensG}\in\mathbb{R}^{\underline{r}_1\times\underline{r}_2\times\cdots\times \underline{r}_d},\dot{\matR}_k\in\mathbb{R}^{(n_k-\underline{r}_k)\times \underline{r}_k}
        \end{array}
        \right\}.
\end{equation}
Specifically, a tangent vector in $\tangent_\tensX\!\tensM_{\underline{\vecr}}$ for $d=3$ is illustrated in Fig.~\ref{fig: core tensor in tangent space}, where a shaded cube represents an arbitrary tensor $\dot{\tensG}\in\mathbb{R}^{\underline{r}_1\times\underline{r}_2\times\cdots\times \underline{r}_d}$ and the blank represents the tensor with zero elements. For the sake of brevity, we adopt these symbols to represent a tensor. 

\begin{figure}[htbp]
    \centering
    \includegraphics[scale=0.833]{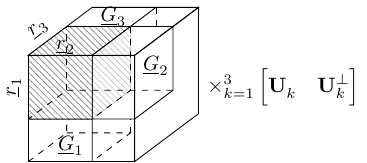}
    \caption{Illustration of a tangent vector in $\tangent_\tensX\!\tensM_{{\underline{\vecr}}}$ at $\tensX=\tensG\times_{k=1}^d\matu_k$ for $d=3$. $\underline{G}_k:=\tensG\times_k\dot{\matR}_k$ with arbitrary $\dot{\matR}_k\in\mathbb{R}^{(n_k-\underline{r}_k)\times \underline{r}_k}$}
    \label{fig: core tensor in tangent space}
\end{figure}

\subsection{Tangent cone of Tucker tensor varieties}
In this subsection, we study the tangent cone of $\tensM_{\leq\vecr}$. Since $\tensM_{\leq\vecr}$ can be constructed through $d$ matrix varieties $\mathbb{R}^{n_k\times n_{-k}}_{\leq r_k}$ of unfolding matrices, i.e., 
\[\tensM_{\leq\vecr}=\bigcap_{k=1}^d\ten_{(k)}\!\left(\mathbb{R}^{n_k\times n_{-k}}_{\leq r_k}\right),\] 
\revise{it is straightforward to have the following lemma.} 

\begin{lemma}\label{lem: tensor tangent cone to matrix tangent cone}
    Given a tensor $\tensX\in\tensM_{\leq\vecr}$, the tangent cone of $\tensM_{\leq\vecr}$ is a subset of the intersection of tensorized tangent cones of unfolding matrices along different modes, i.e., 
    \[\tangent_\tensX\!\tensM_{\leq\vecr}\subseteq\bigcap_{k=1}^d\ten_{(k)}\!\left(\tangent_{\matx_{(k)}}\!\mathbb{R}^{n_k\times n_{-k}}_{\leq r_k}\right).\]
\end{lemma}

Then, we give an explicit parametrization of the tangent cone of Tucker tensor varieties as follows.
\begin{theorem}\label{thm: tangent cone of Tucker}
    Given a Tucker tensor $\tensX=\tensG\times_1\matu_1\cdots\times_d\matu_d\in\mathbb{R}^{n_1\times n_2\times\cdots\times n_d}$ with $\ranktc(\tensX)=\underline{\vecr}\leq\vecr$, any $\tensV$ in the tangent cone of $\tensM_{\leq\vecr}$ at $\tensX$ can be expressed by 
    \begin{equation}
        \label{eq: Tucker tangent cone}
        \begin{aligned}
            \tensV=\tensC\times_{k=1}^d\begin{bmatrix}
                \matu_k & \matu_{k,1}
            \end{bmatrix}+\sum_{k=1}^d\tensG\times_k(\matu_{k,2}\matR_{k,2})\times_{j\neq k}\matu_j,
        \end{aligned}
    \end{equation}
    where $\tensC\in\mathbb{R}^{r_1\times r_2\times\cdots\times r_d}$, $\matR_{k,2}\in\mathbb{R}^{(n_k-r_k)\times \underline{r}_k}$, $\matu_{k,1}\in\St(r_k-\underline{r}_k,n_k)$ and $\matu_{k,2}\in\St(n_k-r_k,n_k)$ are arbitrary that satisfy $[\matu_k\ \matu_{k,1}\ \matu_{k,2}]\in\mathcal{O}(n_k)$ for $k\in[d]$. 
    
\end{theorem}
\begin{proof}
    For all $\tensV=\tensC\times_{k=1}^d\begin{bmatrix}
        \matu_k & \matu_{k,1}
    \end{bmatrix} +\sum_{k=1}^d\tensG\times_k(\matu_{k,2}\matR_{k,2})\times_{j\neq k}\matu_j$, $t^{(i)}=1/i$ for $i\in\mathbb{N}$, we consider a sequence
    \begin{equation*}
        \begin{aligned}
            \tensX^{(i)}&=t_i\tensC\times_{k=1}^d\begin{bmatrix}
        \matu_k+t_i\matu_{k,2}\matR_{k,2} &\ \matu_{k,1}
    \end{bmatrix}+\tensG\times_{k=1}^d(\matu_k+t_i\matu_{k,2}\matR_{k,2})\\
    &\in\bigotimes_{k=1}^d \Span\left(\begin{bmatrix}
        \matu_k+t_i\matu_{k,2}\matR_{k,2} &\ \matu_{k,1}
    \end{bmatrix}\right)\subseteq\tensM_{\leq\vecr},
        \end{aligned}
    \end{equation*}
    where we use the facts~\eqref{eq: tensor product} and $\rank([\matu_k+t_i\matu_{k,2}\matR_{k,2}\ \matu_{k,1}])\leq r_k$. By direct calculations, we yield that $\lim_{i\to\infty
    }(\tensX^{(i)}-\tensX)/t^{(i)}=\tensV$ and thus $\tensV\in\tangent_\tensX\!\tensM_{\leq\vecr}$.

    Then, we prove that any $\tensV\in\tangent_\tensX\!\tensM_{\leq\vecr}$ can be represented by~\eqref{eq: Tucker tangent cone}. Denote the mode-$k$ unfolding matrix of $\tensV$ by $\Xi_k:=\matv_{(k)}$. It follows from Lemma~\ref{lem: tensor tangent cone to matrix tangent cone} that $\Xi_k\in\tangent_{\matx_{(k)}}\!\mathbb{R}^{n_k\times n_{-k}}_{\leq r_k}$ for $k\in[d]$. Since the $k$-th unfolding matrix $\matx_{(k)}$ of $\tensX$ admits the decomposition 
    \[\matx_{(k)}=\matu_k\matG_{(k)}\matv_k^\T=\matu_k^{}\tilde{\matu}_k^{}\tilde{\Sigma}_{k}^{}\tilde{\matv}_k^\T\matv_k^\T,\] 
    where $\matG_{(k)}=\tilde{\matu}_k^{}\tilde{\Sigma}_{k}^{}\tilde{\matv}_k^\T$ is the SVD of the unfolding matrix $\matG_{(k)}$ of $\tensG$ with $\tilde{\matu}_k^{}\in\mathcal{O}(\underline{r}_k)$, and $\matv_{k}:=(\matu_j)^{\otimes j\neq k}$, it follows from $\Xi_k\in\tangent_{\matx_{(k)}}\!\mathbb{R}^{n_k\times n_{-k}}_{\leq r_k}$ and \revise{Fig.~\ref{fig: new representation of matrix tangent cone}} that there exists $\matc_k\in\mathbb{R}^{\underline{r}_k\times\underline{r}_k}$, $\matd_{k,1}\in\mathbb{R}^{(r_k-\underline{r}_k)\times\underline{r}_k}$, $\matd_{k,2}\in\mathbb{R}^{(n_k-r_k)\times\underline{r}_k}$, $\mate_{k,1}\in\mathbb{R}^{\underline{r}_k\times(r_{k}-\underline{r}_k)}$,  $\mate_{k,2}\in\mathbb{R}^{\underline{r}_k\times(n_{-k}-{r}_k)}$, $\revise{\matS_k}\in\mathbb{R}^{(r_{k}-\underline{r}_k)\times(r_{k}-\underline{r}_k)}$, $\matu_{k,1}\in\St(r_k-\underline{r}_k,n_k)$, $\matu_{k,2}\in\St(n_k-r_k,n_k)$, $\matv_{k,1}\in\St(r_k-\underline{r}_k,n_{-k})$, $\matv_{k,2}\in\St(n_{-k}-r_k,n_{-k})$ that satisfy $[\matu_k\tilde{\matu}_k \ \matu_{k,1} \ \matu_{k,2}]\in\mathcal{O}(n_k)$ and $[\matv_k\tilde{\matv}_k \ \matv_{k,1} \ \matv_{k,2}]\in\mathcal{O}(n_{-k})$, such that 
    \[\Xi_k=\begin{bmatrix}
        \matu_k\tilde{\matu}_k &\ \matu_{k,1} &\ \matu_{k,2}
    \end{bmatrix}
    \begin{bmatrix}
        \matc_k & \mate_{k,1} & \mate_{k,2}\\
        \matd_{k,1} & \revise{\matS_k} & 0\\
        \matd_{k,2} & 0 & 0 
    \end{bmatrix}
    \begin{bmatrix}
        \matv_k\tilde{\matv}_k &\ \matv_{k,1} &\ \matv_{k,2}
    \end{bmatrix}^\T.\]

    We aim to find the unknowns $\tensC$ and $\matR_{k,2}$ in~\eqref{eq: Tucker tangent cone} by leveraging the structure of $\Xi_k$. To this end, \revise{we first validate the claim} that 
    \[\tensW:=\tensV-\sum_{k=1}^d\tensG\times_k(\matu_{k,2}\matR_{k,2})\times_{j\neq k}\matu_j\in\bigotimes_{k=1}^d\Span([\matu_k\ \matu_{k,1}])\]
    with $\matR_{k,2}=\matd_{k,2}^{}\tilde{\Sigma}_k^{-1}\tilde{\matu}_k^\T$. 
    \revise{In fact, we observe that} 
    \begin{equation*}
        \begin{aligned}
            \proj_{\matu_{k,2}}\!\matW_{(k)}&=\proj_{\matu_{k,2}}\!\left(\tensV-\sum_{i=1}^d\tensG\times_i(\matu_{i,2}\matR_{i,2})\times_{j\neq i}\matu_j\right)_{(k)}\\
            &=\proj_{\matu_{k,2}}\!\Xi_k-\matu_{k,2}^{}\matR_{k,2}\matG_{(k)}^{}\matv_k^\T\\
            &=\matu_{k,2}^{}\matd_{k,2}^{}\tilde{\matv}_{k}^{\T}\matv_{k}^{\T}-\matu_{k,2}^{}\matR_{k,2}\matG_{(k)}^{}\matv_k^\T\\
            &=\matu_{k,2}^{}\matd_{k,2}^{}\tilde{\Sigma}_k^{-1}\tilde{\matu}_k^\T\tilde{\matu}_k^{}\tilde{\Sigma}_k^{}\tilde{\matv}_{k}^\T\matv_k^\T-\matu_{k,2}^{}\matR_{k,2}\matG_{(k)}^{}\matv_k^\T\\
            &=\matu_{k,2}^{}\matd_{k,2}^{}\tilde{\Sigma}_k^{-1}\tilde{\matu}_k^\T\matG_{(k)}^{}\matv_k^\T-\matu_{k,2}^{}\matR_{k,2}\matG_{(k)}^{}\matv_k^\T\\
            &=\matu_{k,2}^{}\left(\matd_{k,2}^{}\tilde{\Sigma}_k^{-1}\tilde{\matu}_k^\T-\matR_{k,2}\right)\matG_{(k)}^{}\matv_k^\T\\
            &=0
        \end{aligned}
    \end{equation*}
    holds for all $k\in[d]$. The equalities come from $\matv_{(k)}=\Xi_k$, $\proj_{\matu_{k,2}}\!\matu_k=0$, $\matv_k=(\matu_j)^{\otimes j\neq k}$, $\matG_{(k)}=\tilde{\matu}_k^{}\tilde{\Sigma}_{k}^{}\tilde{\matv}_k^\T$, and $\matR_{k,2}=\matd_{k,2}^{}\tilde{\Sigma}_k^{-1}\tilde{\matu}_k^\T$. We obtain that
    \[\matW_{(k)}\in\Span(\matu_{k,2})^\perp=\Span([\matu_k\tilde{\matu}_k\ \ \matu_{k,1}])=\Span([\matu_k\ \matu_{k,1}])\]
    and thus $\tensW\in\bigotimes_{k=1}^d\Span([\matu_k\ \matu_{k,1}])$. 
    
    Consequently, it follows from~\eqref{eq: tensor product} and the claim that there exists a tensor $\tensC\in\mathbb{R}^{r_1\times r_2\times\cdots\times r_d}$ 
    such that 
    \[\tensC\times_{k=1}^d\begin{bmatrix}
        \matu_k & \matu_{k,1}
    \end{bmatrix}=\tensW=\tensV-\sum_{k=1}^d\tensG\times_k(\matu_{k,2}\matR_{k,2})\times_{j\neq k}\matu_j.\]
    Hence, $\tensV$ can be interpreted by~\eqref{eq: Tucker tangent cone}.
    \qed
\end{proof}

Remarkably, in the proof of Theorem~\ref{thm: tangent cone of Tucker}, we employ the low-rank structure in Fig.~\ref{fig: new representation of matrix tangent cone} of unfolding matrices in the tangent cone of Tucker tensor varieties. In other words, the new reformulation in Fig.~\ref{fig: new representation of matrix tangent cone} of the tangent cone of matrix varieties is crucial to develop~\eqref{eq: Tucker tangent cone} in tensor case. We give a geometric illustration of the tangent cone $\tangent_\tensX\!\tensM_{\leq\vecr}$ for $d=3$ in Fig.~\ref{fig: core tensor in tangent cone}, where the shaded cube represents an arbitrary tensor $\tensC\in\mathbb{R}^{r_1\times\cdots\times r_d}$. Note that~\eqref{eq: Tucker tangent cone} boils down to the tangent space~\eqref{eq: new tangent space} for $\underline{\vecr}=\vecr$ in the sense of $\matu_{k,2}=\matu_k^\perp$, $\matR_{k,2}=\matR_k$, $\tensC=\dot{\tensG}$, and removing $\matu_{k,1}$. Moreover, the results are reduced to the tangent space and tangent cone of the matrix case for $d=2$. In fact, the tangent cone of matrix varieties in Fig.~\ref{fig: new representation of matrix tangent cone} can be informally interpreted by compressing the cube from back to front in Fig.~\ref{fig: core tensor in tangent cone} akin to ``playing the accordion''.

\begin{figure}[htbp]
    \centering
    \includegraphics[scale=0.833]{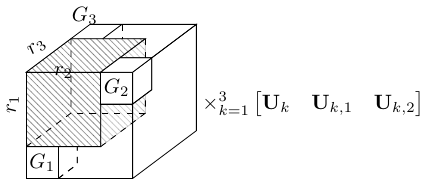}
    \caption{Illustration of an element in $\tangent_\tensX\!\tensM_{\leq\vecr}$ at $\tensX=\tensG\times_{k=1}^d\matu_k$ for $d=3$. $G_k:=\tensG\times_k\matR_{k,2}$ with \revise{parameters $\matR_{k,2}\in\mathbb{R}^{(n_k-r_k)\times\underline{r}_k}$, $\matu_{k,1}\in\St(r_k-\underline{r}_k,n_k)$ and $\matu_{k,2}\in\St(n_k-r_k,n_k)$ satisfying $[\matu_k\ \matu_{k,1}\ \matu_{k,2}]\in\mathcal{O}(n_k)$ for $k\in[d]$}}
    \label{fig: core tensor in tangent cone}
\end{figure}

\paragraph{Uniqueness of the representation}
The representation of the tangent cone vectors in Theorem~\ref{thm: tangent cone of Tucker} can be non-unique. Specifically, we observe that 
\begin{equation*}
    \begin{aligned}
        \tensV&=\breve{\tensC}\times_{k=1}^d[\matu_k \ \breve{\matu}_{k,1}]+\sum_{k=1}^d\tensG\times_k(\breve{\matu}_{k,2}\breve{\matR}_{k,2})\times_{j\neq k}\matu_j
    \end{aligned}
\end{equation*}
is a different parametrization of $\tensV\in\tangent_\tensX\!\tensM_{\leq\vecr}$ for any $\matQ_{k,1}\in{\cal O}(r_k-\underline{r}_k)$ and $\matQ_{k,2}\in{\cal O}(n_k-r_k)$, where 
\[\breve{\tensC}=\tensC\times_{k=1}^d[\matI_{\underline{r}_k} \ \matQ_{k,1}^\T],\ \breve{\matu}_{k,1}=\matu_{k,1}\matQ_{k,1},\ \breve{\matu}_{k,2}=\matu_{k,2}\matQ_{k,2},\ \breve{\matR}_{k,2}=\matQ_{k,2}^\T\matR_{k,2}^{}\]
with $k\in[d]$. Therefore, the representation~\eqref{eq: Tucker tangent cone} is non-unique. Similar to the matrix case, the representation~\eqref{eq: Tucker tangent cone} is unique in the sense of the right orthogonal group actions on $\matu_{k,1}$ and $\matu_{k,2}$ with $k\in[d]$ if 
\[\ranktc(\tensV\times_{k=1}^d\proj_{\matu_k}^\perp)=\ranktc(\tensC\times_{k=1}^d[0\ \matu_{k,1}])=\vecr-\underline{\vecr}.\]
Note that the above restriction boils down to $\rank(\proj_\matu^\perp\!\Xi\proj_\matv^\perp)=r-\underline{r}$ for matrix case~\eqref{eq: new representation of matrix 01}. Moreover, we can provide a compact parametrization of an element in $\tangent_\tensX\!\tensM_{\leq\vecr}$; see Appendix~\ref{app: compact Tucker tangent cone} for details. Though the representation~\eqref{eq: Tucker tangent cone} is not a compact form since we do not impose the rank restriction on $\tensC$, \eqref{eq: Tucker tangent cone} is able to facilitate the computation of projections in the following sections.

\revise{We} observe from the proof of Theorem~\ref{thm: tangent cone of Tucker} that the tangent cone inclusion in Lemma~\ref{lem: tensor tangent cone to matrix tangent cone} is actually an equality, which is claimed in the following corollary. Precisely, for any tensor $\tensV\in\mathbb{R}^{n_1\times n_2\times\cdots\times n_d}$ that satisfies $\matv_{(k)}\in\tangent_{\matx_{(k)}}\!\mathbb{R}^{n_k\times n_{-k}}_{\leq r_k}$, it holds that $\tensV$ can be represented by~\eqref{eq: Tucker tangent cone}, i.e., $\tensV\in\tangent_\tensX\!\tensM_{\leq\vecr}$. 

\begin{corollary}\label{coro: 1}
    The tangent cone of Tucker tensor varieties equals the intersection of tensorized tangent cones of unfolding matrices along different modes, i.e., 
    \[\tangent_\tensX\!\tensM_{\leq\vecr}=\bigcap_{k=1}^d\ten_{(k)}\!\left(\tangent_{\matx_{(k)}}\!\mathbb{R}_{\leq r_k}^{n_k\times n_{-k}}\right).\]
\end{corollary}

    In general, the optimality condition (Definition~\ref{def: stationary}) of~\eqref{eq: problem (P)} is hard to verify in practice (see \S\ref{subsec: metric projection onto tangent cone}). The following proposition gives an equivalent way to verify those stationary points $\tensX^*\in\tensM_{<\vecr}:=\{\tensX\in\mathbb{R}^{n_1\times n_2\times\cdots\times n_d}:\ranktc(\tensX)<\vecr\}$ where all the modes of $\tensX^*$ are rank-deficient.  
    
\begin{proposition}\label{prop: stationary}
    Let $\tensX^*$ be a stationary point of~\eqref{eq: problem (P)} with $\underline{\vecr}^*:=\ranktc(\tensX^*)<\vecr$. Then, it holds that \[\nabla f(\tensX^*)=0.\]
\end{proposition}
\begin{proof}
    Since any tensor $\tensA\in\mathbb{R}^{n_1\times n_2\times\cdots\times n_d}$ can be represented by 
    \[\tensA=\sum_{i_1=1}^{n_1}\cdots\sum_{i_d=1}^{n_d}y_{i_1,\dots,i_d}\vece_{i_1}\circ\cdots\circ \vece_{i_d}\] 
    with rank-$1$ tensors $\vece_{i_1}\circ\cdots\circ \vece_{i_d}$ and $(i_1,\dots,i_d)\in[n_1]\times\cdots\times[n_d]$, it suffices to prove that $\langle\vece_{i_1}\circ\cdots\circ \vece_{i_d},\nabla f(\tensX^*)\rangle=0$. 

    \revise{
        It follows from $\underline{\vecr}^*<\vecr$ that $\tensX^*+t\vece_{i_1}\circ\cdots\circ \vece_{i_d}\in\tensM_{\leq\vecr}$ holds for all $t$ and $i_k\in[n_k]$ with $k\in[d]$. Therefore, it yields $\pm\vece_{i_1}\circ\cdots\circ \vece_{i_d}\in\tangent_{\tensX^*}\!\tensM_{\leq\vecr}$. 
    According to Definition~\ref{def: stationary} and the stationarity of $\tensX^*$, we have
    $\langle\vece_{i_1}\circ\cdots\circ \vece_{i_d},\nabla f(\tensX^*)\rangle=0$ and $\nabla f(\tensX^*)=0$.} \qed
\end{proof}

In contrast with the matrix case where $\mathbb{R}^{m\times n}_{\leq r}=\mathbb{R}^{m\times n}_{< r}\cup\mathbb{R}^{m\times n}_{r}$, the Tucker tensor varieties $\tensM_{\leq\vecr}$ consist of not only points in $\tensM_{<\vecr}$ and $\tensM_{\vecr}$, but also points where some but not all of the modes are rank-deficient. Note that Proposition~\ref{prop: stationary} is restricted on the stationary points in $\tensM_{<\vecr}$.

The explicit form of the tangent cone in Theorem~\ref{thm: tangent cone of Tucker} allows us to obtain several attractive results. Recall the parametrization~\eqref{eq: Tucker tangent cone} that
\begin{equation}
    \label{eq: definition of Vk}
    \tensV=\tensV_0+\sum_{k=1}^d\tensV_k:=\tensC\times_{k=1}^d\begin{bmatrix}
    \matu_k & \matu_{k,1}
\end{bmatrix}+\sum_{k=1}^d\tensG\times_k(\matu_{k,2}\matR_{k,2})\times_{j\neq k}\matu_j.
\end{equation}
Note that $\langle\tensV_i,\tensV_j\rangle=0$ for $i\neq j$ and thus $\|\tensV\|_\frob^2=\sum_{k=0}^d\|\tensV_k\|_\frob^2$. Surprisingly, searching along the two types of directions in $\tensV$ does not leave the Tucker tensor varieties: 1) it follows from~\eqref{eq: tensor product} and Tucker decomposition $\tensX=\tensG\times_{k=1}^d\matu_k$ that 
    \begin{align}
        \tensX+\tensV_0&=\tensG\times_{k=1}^d\matu_k+\tensC\times_{k=1}^d\begin{bmatrix}
            \matu_k & \matu_{k,1}
        \end{bmatrix}\nonumber\\
        &\in\bigotimes_{k=1}^d\Span([\matu_k\ \matu_{k,1}])\subseteq\tensM_{\leq\vecr};\label{eq: retraction-free directions}
    \end{align}
2) for all $k\in[d]$, we observe from $\rank(\matu_k+\matu_{k,2}\matR_{k,2})\leq \underline{r}_k$ that
    \begin{align}
        \tensX+\tensV_k&=\tensG\times_{i=1}^d\matu_i+\tensG\times_k(\matu_{k,2}\matR_{k,2})\times_{j\neq k}\matu_j\nonumber\\
        &=\tensG\times_k(\matu_k+\matu_{k,2}\matR_{k,2})\times_{j\neq k}\matu_j\nonumber\\
        &\in\tensM_{\leq\vecr}.\label{eq: retraction-free directions k}
    \end{align}
It is worth noting that~\eqref{eq: retraction-free directions} and~\eqref{eq: retraction-free directions k} provide $(d+1)$ \emph{retraction-free} search directions \revise{$\tensV_0,\tensV_1,\dots,\tensV_d$} in~$\tangent_\tensX\!\tensM_{\leq\vecr}$, which can be adopted to develop line-search methods on $\tensM_{\leq\vecr}$ without retractions; see section~\ref{sec: rfGRAP} for details. Additionally, the following corollary gives a bound for the retraction $\retr^{\mathrm{HO}}_\tensX$ and is of great importance to prove the convergence for line-search methods on $\tensM_{\leq\vecr}$.

\begin{corollary}\label{coro: bound of retraction}
    The HOSVD retraction satisfies that
    \[\|\tensX-\proj_{\leq\vecr}^{\mathrm{HO}}(\tensX+\tensV)\|_\frob\leq (1+\frac{d}{\sqrt{d+1}})\|\tensV\|_\frob\]
    for all $\tensX\in\tensM_{\leq\vecr}$ and $\tensV\in\tangent_\tensX\!\tensM_{\leq\vecr}$.
\end{corollary}
\begin{proof}
    \revise{It follows from~\eqref{eq: metric projection onto M} and $\tensX+\tensV_k\in\tensM_{\leq\vecr}$ by~\eqref{eq: retraction-free directions}--\eqref{eq: retraction-free directions k} that 
    \[\|\tensX+\tensV-\proj_{\leq\vecr}(\tensX+\tensV)\|^2_\frob=\min_{\tensY\in\tensM_{\leq\vecr}}\|\tensX+\tensV-\tensY\|_\frob^2\leq\|\tensX+\tensV-(\tensX+\tensV_k)\|^2_\frob\]
    for all $k=0,1,\dots,d$, and thus 
    \[(d+1)\|\tensX+\tensV-\proj_{\leq\vecr}(\tensX+\tensV)\|^2_\frob\leq\sum_{k=0}^d\|\tensX+\tensV-(\tensX+\tensV_k)\|^2_\frob.\] 
    It follows from the quasi-optimality~\eqref{eq: quasi-optimal HOSVD} and $\|\tensV\|_\frob^2=\sum_{k=0}^d\|\tensV_k\|_\frob^2$ that
    \begin{equation*}
        \begin{aligned}
            \left\|\tensX+\tensV-\proj^{\mathrm{HO}}_{\leq\vecr}(\tensX+\tensV)\right\|_\frob^2&\leq d\,\|\tensX+\tensV-\proj_{\leq\vecr}(\tensX+\tensV)\|_\frob^2\\
            &\leq \frac{d}{d+1}\sum_{k=0}^d\|\tensX+\tensV-(\tensX+\tensV_k)\|_\frob^2\\
            &=\frac{d}{d+1}\sum_{k=0}^d\sum_{j=0,j\neq k}^d\|\tensV_j\|_\frob^2\\
            &=\frac{d^2}{d+1}\|\tensV\|_\frob^2.
        \end{aligned}
    \end{equation*}}
    Therefore, it holds that
    \[\|\tensX-\proj_{\leq\vecr}^{\mathrm{HO}}(\tensX+\tensV)\|_\frob\leq\|\tensX+\tensV-\proj^{\mathrm{HO}}_{\leq\vecr}(\tensX+\tensV)\|_\frob+\|\tensV\|_\frob\leq (1+\frac{d}{\sqrt{d+1}})\|\tensV\|_\frob.\]
    \qed
\end{proof}

\subsection{Metric projection onto the tangent cone}\label{subsec: metric projection onto tangent cone}
For the sake of fulfilling the basic requirement of projected gradient methods on $\tensM_{\leq\vecr}$, we consider the metric projection of a tensor $\tensA\in\mathbb{R}^{n_1\times n_2\times\cdots\times n_d}$ onto the tangent cone $\tangent_\tensX\!\tensM_{\leq\vecr}$ at $\tensX=\tensG\times_{k=1}^d\matu_k$, which is determined by the following optimization problem
\begin{equation}\label{eq: orth projection on tangent cone}
    \proj_{\tangent_\tensX\!\tensM_{\leq\vecr}}\!\tensA:=\argmin_{\tensV\in\tangent_\tensX\!\tensM_{\leq\vecr}}\|\tensA-\tensV\|_\frob^2.
\end{equation}

Since the tangent cone $\tangent_\tensX\!\tensM_{\leq\vecr}$ is closed, the metric projection exists. Note that the metric projection can be non-unique, but we can always choose a specific one to facilitate projected gradient methods on $\tensM_{\leq\vecr}$. We explore the metric projection by taking the parametrization~\eqref{eq: Tucker tangent cone} and~\eqref{eq: definition of Vk} into~\eqref{eq: orth projection on tangent cone}, and it yields 
\begin{align}
        \left\|\tensA-\tensV\right\|_\frob^2=&\,\|\tensA-\sum_{k=0}^d\tensV_k\|_\frob^2\nonumber\\
            =&\,\|\tensA\|_\frob^2 + \left\|\tensV_0\right\|_\frob^2+\sum_{k=1}^d\left\|\tensV_k\right\|_\frob^2-2\left\langle\tensA,\tensV_0\right\rangle-2\sum_{k=1}^d\left\langle\tensA,\tensV_k\right\rangle\nonumber\\
            =&\,\|\tensA\|_\frob^2+\left\|\tensC\right\|_\frob^2-2\left\langle\tensA\times_{k=1}^d\begin{bmatrix}
                \matu_k & \matu_{k,1}
            \end{bmatrix}^\T,\tensC\right\rangle\nonumber\\
            &+\sum_{k=1}^d\left(\left\|\tensG\times_k\matR_{k,2}\right\|_\frob^2-2\left\langle\tensA\times_k\matu_{k,2}^\T\times_{j\neq k}\matu_j^\T,\tensG\times_k\matR_{k,2}\right\rangle\right).\label{eq: projection step 1}
\end{align}
In order to solve~\eqref{eq: orth projection on tangent cone}, we aim to find the unknowns $\tensC$, $\matu_{k,1}$, $\matu_{k,2}$ and $\matR_{k,2}$ for $k\in[d]$ in~\eqref{eq: projection step 1}. Generally speaking, the computation of $\proj_{\tangent_\tensX\!\tensM_{\leq\vecr}}\!\tensA$ is divided into two steps: 1) by fixing $\matu_{k,1}$ and $\matu_{k,2}$, we yield closed-form solutions of $\tensC$ and $\matR_{k,2}$; 2) we take the solutions into~\eqref{eq: projection step 1} and determine $\matu_{k,1}$ by an optimization problem, \revise{then $\matu_{k,2}$ is naturally obtained from $[\matu_{k}\ \matu_{k,1}\ \matu_{k,2}]\in\mathcal{O}(n_k)$.}

\paragraph{Step 1: representing $\tensC$ and $\matR_{k,2}$ by $\matu_{k,1}$ and $\matu_{k,2}$}
By fixing $\matu_{k,1}$ and $\matu_{k,2}$ for $k\in[d]$, we observe that~\eqref{eq: projection step 1} is a separable quadradic function with respect to $\tensC$ and $\matR_{k,2}$. Therefore, $\tensC$ and $\matR_{k,2}$ in the metric projection~\eqref{eq: orth projection on tangent cone} can be uniquely expressed by
\begin{equation}\label{eq: C and Rk2}
    \begin{aligned}
        \tensC&=\tensA\times_{k=1}^d\begin{bmatrix}
                \matu_k & \matu_{k,1}
            \end{bmatrix}^\T,\\
        \matR_{k,2}&=\left(\tensA\times_k^{}\matu_{k,2}^\T\times_{j\neq k}\matu_j^\T\right)_{(k)}\matG_{(k)}^\dagger.
    \end{aligned}
\end{equation}

\paragraph{Step 2: determining $\matu_{k,1}$}
We determine $\matu_{k,1}$ by taking~\eqref{eq: C and Rk2} into~\eqref{eq: projection step 1} and obtain 
\begin{equation*}
    \begin{aligned}
        &\left\|\tensA-\tensV\right\|_\frob^2\\
        =&\left\|\tensA\right\|_\frob^2-\left\|\tensC\right\|_\frob^2-\sum_{k=1}^d\left\|\tensG\times_k\matR_{k,2}\right\|_\frob^2\\
        =&\left\|\tensA\right\|_\frob^2-\left\|\tensC\right\|_\frob^2-\sum_{k=1}^d\left\|\left(\tensA\times_k\matu_{k,2}^\T\times_{j\neq k}\matu_j^\T\right)_{(k)}\matG_{(k)}^\dagger\matG_{(k)}^{}\right\|_\frob^2\\
        =&\left\|\tensA\right\|_\frob^2-\left\|\tensC\right\|_\frob^2-\sum_{k=1}^d\left\langle\mata_{\neq k}\proj_{\matG_{(k)}^\T},\matu_{k,2}^{}\matu_{k,2}^\T\mata_{\neq k}\proj_{\matG_{(k)}^\T}\right\rangle\\
        =&\left\|\tensA\right\|_\frob^2-\left\|\tensC\right\|_\frob^2-\sum_{k=1}^d\left\langle\mata_{\neq k}\proj_{\matG_{(k)}^\T},(\matI_{n_k}-\matu_k^{}\matu_k^\T-\matu_{k,1}^{}\matu_{k,1}^\T)\mata_{\neq k}\proj_{\matG_{(k)}^\T}\right\rangle\\
        =&\left\|\tensA\right\|_\frob^2-\left\|\tensA\times_{k=1}^d\begin{bmatrix}
                \matu_k & \matu_{k,1}
            \end{bmatrix}^\T\right\|_\frob^2\\
            &-\sum_{k=1}^d\left\|\mata_{\neq k}\proj_{\matG_{(k)}^\T}\right\|_\frob^2+\sum_{k=1}^d\left\|\matu_{k}^\T\mata_{\neq k}\proj_{\matG_{(k)}^\T}\right\|_\frob^2+\sum_{k=1}^d\left\|\matu_{k,1}^\T\mata_{\neq k}\proj_{\matG_{(k)}^\T}\right\|_\frob^2,
    \end{aligned}
\end{equation*}
where $\mata_{\neq k}:=(\tensA\times_{j\neq k}\matu_j^\T)_{(k)}\in\mathbb{R}^{n_k\times r_{-k}}$ and $\proj_{\matG_{(k)}^\T}:=\matG_{(k)}^\dagger\matG_{(k)}^{}$. Note that ${\matu_{k,2}}$ is eliminated by $\matu_k$ and $\matu_{k,1}$. Alternatively, one can also eliminate $\matu_{k,1}$ by $\matu_k$ and $\matu_{k,2}$. Nevertheless, the number of parameters of $\matu_{k,1}\in\mathbb{R}^{n_k\times(r_k-\underline{r}_k)}$ is smaller than $\matu_{k,2}\in\mathbb{R}^{n_k\times(n_k-r_k)}$ when $r_k\ll n_k$. Therefore, it is computationally favorable to formulate~\eqref{eq: orth projection on tangent cone} as
\begin{equation}
    \label{eq: reformulation of orth proj}
    \begin{aligned}
        \min_{\matu_{1,1},\matu_{2,1},\dots,\matu_{d,1}}\ &\ 
        -\left\|\tensA\times_{k=1}^d\begin{bmatrix}
            \matu_k & \matu_{k,1}
        \end{bmatrix}^\T\right\|_\frob^2+\sum_{k=1}^d\left\|\matu_{k,1}^\T\mata_{\neq k}\proj_{\matG_{(k)}^\T}\right\|_\frob^2\\
            \subjectto\ \ \quad\quad&\quad\ \begin{bmatrix}
                \matu_k & \matu_{k,1}
            \end{bmatrix}^\T\begin{bmatrix}
                \matu_k & \matu_{k,1}
            \end{bmatrix}=\matI_{r_k}\ \text{for}\ k\in[d].
    \end{aligned}
\end{equation}
Since the feasible set of~\eqref{eq: reformulation of orth proj} is compact, a global minimizer of~\eqref{eq: reformulation of orth proj} exists. 

Let $(\matu_{1,1}^*,\dots,\matu_{d,1}^*)$ be a global minimizer of~\eqref{eq: reformulation of orth proj}. By using~\eqref{eq: Tucker tangent cone} and~\eqref{eq: C and Rk2}, the projection onto the tangent cone $\tangent_\tensX\!\tensM_{\leq\vecr}$ is given by
\begin{equation}
    \label{eq: orth proj onto tangent cone by fixing Uk Uk1}\proj_{\tangent_\tensX\!\tensM_{\leq\vecr}}\!\tensA=\tensA\times_{k=1}^d\proj_{\matS_{k}^*}+\sum_{k=1}^d\tensG\times_k\left(\proj^\perp_{\matS_{k}^*}\!\left(\tensA\times_{j\neq k}^{}\matu_j^\T\right)_{(k)}\matG_{(k)}^\dagger\right)\times_{j\neq k}\matu_j,
\end{equation}
where $\matS_{k}^*:=[\matu_k\ \matu_{k,1}^*]$. We observe that the metric projection~\eqref{eq: orth proj onto tangent cone by fixing Uk Uk1} relies on the projection $\proj_{\matS_{k}^*}$ but not a specific $\matu_{k,1}^*$. Additionally,~\eqref{eq: orth proj onto tangent cone by fixing Uk Uk1} boils down to the projection onto the tangent space~\eqref{eq: orth proj onto tangent space Tucker} when $\vecr=\underline{\vecr}$. 

\paragraph{Connection to matrix varieties} It is worth noting that the projection~\eqref{eq: orth proj onto tangent cone by fixing Uk Uk1} also coincides with~\eqref{eq: projection onto tangents matrix} in matrix case. Specifically, given a matrix $\matx\in\mathbb{R}^{m\times n}_{\underline{r}}$ and the SVD of $\matx=\matu\Sigma\matv^\T$, by following the steps in tensor case, \eqref{eq: reformulation of orth proj}~boils down to 
\begin{equation}
    \label{eq: projection when d=2}
    \max_{\matu_{1,1},\matv_{2,1}}\|\matu_{1,1}^\T\mata\matv_{2,1}\|_\frob^2,\quad \subjectto\ \begin{bmatrix}
                \matu & \matu_{1,1}
            \end{bmatrix}^\T\begin{bmatrix}
                \matu & \matu_{1,1}
            \end{bmatrix}=\begin{bmatrix}
                \matv & \matv_{2,1}
            \end{bmatrix}^\T\begin{bmatrix}
                \matv & \matv_{2,1}
            \end{bmatrix}=\matI_{r}.
\end{equation}
In fact, \eqref{eq: projection when d=2} has a closed-form solution $(\matu^*,\matv^*)$, which is the leading $(r-\underline{r})$ singular vectors of the matrix $\proj_\matu^\perp\!\mata\!\proj_\matv^\perp$. Specifically, \revise{since $\matu^\T\matu^*=0$ and $\matv^\T\matv^*=0$, $(\matu^*,\matv^*)$ is a feasible point of~\eqref{eq: projection when d=2}}. Furthermore, for all feasible points $(\matu_{1,1},\matv_{2,1})$, it follows from $\proj_\matu^\perp\!\matu_{1,1}^{}=\matu_{1,1}$,  $\proj_\matv^\perp\!\matv_{2,1}^{}=\matv_{2,1}$, and the Eckart--Young theorem that 
\[\|\matu_{1,1}^\T\mata\matv_{2,1}^{}\|_\frob^2=\|\matu_{1,1}^\T\!\proj_\matu^\perp\!\mata\!\proj_\matv^\perp\!\matv_{2,1}^{}\|_\frob^2\leq\|(\matu^*)^\T\!\proj_\matu^\perp\!\mata\!\proj_\matv^\perp\!\matv^*\|_\frob^2=\|(\matu^*)^\T\mata\matv^*\|_\frob^2.\]
Therefore, the optimality of $(\matu^*,\matv^*)$ of~\eqref{eq: projection when d=2} is verified.

We yield from~\eqref{eq: orth proj onto tangent cone by fixing Uk Uk1} that
\begin{equation*}
    \begin{aligned}
        \proj_{\tangent_\matx\!\mathbb{R}^{m\times n}_{\leq r}}\!\mata&=\proj_{[\matu\ \matu^*]}\!\mata\!\proj_{[\matv\ \matv^*]}+\proj_{[\matu\ \matu^*]}^\perp\!\mata\!\proj_{\matv}+\proj_\matu\!\mata\!\proj_{[\matv\ \matv^*]}^\perp\\
        &=\proj_{\tangent_\matx\!\mathbb{R}^{m\times n}_{\underline{r}}}\!\mata+\proj_{\leq (r-\underline{r})}\!\left(\proj_\matu^\perp\!\mata\!\proj_\matv^\perp\right),
    \end{aligned}
\end{equation*}
which coincides with the known results in~\eqref{eq: projection onto tangents matrix}.

In contrast with the matrix case, the global \revise{minimizer} of~\eqref{eq: reformulation of orth proj} for $d\geq 3$ is computationally intractable. Therefore, pursuing an approximate projection emerges as a more practical way to approach the metric projection.

\subsection{An approximate projection}
The projection in~\eqref{eq: orth projection on tangent cone} is unavailable in general, and hence it is natural to consider an approximation. To this end, given $\tensX=\tensG\times_{k=1}^d\matu_k$ and any $\tilde{\matu}_{k,1}\in\St(r_k-\underline{r}_k,n_k)$ that satisfies $\proj_{\matu_k}\!\tilde{\matu}_{k,1}^{}=0$ for $k\in[d]$, we construct an approximate projection via~\eqref{eq: orth proj onto tangent cone by fixing Uk Uk1} by substituting $\tilde{\matu}_{k,1}^{}$ for $\matu_{k,1}^*$ as follows.

\begin{proposition}\label{prop: approximate projection}
    Given $\tilde{\matu}_{k,1}\in\St(r_k-\underline{r}_k,n_k)$ that satisfies $\proj_{\matu_k}\!\tilde{\matu}_{k,1}=0$ for $k\in[d]$, the approximate projection, defined by 
    \begin{equation}
        \label{eq: an approximate projection}
    \approj_{\tangent_\tensX\!\tensM_{\leq\vecr}}\!\tensA:=\tensA\times_{k=1}^d\proj_{\tilde{\matS}_{k}}+\sum_{k=1}^d\tensG\times_k\left(\proj^\perp_{\tilde{\matS}_{k}}\!\left(\tensA\times_{j\neq k}^{}\matu_j^\T\right)_{(k)}\matG_{(k)}^\dagger\right)\times_{j\neq k}\matu_j,
    \end{equation}
    satisfies $\approj_{\tangent_\tensX\!\tensM_{\leq\vecr}}\!\tensA\in\tangent_{\tensX}\!\tensM_{\leq\vecr}$ and $\langle \tensA, \approj_{\tangent_\tensX\!\tensM_{\leq\vecr}}\!\tensA\rangle=\|\approj_{\tangent_\tensX\!\tensM_{\leq\vecr}}\!\tensA\|_\frob^2$, where $\tilde{\matS}_{k}:=[\matu_k\ \tilde{\matu}_{k,1}]$.
\end{proposition}
\begin{proof}
    For the sake of brevity, we introduce the notation $\approj_{\tangent_\tensX\!\tensM_{\leq\vecr}}\!\tensA=\tilde{\tensV}_0+\sum_{k=1}^d\tilde{\tensV}_k$ in a similar fashion as~\eqref{eq: definition of Vk} for~\eqref{eq: an approximate projection}. It is direct to verify that 
    \begin{equation*}
        \begin{aligned}
            \tilde{\tensV}_0 &=\left(\tensA\times_{j=1}^d[\matu_j\ \tilde{\matu}_{j,1}]^\T\right)\times_{i=1}^d[\matu_i\ \tilde{\matu}_{i,1}],\\
            \tilde{\tensV}_k &=\tensG\times_k\left(\tilde{\matu}_{k,2}^{}\tilde{\matu}_{k,2}^\T\left(\tensA\times_{j\neq k}^{}\matu_j^\T\right)_{(k)}\matG_{(k)}^\dagger\right)\times_{j\neq k}\matu_j,
        \end{aligned}
    \end{equation*}
    where $\tilde{\matu}_{k,2}\in\St(n_k-r_k,n_k)$ satisfies $[\matu_k\ \tilde{\matu}_{k,1}\ \tilde{\matu}_{k,2}]\in\mathcal{O}(n_k)$. Hence, $\approj_{\tangent_\tensX\!\tensM_{\leq\vecr}}\!\tensA$ is of the form~\eqref{eq: Tucker tangent cone} in the tangent cone $\tangent_{\tensX}\!\tensM_{\leq\vecr}$.

    Note that $\langle\tilde{\tensV}_i,\tilde{\tensV}_j\rangle=0$ for $i\neq j$ and thus $\|\approj_{\tangent_\tensX\!\tensM_{\leq\vecr}}\!\tensA\|_\frob^2=\sum_{k=0}^d\|\tilde{\tensV}_k\|_\frob^2$. By using  $\proj_{\tilde{\matS}_{k}}^2=\proj_{\tilde{\matS}_{k}}=\proj_{\tilde{\matS}_{k}}^\T$, $(\proj^\perp_{\tilde{\matS}_{k}})^2=\proj^\perp_{\tilde{\matS}_{k}}=(\proj^\perp_{\tilde{\matS}_{k}})^\T$, $(\matG_{(k)}^\dagger\matG_{(k)}^{})^2=\matG_{(k)}^\dagger\matG_{(k)}^{}=(\matG_{(k)}^\dagger\matG_{(k)}^{})^\T$, and $\matv_k^\T\matv_k^{}=\matI_{r_{-k}}$ with $\matv_k=(\matu_j)^{\otimes j\neq k}$, we yield that 
    \begin{equation*}
        \begin{aligned}
            \left\langle \tensA, \approj_{\tangent_\tensX\!\tensM_{\leq\vecr}}\!\tensA\right\rangle
            &=\left\langle \tensA, \tilde{\tensV}_0+\sum_{k=1}^d\tilde{\tensV}_k\right\rangle
            =\langle \tensA, \tensA\times_{k=1}^d\proj_{\tilde{\matS}_{k}}\rangle+\sum_{k=1}^d\langle \tensA,\tilde{\tensV}_k\rangle\\
            &=\left\|\tilde{\tensV}_0\right\|_\frob^2+\sum_{k=1}^d\left\langle\mata_{(k)},\left(\proj_{\tilde{\matS}_{k}}^\perp\!\mata_{(k)}\matv_k\matG_{(k)}^\dagger\right)\matG_{(k)}\matv_k^\T\right\rangle\\
            &=\left\|\tilde{\tensV}_0\right\|_\frob^2+\sum_{k=1}^d\left\|\left(\proj_{\tilde{\matS}_{k}}^\perp\!\mata_{(k)}\matv_k\matG_{(k)}^\dagger\right)\matG_{(k)}\matv_k^\T\right\|_\frob^2\\
            &=\left\|\tilde{\tensV}_0\right\|_\frob^2+\sum_{k=1}^d\left\|\tensG\times_k\left(\proj^\perp_{\tilde{\matS}_{k}}\!\left(\tensA\times_{j\neq k}^{}\matu_j^\T\right)_{(k)}\matG_{(k)}^\dagger\right)\times_{j\neq k}\matu_j\right\|_\frob^2\\
            &=\|\tilde{\tensV}_0\|_\frob^2+\sum_{k=1}^d\|\tilde{\tensV}_k\|_\frob^2\\
            &=\left\|\approj_{\tangent_\tensX\!\tensM_{\leq\vecr}}\!\tensA\right\|_\frob^2.
        \end{aligned}
    \end{equation*}\qed
\end{proof}

A direct consequence of Proposition~\ref{prop: approximate projection} is that $\langle \tensA, \proj_{\tangent_\tensX\!\tensM_{\leq\vecr}}\!\tensA\rangle=\|\proj_{\tangent_\tensX\!\tensM_{\leq\vecr}}\!\tensA\|_\frob^2$ and $\|\tensA-\approj_{\tangent_\tensX\!\tensM_{\leq\vecr}}\!\tensA\|_\frob^2=\|\tensA\|_\frob^2-\|\approj_{\tangent_\tensX\!\tensM_{\leq\vecr}}\!\tensA\|_\frob^2$. Therefore, it follows from~\eqref{eq: orth projection on tangent cone} that
    \begin{align}
    \|\approj_{\tangent_\tensX\!\tensM_{\leq\vecr}}\!\tensA\|_\frob^2&=\|\tensA\|_\frob^2-\|\tensA-\approj_{\tangent_\tensX\!\tensM_{\leq\vecr}}\!\tensA\|_\frob^2\nonumber\\
    &\leq\|\tensA\|_\frob^2-\|\tensA-\proj_{\tangent_\tensX\!\tensM_{\leq\vecr}}\!\tensA\|_\frob^2=\|\proj_{\tangent_\tensX\!\tensM_{\leq\vecr}}\!\tensA\|_\frob^2.\label{eq: tildePA and PA}
    \end{align}

The approximate projection can be explicitly computed by Algorithm~\ref{alg: orth proj onto tangent cone}. 
\revise{Specifically, since it is not straightforward to access the global minimizer of~\eqref{eq: reformulation of orth proj}, we consider randomly choosing matrices $\tilde{\matu}_{k,1}\in\St(r_k-\underline{r}_k,n_k)$ such that $\matu_k^\perp\tilde{\matu}_{k,1}^{}=0$ for $k\in[d]$. For instance, one can resort to QR factorization of the matrix $[\matu_{k}\ \matI_{n_k}]$ and randomly choose $(r_k-\underline{r}_k)$ columns from the last $(n_k-\underline{r}_k)$ columns of the Q-factor. Even though $\tilde{\matu}_{k,1}$ is chosen randomly, $\approj_{\tangent_\tensX\!\tensM_{\leq\vecr}}(-\nabla f(\tensX))\in\tangent_\tensX\!\tensM_{\leq\vecr}$ is still a descent direction for $f$ since
\[\langle -\nabla f(\tensX), \approj_{\tangent_\tensX\!\tensM_{\leq\vecr}}(-\nabla f(\tensX))\rangle=\|\approj_{\tangent_\tensX\!\tensM_{\leq\vecr}}(-\nabla f(\tensX))\|_\frob^2\geq 0\]
by letting $\tensA=-\nabla f(\tensX)$ in Proposition~\ref{prop: approximate projection}. }

\begin{algorithm}
    \caption{The approximate projection onto $\tangent_\tensX\!\tensM_{\leq\vecr}$}
    \label{alg: orth proj onto tangent cone}
    \begin{algorithmic}[1]
        \REQUIRE $\tensX=\tensG\times_1\matu_1\cdots\times_d\matu_d$ with $\ranktc(\tensX)=\underline{\vecr}\leq\vecr$, and a tensor $\tensA\in\mathbb{R}^{n_1\times\cdots\times n_d}$.
        \STATE Choose $\tilde{\matu}_{1,1},\dots,\tilde{\matu}_{d,1}$ \revise{randomly such that $\matu_k^\perp\tilde{\matu}_{k,1}^{}=0$.} 
        \STATE Compute the approximate projection by~\eqref{eq: an approximate projection}.
        \ENSURE $\approj_{\tangent_\tensX\!\tensM_{\leq\vecr}}\!\tensA$.
    \end{algorithmic}
\end{algorithm}

\section{Gradient-related approximate projection method}\label{sec: GRAP}
In this section, we aim to design line search methods by taking advantage of the tangent cone in Theorem~\ref{thm: tangent cone of Tucker}. One instinctive approach is the projected gradient descent method 
\[\tensX^{(t+1)}=\proj_{\leq\vecr}\!\left(\tensX^{(t)}+s^{(t)}\proj_{\tangent_{\tensX^{(t)}}\!\tensM_{\leq\vecr}}(-\nabla f(\tensX^{(t)}))\right),\]
which is a generalization of the Riemannian gradient descent for minimizing $f$ on $\tensM_\vecr$; see, e.g.,~\cite{schneider2015convergence,olikier2023first} for matrix varieties. However, the metric projections $\proj_{\tangent_{\tensX^{(t)}}\!\tensM_{\leq\vecr}}$ and $\proj_{\leq\vecr}$ do not enjoy closed-form expressions. Therefore, we propose the gradient-related approximate projection method by exploiting the approximate projection~\eqref{eq: an approximate projection} and HOSVD. 

\subsection{Algorithm}
Algorithm~\ref{alg: GRAP} lists the proposed gradient-related approximate projection method (GRAP) for solving~\eqref{eq: problem (P)}. 

\begin{algorithm}[htbp]
    \caption{gradient-related approximate projection method (GRAP)}
    \label{alg: GRAP}
    \begin{algorithmic}[1]
        \REQUIRE Initial guess $\tensX^{(0)}\in\tensM_{\leq\vecr}$, $\omega=(0,1]$, backtracking parameters $\rho, a\in(0,1), {{s}_{\min}}>0$.
        
        \WHILE{the stopping criteria are not satisfied}
        \revise{
            \STATE Compute $g^{(t)}=\approj_{\tangent_{\tensX^{(t)}}\!\tensM_{\leq\vecr}}(-\nabla f(\tensX^{(t)}))$ by Algorithm~\ref{alg: orth proj onto tangent cone} until the angle condition~\eqref{eq: angle condition} is satisfied.}\label{alg: GRAP restart}
        \STATE Choose stepsize $s^{(t)}$ by Armijo backtracking line search~\eqref{eq: Armijo}. 
        \STATE Update $\tensX^{(t+1)}=\proj_{\leq\vecr}^{\mathrm{HO}}\!\left(\tensX^{(t)}+s^{(t)}g^{(t)}\right)$ \revise{and $t=t+1$}. \label{line: 4 in GRAP}
        \ENDWHILE
        \ENSURE $\tensX^{(t)}$.
    \end{algorithmic}
\end{algorithm}

Instead of solving~\eqref{eq: orth proj onto tangent cone by fixing Uk Uk1} for exact metric projection, the GRAP method substitutes~\eqref{eq: orth proj onto tangent cone by fixing Uk Uk1} by the approximate projection~\eqref{eq: an approximate projection}. Additionally, the metric projection $\proj_{\leq\vecr}$ is replaced by $\proj_{\leq\vecr}^{\mathrm{HO}}$. In summary, the update of iterates in GRAP is 
\[\tensX^{(t+1)}=\proj^{\mathrm{HO}}_{\leq\vecr}\!\left(\tensX^{(t)}+s^{(t)}\approj_{\tangent_{\tensX^{(t)}}\!\tensM_{\leq\vecr}}(-\nabla f(\tensX^{(t)}))\right).\]
Moreover, to ensure that a search direction $g^{(t)}\in\tangent_{\tensX^{(t)}}\!\tensM_{\leq\vecr}$ is sufficiently gradient-related, the \emph{angle condition}
\begin{equation}
    \label{eq: angle condition}
    \langle-\nabla f(\tensX^{(t)}),g^{(t)}\rangle\geq\omega\,\|\proj_{\tangent_{\tensX^{(t)}}\!\tensM_{\leq\vecr}}(-\nabla f(\tensX^{(t)}))\|_\frob\|g^{(t)}\|_\frob
\end{equation}
is imposed with $\omega\in(0,1]$.

Specifically, if $g^{(t)}=\approj_{\tangent_{\tensX^{(t)}}\!\tensM_{\leq\vecr}}(-\nabla f(\tensX^{(t)}))$ is rejected by the angle condition, we repeat Algorithm~\ref{alg: orth proj onto tangent cone} with other $\tilde{\matu}_{k,1}$ until the angle condition is satisfied. The rationale behind this is that the approximate projection~\eqref{eq: an approximate projection} depends on $\Span([\matu_k\ \tilde{\matu}_{k,1}])$ but not a specific $\tilde{\matu}_{k,1}$, and going through all choices of $\Span(\tilde{\matu}_{k,1})$ is able to find the exact projection~\eqref{eq: orth proj onto tangent cone by fixing Uk Uk1} of $-\nabla f(\tensX^{(t)})$, which satisfies the angle condition. In practice, the angle condition can be ignored if $\tensX^{(t)}\in\tensM_\vecr$ since the search direction $g^{(t)}$ boils down to  the exact projection.

For the selection of stepsize, we consider the Armijo backtracking line search. Specifically, given an initial stepsize $s_0^{(t)}>0$, find the smallest integer $l$, such that for {$s^{(t)}=\rho^l 
s_0^{(t)}>{s}_{\min}$}, the inequality 
\begin{equation}
    f(\tensX^{(t)})-f(\proj_{\leq\vecr}^{\mathrm{HO}}(\tensX^{(t)}+s^{(t)}g^{(t)}))\geq s^{(t)} a \langle-\nabla f(\tensX^{(t)}),g^{(t)}\rangle\label{eq: Armijo}
\end{equation}
holds, where $\rho,\ a\in(0,1),\ {{s}_{\min}}>0$ are backtracking parameters. \revise{The Armijo condition~\eqref{eq: Armijo} is always achievable if $\|g^{(t)}\|_\frob\neq 0$. Denote $\tensX(s)=\proj_{\leq\vecr}^{\mathrm{HO}}(\tensX^{(t)}+s g^{(t)})$. It follows from the Taylor expansion and Proposition~\ref{prop: retraction} that
\begin{equation*}
    \begin{aligned}
        f(\tensX(s))&=f(\tensX^{(t)})+\langle\nabla f(\tensX^{(t)}),\proj_{\leq\vecr}^{\mathrm{HO}}(\tensX^{(t)}+sg^{(t)})-\tensX^{(t)}\rangle+o(\|\tensX(s)-\tensX^{(t)}\|_\frob)\\
        &=f(\tensX^{(t)})+\langle\nabla f(\tensX^{(t)}),sg^{(t)}+o(s)\rangle+o(\|\tensX(s)-\tensX^{(t)}\|_\frob)\\
        &=f(\tensX^{(t)})+sa\langle\nabla f(\tensX^{(t)}),g^{(t)}\rangle+s(1-a)\langle\nabla f(\tensX^{(t)}),g^{(t)}\rangle+o(s).
    \end{aligned}
\end{equation*}
Since $\langle-\nabla f(\tensX^{(t)}),g^{(t)}\rangle=\|g^{(t)}\|_\frob^2>0$ from Proposition~\ref{prop: approximate projection}, it holds that 
\begin{equation*}
    \begin{aligned}
        f(\tensX^{(t)})-f(\tensX(s))&=sa\langle-\nabla f(\tensX^{(t)}),g^{(t)}\rangle+s(1-a)\|g^{(t)}\|_\frob^2+o(s)\\
        &\geq sa\langle-\nabla f(\tensX^{(t)}),g^{(t)}\rangle
    \end{aligned}
\end{equation*}
for sufficiently small $s>0$. }

\revise{It is worth noting that the proposed GRAP method is a line search method on $\tensM_{\leq\vecr}$ that is able to deal with rank-deficient points in $\tensM_{\leq\vecr}\setminus\tensM_{\vecr}$, although GRAP method will not generate a rank-deficient point in the most likely cases. Similar observations can be found in the matrix case~\cite{schneider2015convergence}.}

\subsection{Global convergence}
Let $\{\tensX^{(t)}\}_{t\geq 0}$ be an infinite sequence generated by Algorithm~\ref{alg: GRAP}. We prove that the stationary measure $\|\proj_{\tangent_{\tensX^{(t)}}\!\tensM_{\leq\vecr}}(-\nabla f(\tensX^{(t)}))\|_\frob$ converges to $0$.
\begin{theorem}\label{thm: GRAP global}
    Let $\{\tensX^{(t)}\}_{t\geq 0}$ be an infinite sequence generated by Algorithm~\ref{alg: GRAP}. Assume $f$ is bounded below \revise{by $f^*$}, it holds that 
    \[\lim_{t\to\infty} \|\proj_{\tangent_{\tensX^{(t)}}\!\tensM_{\leq\vecr}}(-\nabla f(\tensX^{(t)}))\|_\frob=0.\]
    Moreover, Algorithm~\ref{alg: GRAP} returns $\tensX^{(t)}\in\tensM_{\leq\vecr}$ satisfying $\|\!\proj_{\tangent_{\tensX^{(t)}}\!\tensM_{\leq\vecr}}(-\nabla f(\tensX^{(t)}))\|_\frob<\epsilon$ after $\left\lceil{f(\tensX^{(0)})}/{({s}_{\min}a\,\omega^2\epsilon^2)}\right\rceil$ iterations at most. 
\end{theorem}
\begin{proof}
    It follows from~\eqref{eq: angle condition}--\eqref{eq: Armijo} and Proposition~\ref{prop: approximate projection} that 
    \begin{equation*}
        \begin{aligned}
            f(\tensX^{(t)})-f(\tensX^{(t+1)})&\geq s^{(t)} a \langle-\nabla f(\tensX^{(t)}),g^{(t)}\rangle\\
            &\geq {s}_{\min}a\|g^{(t)}\|_\frob^2\\
            &\geq {s}_{\min}a\,\omega^2\|\proj_{\tangent_{\tensX^{(t)}}\!\tensM_{\leq\vecr}}(-\nabla f(\tensX^{(t)}))\|_\frob^2.
        \end{aligned}
    \end{equation*}
    Therefore, it yields
    \[f(\tensX^{(0)})-f^*\geq\sum_{t=0}^{\infty} {s}_{\min}a\,\omega^2\|\proj_{\tangent_{\tensX^{(t)}}\!\tensM_{\leq\vecr}}(-\nabla f(\tensX^{(t)}))\|_\frob^2,\]
    and thus $\|\proj_{\tangent_{\tensX^{(t)}}\!\tensM_{\leq\vecr}}(-\nabla f(\tensX^{(t)}))\|_\frob$ converges to $0$. 
    
    Moreover, assume that $\|\proj_{\tangent_{\tensX^{(t)}}\!\tensM_{\leq\vecr}}(-\nabla f(\tensX^{(t)}))\|_\frob\geq\epsilon$ holds for $t=0,1,\dots,T$. We have 
    \[f(\tensX^{(0)})-f(\tensX^{(T)})\geq{s}_{\min}a\,\omega^2\epsilon^2 T,\]
    and thus $T\leq{f(\tensX^{(0)})}/{({s}_{\min}a\,\omega^2\epsilon^2)}$.
    \qed
\end{proof}

It is worth noting that it requires $\mathcal{O}(\varepsilon^{-2})$ steps to achieve an $\varepsilon$-stationary point, which coincides with the classical results in Riemannian optimization; see, e.g.,~\cite[Theorem 2.5]{boumal2019global}.

\subsection{Local convergence}
Given an infinite sequence $\{\tensX^{(t)}\}_{t\geq 0}$ generated by Algorithm~\ref{alg: GRAP}, we analyze the local convergence by exploiting the \emph{\L{}ojasiewicz gradient inequality}~\cite[Definition 2.1]{schneider2015convergence}. A point $\tensX\in\tensM_{\leq\vecr}$ is said to satisfy the \L{}ojasiewicz gradient inequality if there exists $\delta,L>0$ and $\theta\in(0,1/2]$ such that
\begin{equation}
    \label{eq: Lojasiewicz}
    |f(\tensX)-f(\tensY)|^{1-\theta}\leq L\|\proj_{\tangent_\tensX\!\tensM_{\leq\vecr}}(-\nabla f(\tensY))\|_\frob
\end{equation}
holds for all $\tensY\in\tensM_{\leq\vecr}$ with $\|\tensY-\tensX\|_\frob\leq\delta$. Under the assumption that $f$ satisfies~\eqref{eq: Lojasiewicz}, we can prove the local convergence of Algorithm~\ref{alg: GRAP} is as follows.

\begin{theorem}\label{thm: GRAP local}
    Let $\{\tensX^{(t)}\}_{t\geq 0}$ be an infinite sequence generated by Algorithm~\ref{alg: GRAP}. Assume that $f$ is bounded below by $f^*$ and satisfies the \L{}ojasiewicz gradient inequality. If $\{\tensX^{(t)}\}_{t\geq 0}$ has an accumulation point $\tensX^*$, then $\tensX^{(t)}$ converges to $\tensX^*$.
    Furthermore, if $\ranktc(\tensX^*)=\vecr$, then the stationary measure $\|\proj_{\tangent_{\tensX^*}\!\tensM_{\leq\vecr}}(-\nabla f(\tensX^*))\|_\frob=\|\grad f(\tensX^*)\|_\frob=0$ and 
    \[\|\tensX^{(t)}-\tensX^*\|_\frob\leq C
    \left\{
        \begin{array}{ll}
            \mathrm{e}^{-ct},&\ \text{if}\ \theta=\frac{1}{2},\\
        t^{-\frac{\theta}{1-2\theta}},&\ \text{if}\ 0<\theta<\frac{1}{2}
        \end{array}
    \right.\]
    hold for some $C,c>0$.
\end{theorem}
\begin{proof}
    Using~\cite[Theorem 2.3, Corollary 2.11]{schneider2015convergence}, it suffices to prove the following three claims: 1) there exists $c>0$, such that 
    \[f(\tensX^{(t)})-f(\tensX^{(t+1)})\geq c\,\|\proj_{\tangent_{\tensX^{(t)}}\!\tensM_{\leq\vecr}}(-\nabla f(\tensX^{(t)}))\|_\frob\|\tensX^{(t)}-\tensX^{(t+1)}\|_\frob;\] 2) the stationary measure $\|\proj_{\tangent_{\tensX^{(t)}}\!\tensM_{\leq\vecr}}(-\nabla f(\tensX^{(t)}))\|_\frob=0$ implies $\tensX^{(t+1)}=\tensX^{(t)}$ for sufficiently large $t$; 3) if $\ranktc(\tensX^*)=\vecr$, $\tensX^{(t)}\in\tensM_\vecr$ for sufficiently large $t$. 

    For the first claim, it follows from~\eqref{eq: angle condition}--\eqref{eq: Armijo} and Corollary~\ref{coro: bound of retraction} that
    \begin{equation*}
        \begin{aligned}
            f(\tensX^{(t)})-f(\tensX^{(t+1)})&\geq  a\omega\,\|\proj_{\tangent_{\tensX^{(t)}}\!\tensM_{\leq\vecr}}(-\nabla f(\tensX^{(t)}))\|_\frob\|{s}^{(t)}g^{(t)}\|_\frob\\
            &\geq\frac{a\omega}{M}\|\proj_{\tangent_{\tensX^{(t)}}\!\tensM_{\leq\vecr}}(-\nabla f(\tensX^{(t)}))\|_\frob\|\tensX^{(t)}-\tensX^{(t+1)}\|_\frob,
        \end{aligned}
    \end{equation*}
    \revise{where $M:=1+d/\sqrt{d+1}>0$. }

    Moreover, assume that $\|\proj_{\tangent_{\tensX^{(t)}}\!\tensM_{\leq\vecr}}(-\nabla f(\tensX^{(t)}))\|_\frob=0$, holds for some large~$t$. \revise{It follows from~\eqref{eq: tildePA and PA} that 
    \[\|g^{(t)}\|_\frob=\|\approj_{\tangent_{\tensX^{(t)}}\!\tensM_{\leq\vecr}}(-\nabla f(\tensX^{(t)}))\|_\frob\leq\|\proj_{\tangent_{\tensX^{(t)}}\!\tensM_{\leq\vecr}}(-\nabla f(\tensX^{(t)}))\|_\frob=0,\] 
    and thus $g^{(t)}=0$ and $\tensX^{(t+1)}=\tensX^{(t)}$. }

    \revise{The third claim is straightforward} since $\tensM_\vecr$ is open in $\tensM_{\leq\vecr}$.
    \qed
\end{proof}

\subsection{Discussion: \revise{apocalyptic points}}
Observe that even though $\{\tensX^{(t)}\}_{t\geq 0}$ has an accumulation point $\tensX^*$, the condition $\lim_{t\to\infty}\|\proj_{\tangent_{\tensX^{(t)}}\!\tensM_{\leq\vecr}}(-\nabla f(\tensX^{(t)}))\|_\frob=0$ does not necessarily guarantee that~$\tensX^*$ is a stationary point of $f$, i.e., $\|\proj_{\tangent_{\tensX^*}\!\tensM_{\leq\vecr}}(-\nabla f(\tensX^*))\|_\frob$ can be nonzero; see the following example.
\begin{example}\label{eg: apocalypse}
    Given $\tensA=\vece_1\circ\vece_1\circ\vece_1+\vece_3\circ\vece_3\circ\vece_3\in\mathbb{R}^{n\times n\times n}$ and $\vecr=(2,2,2)$,
    consider the objective function $f(\tensX)=\|\tensX-\tensA\|_\mathrm{F}^2$, and initial guess $\tensX^{(0)}=\vece_1\circ\vece_1\circ\vece_1+\vece_2\circ\vece_2\circ\vece_2$. Then, the sequence generated by Algorithm~\ref{alg: GRAP} with a constant stepsize $s^{(t)}=\alpha\in(0,1)$ is explicitly
    \[\tensX^{(t)}=\vece_1\circ\vece_1\circ\vece_1+(1-\alpha)^{t}\vece_2\circ\vece_2\circ\vece_2,\]
    which converges to $\tensX^*=\vece_1\circ\vece_1\circ\vece_1$. According to Theorem~\ref{thm: GRAP global}, the sequence satisfies that  $\lim_{t\to\infty}\|\proj_{\tangent_{\tensX^{(t)}}\!\tensM_{\leq\vecr}}(-\nabla f(\tensX^{(t)}))\|_\frob=0$.    
    However, since $\nabla f(\tensX^*)\neq 0,$ $\tensX^*$ is not a stationary point of~$f$ by using \revise{Proposition~\ref{prop: stationary}}.    
\end{example}

More recently, such a phenomenon has been investigated and named by \emph{apocalypse}~\cite{levin2023finding}. Specifically, a point $\tensX^*\in\tensM_{\leq\vecr}$ is called \emph{apocalyptic} if there exists a sequence $\{\tensX^{(t)}\}\subseteq\tensM_{\leq\vecr}$ converging to $\tensX^*$ and a smooth function $f$, such that 
\[\lim_{t\to\infty}\|\proj_{\tangent_{\tensX^{(t)}}\!\tensM_{\leq\vecr}}(-\nabla f(\tensX^{(t)}))\|_\frob=0\ \ \text{but}\ \ \|\proj_{\tangent_{\tensX^*}\!\tensM_{\leq\vecr}}(-\nabla f(\tensX^*))\|_\frob>0.\]
The triplet $(\tensX^*,\{\tensX^{(t)}\},f)$ is called an apocalypse. Similar to the matrix varieties, we observe that the Tucker tensor varieties suffer from apocalypse at rank-deficient points in $\tensM_{\leq\vecr}$.
\begin{proposition}\label{prop: apocalypse of Tucker}
    If $\tensX^*\in\tensM_{\leq\vecr}$ has $\mathrm{rank}_{\mathrm{tc}}(\tensX^*)=\underline{\vecr}<\vecr$, then $\tensX^*$ is apocalyptic.
\end{proposition}
\begin{proof}
    See Appendix~\ref{app: apocalypse of Tucker}. \qed
\end{proof}

Circumventing apocalypse for Tucker tensor varieties is far beyond the main purpose of this paper. Hence, we leave it for future research. 

\section{A retraction-free gradient-related approximate projection method}\label{sec: rfGRAP}
In Algorithm~\ref{alg: GRAP}, the retraction $\retr^{\mathrm{HO}}_{\tensX}$, which requires computation of HOSVD to a $d$-dimensional tensor with size $(\underline{r}_1+r_1)\times\cdots\times(\underline{r}_d+r_d)$, is inevitable to preserve the constraint, Tucker tensor varieties. One is curious about exploiting information of the tangent cone $\tangent_\tensX\!\tensM_{\leq\vecr}$ to construct search directions and facilitate line search without retraction to save computational cost. In this section, we \revise{propose} partial projections to develop retraction-free line search method on $\tensM_{\leq\vecr}$.

\subsection{New partial projections}
Recall that any $\tensV$ in the tangent cone $\tangent_\tensX\!\tensM_{\leq{\vecr}}$ at $\tensX=\tensG\times_{k=1}^d\matu_k$ with $\ranktc(\tensX)=\underline{\vecr}$ can be parametrized in terms of $(\tensC,\{\matu_{k,1}\}_{k=1}^d,\{\matu_{k,2}\}_{k=1}^d,\{\matR_{k,2}\}_{k=1}^d)$ by~\eqref{eq: Tucker tangent cone}, i.e., 
\[\tensV=\tensV_0+\sum_{k=1}^d\tensV_k=\tensC\times_{k=1}^d\begin{bmatrix}
            \matu_k & \matu_{k,1}
        \end{bmatrix}+\sum_{k=1}^d\tensG\times_k(\matu_{k,2}\matR_{k,2})\times_{j\neq k}\matu_j.\]
In view of~\eqref{eq: retraction-free directions}--\eqref{eq: retraction-free directions k}, searching along such $(d+1)$ directions $\{\tensV_k\}$ is able to get rid of retractions. In the light of this, we consider partial projections that are of the form in $\{\tensV_k\}$.

Given $\tensA\in\mathbb{R}^{n_1\times n_2\times\cdots\times n_d}$, the partial projections are defined as follows.
    \begin{align}
        \proj_{0}(\tensA)&:=\argmin_{\tensV_0}\left\{\|\tensV_0-\tensA\|:\tensV_0=\tensC\times_{k=1}^d\begin{bmatrix}
            \matu_k & \matu_{k,1}
        \end{bmatrix}\in\tangent_\tensX\!\tensM_{\leq\vecr}\right\}\label{eq: partial G},\\
        \proj_{k}(\tensA)&:=\argmin_{\tensV_k}\left\{\|\tensV_k-\tensA\|:\tensV_k=\tensG\times_k(\matu_{k,2}\matR_{k,2})\times_{j\neq k}\matu_j\in\tangent_\tensX\!\tensM_{\leq\vecr}\right\}.\label{eq: partial k}
    \end{align}
Since~\eqref{eq: partial G} does not enjoy a closed-form solution, we consider its approximation 
\begin{equation}
    \label{eq: partial projection P0}
    \approj_{0}(\tensA):=\tensA\times_{k=1}^d\proj_{\tilde{\matS}_{k}},
\end{equation} 
which is exactly the first term of the approximate projection $\approj_{\tangent_\tensX\!\tensM_{\leq\vecr}}(\tensA)$ in~\eqref{eq: an approximate projection} for given $\tilde{\matu}_{k,1}\in\St(r_k-\underline{r}_k,n_k)$ with $\tilde{\matu}_{k,1}^\T\matu_k^{}=0$, where $\tilde{\matS}_{k}:=[\matu_k\ \tilde{\matu}_{k,1}]$ for $k\in[d]$. It is worth noting that $\tilde{\proj}_{0}(\tensA)$ requires the parameters $\{\tilde{\matu}_{k,1}\}_{k=1}^d$ a prior.
Moreover, by fixing $\matu_{k,2}$ in~\eqref{eq: partial k} and using $[\matu_k\ \matu_{k,1}\ \matu_{k,2}]\in\mathcal{O}(n_k)$ and~\eqref{eq: orth proj onto tangent cone by fixing Uk Uk1}, $\proj_{k}(\tensA)$ is in the form of
\[\proj_{k}(\tensA)=\tensG\times_k\left(\proj_{\matu_{k,2}}\left(\tensA\times_{j\neq k}\matu_j^\T\right)_{(k)}\matG_{(k)}^\dagger\right)\times_{j\neq k}\matu_j.\]
Similarly, since $\matu_{k,2}$ is unknown, we consider substituting the projection $\proj_{\matu_{k,2}}$ by $\proj_{\matu_{k}}^\perp$ and yield an approximation
\begin{equation}
    \label{eq: partial proj Pk}
    \tilde{\proj}_{k}(\tensA)=\tensG\times_k\left(\proj_{\matu_{k}}^\perp\left(\tensA\times_{j\neq k}\matu_j^\T\right)_{(k)}\matG_{(k)}^\dagger\right)\times_{j\neq k}\matu_j,
\end{equation}
\revise{which is different from $\tilde{\tensV}_k$ for the approximate projection~\eqref{eq: an approximate projection} in Proposition~\ref{prop: approximate projection} since $\proj_{\matu_{k}}^\perp\neq\proj_{\tilde{\matS}_{k}}^\perp$.}

\revise{The partial projections in~\eqref{eq: partial projection P0} and~\eqref{eq: partial proj Pk} satisfy that} $\tilde{\proj}_{k}(\tensA)\in\tangent_{\tensX}\!\tensM_{\leq\vecr}$ and $\langle \tensA,\tilde{\proj}_{k}(\tensA)\rangle=\|\tilde{\proj}_{k}(\tensA)\|_\frob^2$, which can be proved in a similar fashion as Proposition~\ref{prop: approximate projection}. Additionally, we can prove $\ranktc(\tensX+\tilde{\proj}_{k}(\tensA))\leq\vecr$, i.e., \[\tensX+\tilde{\proj}_{k}(\tensA)\in\tensM_{\leq\vecr},\] in a similar fashion as~\eqref{eq: retraction-free directions}--\eqref{eq: retraction-free directions k}. However, this property does not necessarily hold for two different partial projections $\approj_{j}(\tensA)$ and $\tilde{\proj}_{k}(\tensA)$ with $j\neq k$, i.e., $\ranktc(\tensX+\approj_{j}(\tensA)+\tilde{\proj}_{k}(\tensA))$ can be larger than~$\vecr$.

\subsection{Algorithm and convergence results}
To sum up, we propose the partial projection operator defined by
\begin{equation}
    \label{eq: proposed partial projection}
    \hat{\proj}_{\tangent_{\tensX}\!\tensM_{\leq\vecr}}(\tensA):=\argmax_{\tensV\in\{\tilde{\proj}_{0}(\tensA),\dots,\approj_d(\tensA)\}}\|\tensV\|_\frob.
\end{equation}
By using the partial projection~\eqref{eq: proposed partial projection}, we propose a retraction-free gradient-related approximate projection method (rfGRAP) in Algorithm~\ref{alg: rfGRAP}. The iteration of the proposed rfGRAP method is
\[\tensX^{(t+1)}=\tensX^{(t)}+s^{(t)}\hat{\proj}_{\tangent_{\tensX^{(t)}}\!\tensM_{\leq\vecr}}(-\nabla f(\tensX^{(t)})),\]
where $s^{(t)}$ is also computed by Armijo backtracking line search~\eqref{eq: Armijo}. In contrast with the proposed GRAP method (Algorithm~\ref{alg: GRAP}), there is no retraction in Algorithm~\ref{alg: rfGRAP}. 

\begin{algorithm}[htbp]
    \caption{Retraction-free gradient-related approximate projection method (rfGRAP)}
    \label{alg: rfGRAP}
    \begin{algorithmic}[1]
        \REQUIRE Initial guess $\tensX^{(0)}\in\tensM_{\leq\vecr}$, \revise{$\omega\in(0,1/\sqrt{d+1})$,} backtracking parameters $\rho, a\in(0,1), {{s}_{\min}}>0$.
        \WHILE{the stopping criteria are not satisfied}
        \revise{
            \STATE Compute $g^{(t)}=\hat{\proj}_{\tangent_{\tensX^{(t)}}\!\tensM_{\leq\vecr}}(-\nabla f(\tensX^{(t)}))$ by~\eqref{eq: proposed partial projection} with random $\tilde{\matu}_{1,1},\tilde{\matu}_{2,1},\dots,\tilde{\matu}_{d,1}$ until the angle condition~\eqref{eq: angle condition} is satisfied.}
        \STATE Choose stepsize $s^{(t)}$ by Armijo backtracking line search~\eqref{eq: Armijo}.
        \STATE Update $\tensX^{(t+1)}=\tensX^{(t)}+s^{(t)}g^{(t)}$ \revise{and $t=t+1$}.
        \ENDWHILE
        \ENSURE $\tensX^{(t)}$
    \end{algorithmic}
\end{algorithm}

\revise{Similar to Algorithm~\ref{alg: GRAP}, if $g^{(t)}=\hat{\proj}_{\tangent_{\tensX^{(t)}}\!\tensM_{\leq\vecr}}(-\nabla f(\tensX^{(t)}))$ does not satisfy the angle condition~\eqref{eq: angle condition}, we repeat the projection with other $\tilde{\matu}_{k,1}$ until the angle condition is satisfied. Since the partial projection~\eqref{eq: proposed partial projection} adopts partial information of the tangent cone, we choose the parameter $\omega\in(0,1/\sqrt{d+1})$. Let $\tilde{\matu}_{k,1}$ in~\eqref{eq: partial projection P0} be the global minimizer $\matu_{k,1}^*$ of~\eqref{eq: reformulation of orth proj}, it follows from~\eqref{eq: partial projection P0}--\eqref{eq: proposed partial projection} that 
\begin{equation*}
    \begin{aligned}
        &~~~\|\hat{\proj}_{\tangent_{\tensX}\!\tensM_{\leq\vecr}}(\tensA)\|_\frob^2=\max_{k=0,1,\dots,d}\{\|\approj_k(\tensA)\|_\frob^2\}\geq\frac{1}{d+1}\sum_{k=0}^d\|\approj_k(\tensA)\|_\frob^2\\
        &=\frac{1}{d+1}\big(\| \tensA\times_{k=1}^d\proj_{{\matS}^*_{k}}\|_\frob^2+\sum_{k=1}^d\|\tensG\times_k(\proj_{\matu_{k}}^\perp\!( \tensA\times_{j\neq k}\matu_j^\T)_{(k)}\matG_{(k)}^\dagger)\times_{j\neq k}\matu_j\|_\frob^2\big)\\
        &\geq\frac{1}{d+1}\big(\| \tensA\times_{k=1}^d\proj_{{\matS}^*_{k}}\|_\frob^2+\sum_{k=1}^d\|\tensG\times_k(\proj_{{\matS}^*_{k}}^\perp(\tensA\times_{j\neq k}\matu_j^\T)_{(k)}\matG_{(k)}^\dagger)\times_{j\neq k}\matu_j\|_\frob^2\big)\\
        &=\frac{1}{d+1}\|\proj_{\tangent_\tensX\!\tensM_{\leq\vecr}}(\tensA)\|_\frob^2,
    \end{aligned}
\end{equation*}
where $\matS^*_k=[\matu_k\ \matu_{k,1}^*]$ and $\Span(\matS^*_k)^\perp\subseteq\Span(\matu_k)^\perp$. Therefore, we obtain that 
\begin{equation*}
    \begin{aligned}
        \langle \tensA,\hat{\proj}_{\tangent_{\tensX}\!\tensM_{\leq\vecr}}(\tensA)\rangle&=\|\hat{\proj}_{\tangent_{\tensX}\!\tensM_{\leq\vecr}}(\tensA)\|_\frob^2\geq\frac{1}{\sqrt{d+1}}\|\proj_{\tangent_\tensX\!\tensM_{\leq\vecr}}(\tensA)\|_\frob\|\hat{\proj}_{\tangent_{\tensX}\!\tensM_{\leq\vecr}}(\tensA)\|_\frob.
    \end{aligned}
\end{equation*}
Consequently, it is reasonable to set the parameter $\omega$ in~\eqref{eq: angle condition} of Algorithm~\ref{alg: rfGRAP} by~$\omega\in(0,1/\sqrt{d+1})$. Note that by using $\langle \tensA,\tilde{\proj}_{k}(\tensA)\rangle=\|\tilde{\proj}_{k}(\tensA)\|_\frob^2$ for $k=0,1,\dots,d$, we obtain that $\|\hat{\proj}_{\tangent_{\tensX}\!\tensM_{\leq\vecr}}(\tensA)\|_\frob\leq\|\proj_{\tangent_{\tensX}\!\tensM_{\leq\vecr}}(\tensA)\|_\frob$ similar to~\eqref{eq: tildePA and PA}.
}

The global and local convergence of the rfGRAP method can be proved in a similar fashion as Theorem~\ref{thm: GRAP global} and Theorem~\ref{thm: GRAP local}.
\begin{theorem}\label{thm: rfGRAP convergence}
    Let $\{\tensX^{(t)}\}_{t\geq 0}$ be an infinite sequence generated by Algorithm~\ref{alg: rfGRAP}. Assume $f$ is bounded below \revise{by $f^*$} and satisfies the \L{}ojasiewicz gradient inequality. It holds that 
    \[\lim_{t\to\infty} \|\proj_{\tangent_{\tensX^{(t)}}\!\tensM_{\leq\vecr}}(-\nabla f(\tensX^{(t)}))\|_\frob=0.\]
    The method returns $\tensX^{(t)}\in\tensM_{\leq\vecr}$ satisfying $\|\proj_{\tangent_{\tensX^{(t)}}\!\tensM_{\leq\vecr}}(-\nabla f(\tensX^{(t)}))\|_\frob<\epsilon$ after $\left\lceil{f(\tensX^{(0)})}/{({s}_{\min}a\,\omega^2\epsilon^2)}\right\rceil$ iterations at most.
    \revise{
    If $\{\tensX^{(t)}\}_{t\geq 0}$ has an accumulation point $\tensX^*$, then $\tensX^{(t)}$ converges to $\tensX^*$. Furthermore, if $\ranktc(\tensX^*)=\vecr$, then the stationary measure $\|\proj_{\tangent_{\tensX^*}\!\tensM_{\leq\vecr}}(-\nabla f(\tensX^*))\|_\frob=\|\grad f(\tensX^*)\|_\frob=0$} and 
    \[\|\tensX^{(t)}-\tensX^*\|_\frob\leq C
    \left\{
        \begin{array}{ll}
            \mathrm{e}^{-ct},&\ \text{if}\ \theta=\frac{1}{2},\\
        t^{-\frac{\theta}{1-2\theta}},&\ \text{if}\ 0<\theta<\frac{1}{2}
        \end{array}
    \right.\]
    \revise{hold for some $C,c>0$.}
\end{theorem}

\subsection{Connection to matrix varieties}
We investigate the connection between the proposed rfGRAP method and other existing methods in the matrix case. 

Specifically, given a matrix $\mata\in\mathbb{R}^{m\times n}$ and the SVD $\matx=\matu\Sigma\matv^\T$ of $\matx\in\mathbb{R}^{m\times n}_{\underline{r}}$, the proposed partial projections $\{\approj_{k}(\mata)\}_{k=0}^d$ in~\eqref{eq: partial projection P0}--\eqref{eq: partial proj Pk} boils down to 
\begin{equation*}
    \begin{aligned}
        \approj_{0}(\mata)&=\proj_{\begin{bmatrix}
                \matu & {\matu}_{1}
            \end{bmatrix}}\!\mata\!\proj_{\begin{bmatrix}
                \matv & {\matv}_{1}
            \end{bmatrix}},\\
            \approj_{1}(\mata)&=\proj_\matu^\perp\!\mata\!\proj_\matv,\\
            \approj_{2}(\mata)&=\proj_\matu\!\mata\!\proj_\matv^\perp
        \end{aligned}
    \end{equation*}
when $d=2$, where ${\matu}_1\in\St(r-\underline{r},m),{\matv}_1\in\St(r-\underline{r},n)$, i.e., $\tilde{\matu}_{1,1}$ and $\tilde{\matu}_{2,1}$ in~\eqref{eq: partial projection P0}, are selected by the leading $(r-\underline{r})$ left and right singular vectors of $\proj_\matu^\perp\!\mata\!\proj_\matv^\perp$. Given $\matu_2\in\St(m-r,m)$ with $[\matu\ {\matu}_{1}\ \matu_2]\in\mathcal{O}(m)$ and $\matv_2\in\St(n-r,n)$ with $[\matv\ {\matv}_{1}\ \matv_2]\in\mathcal{O}(n)$. Subsequently, the partial projections can be illustrated by 
\begin{equation*}
    \begin{aligned}
        \approj_{0}(\mata)&=
        \begin{bmatrix}
            \matu & \matu_1 & \matu_2
        \end{bmatrix}
        \vcenter{\hbox{\includegraphics[scale=1.2]{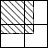}}}
        \begin{bmatrix}
            \matv & \matv_1 & \matv_2
        \end{bmatrix}^\T,\\
        \approj_{1}(\mata)&=
        \begin{bmatrix}
            \matu & \matu_1 & \matu_2
        \end{bmatrix}
        \vcenter{\hbox{\includegraphics[scale=1.2]{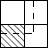}}}
        \begin{bmatrix}
            \matv & \matv_1 & \matv_2
        \end{bmatrix}^\T,\\
        \approj_{2}(\mata)&=
        \begin{bmatrix}
            \matu & \matu_1 & \matu_2
        \end{bmatrix}
        \vcenter{\hbox{\includegraphics[scale=1.2]{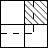}}}
        \begin{bmatrix}
            \matv & \matv_1 & \matv_2
        \end{bmatrix}^\T
    \end{aligned}
\end{equation*}
in the sense of Fig.~\ref{fig: new representation of matrix tangent cone}, which is different from the ``partial projections" 
\begin{equation*}
    \begin{aligned}
        \breve{\proj}_{1}(\mata)&=
        \begin{bmatrix}
            \matu & \matu_1 & \matu_2
        \end{bmatrix}
        \vcenter{\hbox{\includegraphics[scale=1.2]{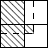}}}
        \begin{bmatrix}
            \matv & \matv_1 & \matv_2
        \end{bmatrix}^\T,\\ 
        \breve{\proj}_{2}(\mata)&=
        \begin{bmatrix}
            \matu & \matu_1 & \matu_2
        \end{bmatrix}
        \vcenter{\hbox{\includegraphics[scale=1.2]{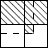}}}
        \begin{bmatrix}
            \matv & \matv_1 & \matv_2
        \end{bmatrix}^\T
    \end{aligned}
\end{equation*}
in~\cite[\S 3.4]{schneider2015convergence}. Therefore, the proposed partial projection $\hat{\proj}_{\tangent_{\matx}\!\mathbb{R}^{m\times n}_{\leq r}}$ is also able to serve as a new retraction-free search direction in optimization on matrix varieties.

\section{A Tucker rank-adaptive method} \label{sec: TRAM}
In practice, choosing an appropriate rank parameter $\vecr$ appears to be a challenging task. A larger $\vecr$ can expand the search space and potentially yield a better optimum while simultaneously increasing computational costs. In addition, as shown in Example~\ref{eg: apocalypse}, even though a sequence generated by GRAP satisfies that $\|\proj_{\tangent_{\tensX^{(t)}}\!\tensM_{\leq\vecr}}(-\nabla f(\tensX^{(t)}))\|_\frob$ converges to $0$, the rank deficiency at an accumulation point can hamper it to be a stationary point. A similar challenge arises in Riemannian optimization methods applied on the fixed-rank Tucker tensor manifold $\tensM_{\vecr}$; see, e.g.,~\cite[Section 5.1]{dong2022new}. Therefore, in this section, we are motivated to propose a rank-adaptive method to adjust the rank of an iterate $\tensX^{(t)}=\tensG^{(t)}\times_{k=1}^d\matu_k^{(t)}$ by designing new rank-adaptive strategies with the search directions in
\begin{equation}
    \label{eq: search in TRAM}
    \tangent_{\tensX^{(t)}}\!\tensM_{\underline{\vecr}^{(t)}}+\normal_{\leq\vecl^{(t)}}(\tensX^{(t)})\subseteq\tangent_{\tensX^{(t)}}\!\tensM_{\leq\vecr},
\end{equation}
where $\normal_{\leq\vecl^{(t)}}(\tensX^{(t)}):=\tensM_{\leq\vecl^{(t)}}\cap\left(\bigotimes_{k=1}^d\Span(\matu_k^{(t)})^\perp\right)\subseteq\normal_{\tensX^{(t)}}\!\tensM_{\underline{\vecr}^{(t)}}$ and $\vecl^{(t)}\in\mathbb{N}_+^d$ satisfies $\vecl^{(t)}\leq\vecr-\underline{\vecr}^{(t)}$.

Generally speaking, we first apply Riemannian optimization on $\tensM_{\underline{\vecr}^{(t)}}$ in section~\ref{subsec: fixed rank}. Then, rank-decreasing and rank-increasing procedures are developed to automatically adjust the rank of an iterate $\tensX^{(t)}$ in sections~\ref{subsec: rank decrease} and~\ref{subsec: rank increase}. In summary, a new Tucker rank-adaptive method (TRAM) is proposed and analyzed in sections~\ref{subsec: TRAM} and~\ref{subsec: TRAM convergence}. The implementation details of TRAM are provided in section~\ref{subsec: practical TRAM}.

\subsection{Line search on fixed-rank manifold}\label{subsec: fixed rank}
Given a point $\tilde{\tensX}^{(t)}\in\tensM_{\underline{\vecr}^{(t)}}$, we observe from~\eqref{eq: Tucker tangent cone} that $\tangent_{\tilde{\tensX}^{(t)}}\!\tensM_{\underline{\vecr}^{(t)}}\subseteq\tangent_{\tilde{\tensX}^{(t)}}\!\tensM_{\leq\vecr}$. Therefore, the negative Riemannian gradient of $f$ at $\tilde{\tensX}^{(t)}$ on $\tensM_{\underline{\vecr}^{(t)}}$ provides a convincing search direction that enjoys a closed-form expression~\eqref{eq: orth proj onto tangent space Tucker}. The Riemannian gradient descent method (RGD) on $\tensM_{\underline{\vecr}^{(t)}}$ is shown in Algorithm~\ref{alg: RGD}. 

\begin{algorithm}[htbp]
    \caption{Riemannian gradient descent method (RGD) on $\tensM_{\underline{\vecr}^{(t)}}$}
    \label{alg: RGD}
    \begin{algorithmic}[1]
        \REQUIRE Initial guess $\tensY^{(0)}=\tilde{\tensX}^{(t)}\in\tensM_{\underline{\vecr}^{(t)}}$; backtracking parameters $\rho, a\in(0,1), {{s}_{\min}}>0$.
        \WHILE{the stopping criteria are not satisfied}
        \STATE Compute $g^{(i)}=-\grad f(\tensY^{(i)})=\proj_{\tangent_{\tensY^{(i)}}\!\tensM_{\underline{\vecr}^{(t)}}}(-\nabla f(\tensY^{(i)}))$ by~\eqref{eq: orth proj onto tangent space Tucker}.
        \STATE Choose stepsize $s^{(i)}$ by Armijo backtracking line search~\eqref{eq: Armijo}.
        \STATE Update $\tensY^{(i+1)}=\proj_{\underline{\vecr}^{(t)}}^{\mathrm{HO}}\left(\tensY^{(i)}+s^{(i)}g^{(i)}\right)$ \revise{and $i=i+1$}.
        \ENDWHILE
        \ENSURE The last iterate ${\tensX}^{(t)}=\tensY^{(i)}$.
    \end{algorithmic}
\end{algorithm}

Let $\{\tensY^{(i)}\}_{i\geq 0}$ be the sequence generated by RGD with $\tensY^{(0)}=\tilde{\tensX}^{(t)}$. The RGD method updates $\tensY^{(i)}$ by
\[\tensY^{(i+1)}=\proj_{\underline{\vecr}^{(t)}}^{\mathrm{HO}}\left(\tensY^{(i)}-s^{(i)}\grad f(\tensY^{(i)})\right).\] 
For the selection of stepsize, we apply the Armijo backtracking line search~\eqref{eq: Armijo} to ensure the convergence. The algorithm terminates if: 
1) rank deficiency is detected, namely, at least one of the mode-$k$ unfolding matrices of $\tensY$ satisfies $\sigma_{\min}(\maty_{(k)}^{(i)})/\sigma_{\max}(\maty_{(k)}^{(i)})\leq\Delta$; 2) the Riemannian gradient satisfies $\|\grad f(\tensY^{(i)})\|_\frob\leq\varepsilon_R^{(t)}$ with threshold $\varepsilon_R^{(t)}>0$. Note that we always check the rank deficiency in prior to the stationarity.

\subsection{Rank-decreasing procedure}\label{subsec: rank decrease}
Given ${\tensX}^{(t)}={\tensG}^{(t)}\times_{k=1}^d{\matu}_k^{(t)}$ returned by Algorithm~\ref{alg: RGD}, if the RGD method terminates upon detecting rank deficiency, we proceed by implementing a rank-decreasing procedure, which is able to reduce the number of parameters and thus save storage. Specifically, we produce a rank-$\hat{\vecr}$ truncation of ${\tensX}^{(t)}$ with $$\hat{r}_k:=\min\{i:\revise{\sigma_{i+1,k}}<\Delta\sigma_{1,k}\}\qquad \text{or}\qquad \hat{r}_k:=\underline{r}_k^{(t)}\ \ \text{if}\ \ \sigma_{\underline{r}_k^{(t)},k}\geq\Delta\sigma_{1,k},$$ 
where $\sigma_{1,k}\geq\cdots\geq\sigma_{\underline{r}_k^{(t)},k}$ are the singular values of ${\matx}_{(k)}^{(t)}$ and $\Delta\in(0,1)$ is a threshold. Subsequently, we yield a truncated low-rank tensor $\proj_{\leq\hat{\vecr}}^{\mathrm{HO}}({\tensX}^{(t)})$. To ensure the convergence, we adaptively shrink $\Delta$ by $\rho_1\Delta$ with $\rho_1\in(0,1)$ until it holds that $f({\tensX}^{(t)})\geq f(\proj_{\leq\hat{\vecr}}^{\mathrm{HO}}({\tensX}^{(t)}))$. Then, we set 
\[\tilde{\tensX}^{(t+1)}=\proj_{\leq\hat{\vecr}}^{\mathrm{HO}}({\tensX}^{(t)})\in\tensM_{\leq\vecr
}\ \ \text{with}\ \ \ranktc(\tilde{\tensX}^{(t+1)})\leq\underline{\vecr}^{(t)}.\]

Figure~\ref{fig: rank decreasing} depicts the rank-decreasing procedure for $d=3$. The detailed rank-decreasing procedure is illustrated in Algorithm~\ref{alg: rank decrease}.

\begin{figure}[htbp]
    \centering
    \includegraphics[width=\textwidth]{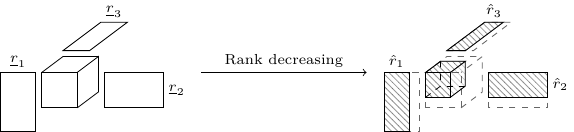}
    \caption{Illustration of rank-decreasing procedure for $d=3$}
    \label{fig: rank decreasing}
\end{figure}

\begin{algorithm}
    \caption{Rank-decreasing procedure}
    \label{alg: rank decrease}
    \begin{algorithmic}[1]
        \REQUIRE ${\tensX}^{(t)}={\tensG}^{(t)}\times_{k=1}^d{\matu}_k^{(t)}$, $\Delta>0$, $\rho_1\in(0,1)$.
        \FOR{$k=1,\dots,d$}
            \STATE Compute the singular values $\sigma_{1,k}\geq\cdots\geq\sigma_{\underline{r}_k^{(t)},k}>0$ of ${\matG}_{(k)}^{(t)}$. 
        \ENDFOR
        \REPEAT
            \STATE Find $\hat{r}_k=\min\{i:\revise{\sigma_{i+1,k}}<\Delta\sigma_{\revise{1,k}}\}$ for $k=1,\dots,d$.
            \STATE Compute $\hat{\tensX}=\proj_{\leq\hat{\vecr}}^{\mathrm{HO}}({\tensX}^{(t)})$.
            \STATE Set $\Delta=\rho_1\Delta$.
        \UNTIL{$f(\hat{\tensX})\leq f({\tensX}^{(t)})$}
        \ENSURE $\tilde{\tensX}^{(t+1)}=\hat{\tensX}$ and Tucker rank $\underline{\vecr}^{(t+1)}=\hat{\vecr}$.
    \end{algorithmic}
\end{algorithm}

\subsection{Rank-increasing procedure}\label{subsec: rank increase}
Given ${\tensX}^{(t)}={\tensG}^{(t)}\times_{k=1}^d{\matu}_k^{(t)}$ returned by Algorithm~\ref{alg: RGD} with initial guess $\tilde{\tensX}^{(t)}$, if ${\tensX}^{(t)}$ is an $\varepsilon_R^{(t)}$-stationary point and $\underline{\vecr}^{(t)}<\vecr$, it is reasonable to consider increasing the rank of ${\tensX}^{(t)}$ in pursuit of higher accuracy. As Remark~\ref{rem: 1} suggests, given a matrix $\matx\in\mathbb{R}^{m\times n}_{\underline{r}}$, adding a matrix in normal part $\normal_{\leq\ell}(\matx)$ can increase the rank of $\matx$. For Tucker tensors, similarly, we observe that
\[\underline{\vecr}^{(t)}<\ranktc({\tensX}^{(t)}+\tensN^{(t)}_{\leq\vecl^{(t)}})\leq\vecr\]
holds for all $\tensN_{\leq\vecl^{(t)}}^{(t)}\in\normal_{\leq\vecl^{(t)}}({\tensX}^{(t)})$ defined in~\eqref{eq: search in TRAM} and $0<\vecl^{(t)}\leq\vecr-\underline{\vecr}^{(t)}$. Therefore, we can implement line search along $\tensN_{\leq\vecl^{(t)}}^{(t)}$ with $\langle\tensN_{\leq\vecl^{(t)}}^{(t)},-\nabla f({\tensX}^{(t)})\rangle\geq 0$ to increase the rank of ${\tensX}^{(t)}$ and decrease the function value at the same time. Specifically, for any ${\matu}_{k,1}^{(t)}\in\St(\ell_k^{(t)},n_k)$ with $({\matu}_{k,1}^{(t)})^\T{\matu}_{k}^{(t)}=0$, the direction 
\[\tensN_{\leq\vecl^{(t)}}^{(t)}:=-\nabla f({\tensX}^{(t)})\times_{k=1}^d\proj_{{\matu}_{k,1}^{(t)}}\in\normal_{\leq\vecl^{(t)}}({\tensX}^{(t)})\]
is always a descent direction. To ensure the convergence, we apply the Armijo backtracking line search~\eqref{eq: Armijo}. 
Subsequently, we yield a new tensor 
\[\tilde{\tensX}^{(t+1)}={\tensX}^{(t)}+s\tensN^{(t)}_{\leq\vecl^{(t)}}\in\tensM_{\leq\vecr}\ \ \text{with}\ \ \ranktc(\tilde{\tensX}^{(t+1)})>\underline{\vecr}^{(t)}.\]
In the sense of tensor space, the rank-increasing procedure updates the tensor ${\tensX}^{(t)}$ to the tensor $\tilde{\tensX}^{(t+1)}$ in a larger tensor space,
\[\bigotimes_{k=1}^d\Span({\matu}_k^{(t)}) \longrightarrow \bigotimes_{k=1}^d\left(\Span({\matu}_k^{(t)})+\Span({\matu}_{k,1}^{(t)})\right).\]

A geometric illustration of the rank-increasing procedure for $d=3$ is depicted in Fig.~\ref{fig: rank increasing}. Algorithm~\ref{alg: rank increase} summarizes the proposed rank-increasing procedure.

\begin{figure}[htbp]
    \centering
    \includegraphics[width=\textwidth]{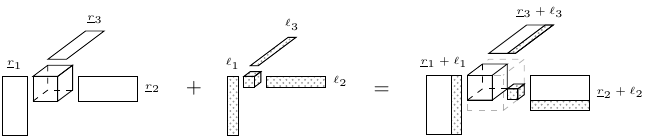}
    \caption{Illustration of rank-increasing procedure for $d=3$}
    \label{fig: rank increasing}
\end{figure}

\begin{algorithm}
    \caption{Rank-increasing procedure}
    \label{alg: rank increase}
    \begin{algorithmic}[1]
        \REQUIRE ${\tensX}^{(t)}={\tensG}^{(t)}\times_{k=1}^d{\matu}_k^{(t)}$, $\vecl^{(t)}$. 
        \STATE Select ${\matu}_{1,1}^{(t)},{\matu}_{2,1}^{(t)},\dots,{\matu}_{d,1}^{(t)}$ \revise{with ${\matu}_{k,1}^{(t)}\in\St(\ell_k^{(t)},n_k)$ and $({\matu}_{k,1}^{(t)})^\T{\matu}_{k}^{(t)}=0$ randomly}.
        \STATE Compute $\hat{\tensG}^{(t)}=-\nabla f({\tensX}^{(t)})\times_{k=1}^d({\matu}_{k,1}^{(t)})^\T$.
        \STATE Choose stepsize $s$ by Armijo backtracking line search~\eqref{eq: Armijo}.
        \STATE Merge the cores $\bar{\tensG}^{(t)}=\diag({\tensG}^{(t)},s\hat{\tensG}^{(t)})$, and the factor matrices $\bar{\matu}_k^{(t)}=\begin{bmatrix}
            {\matu}_k^{(t)} & {\matu}_{k,1}^{(t)}
        \end{bmatrix}$.
        \ENSURE Rank-increased tensor $\tilde{\tensX}^{(t+1)}=\bar{\tensG}^{(t)}\times_{k=1}^d\bar{\matu}_k^{(t)}$ with Tucker rank $\underline{\vecr}^{(t)}+\vecl^{(t)}$.
    \end{algorithmic}
\end{algorithm}

\subsection{Tucker rank-adaptive method}\label{subsec: TRAM}
We propose the Tucker rank-adaptive method (TRAM) to solve the optimization problem~\eqref{eq: problem (P)}, as listed in Algorithm~\ref{alg: TRAM}. 

\begin{algorithm}[htbp]
    \caption{Tucker rank-adaptive method for~\eqref{eq: problem (P)} (TRAM)}
    \label{alg: TRAM}
    \begin{algorithmic}[1]
        \REQUIRE Initial guess $\tensX^{(0)}=\tilde{\tensX}^{(0)}\in\tensM_{\leq\vecr}$ with $\ranktc({\tensX}^{(0)})=\ranktc(\tilde{\tensX}^{(0)})=\underline{\vecr}^{(0)}$; parameters $\varepsilon_R^{(0)}>0$, $\rho_R\in(0,1)$; rank-decreasing parameters $\Delta>0$, $\rho_1\in(0,1)$; rank-increasing parameter $\{\vecl^{(t)}\}_{t\geq 0}$; backtracking parameters $\rho, a\in(0,1), {{s}_{\min}}>0$.
        \WHILE{the stopping criteria are not satisfied}
        \STATE Compute ${\tensX}^{(t)}$ by Algorithm~\ref{alg: RGD} with initial guess $\tilde{\tensX}^{(t)}$ and threshold $\varepsilon_R^{(t)}$. Obtain the Riemannian gradient ${\tensT}^{(t)}=\grad f({\tensX}^{(t)})$. 
        \IF{Algorithm~\ref{alg: RGD} is terminated by detecting rank deficiency}
            \STATE Apply rank-decreasing procedure (Algorithm~\ref{alg: rank decrease}) to ${\tensX}^{(t)}$ and yield $\tilde{\tensX}^{(t+1)}$.
            \IF{$\ranktc(\tilde{\tensX}^{(t+1)})=\underline{\vecr}^{(t)}$}
                \STATE Break.
            \ENDIF
        \ELSE 
            \IF{$\underline{\vecr}^{(t)}=\vecr$}
                \STATE Set $\tilde{\tensX}^{(t+1)}={\tensX}^{(t)}$ and $\varepsilon_R^{(t+1)}=\rho_R^{}\varepsilon_R^{(t)}$.\label{line: 7}
            \ELSE
                \STATE Compute $\tensN_{\leq\vecl^{(t)}}^{(t)}=\hat{\tensG}^{(t)}\times_{k=1}^d{\matu}_{k,1}^{(t)}$ by lines 1--2 in Algorithm~\ref{alg: rank increase}.
                \IF{$\|\tensN_{\leq\vecl^{(t)}}^{(t)}\|_\frob\geq\varepsilon_1\|{\tensT}^{(t)}\|_\frob$ and $\varepsilon_2\|\nabla f({\tensX}^{(t)})\|_\frob\leq\|\tensT^{(t)}\|_\frob$} 
                    \STATE Apply rank-increasing procedure (Algorithm~\ref{alg: rank increase}) and yield $\tilde{\tensX}^{(t+1)}$. \label{line: 11}
                \ELSIF{$\varepsilon_2\|\nabla f({\tensX}^{(t)})\|_\frob>\|\tensT^{(t)}\|_\frob$} 
                    \STATE Update $\tilde{\tensX}^{(t+1)}=\proj_{\leq\vecr}^{\mathrm{HO}}({\tensX}^{(t)}+s^{(t)}\approj_{\tangent_{\tensX^{(t)}}\!\tensM_{\leq\vecr}}(-\nabla f({\tensX}^{(t)})))$ by lines~\ref{alg: GRAP restart}--\ref{line: 4 in GRAP} in Algorithm~\ref{alg: GRAP}.\label{line: 13}
                \ELSE
                    \STATE Set $\tilde{\tensX}^{(t+1)}={\tensX}^{(t)}$ and $\varepsilon_R^{(t+1)}=\rho_R^{}\varepsilon_R^{(t)}$. \label{line: 15}
                \ENDIF
            \ENDIF
        \ENDIF
        \STATE $t=t+1$.
        \ENDWHILE
        \ENSURE $\tensX^{(t)}$
    \end{algorithmic}
\end{algorithm}

The method begins with the execution of Algorithm~\ref{alg: RGD} using an initial guess of $\tilde{\tensX}^{(t)}$ and returns a result ${\tensX}^{(t)}$. Depending on different properties of ${\tensX}^{(t)}$, the rank adjustment proceeds as follows. If ${\tensX}^{(t)}$ is found to be rank-deficient, then the rank-decreasing procedure in Algorithm~\ref{alg: rank decrease} is activated to prevent the potential rank degeneracy. If not, one can consider increasing the rank to improve the accuracy. To this end, we first check
\begin{equation}
    \label{eq: increase accept}
    \|\tensN_{\leq\vecl^{(t)}}^{(t)}\|_\frob\geq\varepsilon_1\|{\tensT}^{(t)}\|_\frob \quad\text{and}\quad \varepsilon_2\|\nabla f({\tensX}^{(t)})\|_\frob\leq\|\tensT^{(t)}\|_\frob
\end{equation}
with ${\tensT}^{(t)}:=\grad f({\tensX}^{(t)})$, which implies that rank increasing does work. We implement the rank-increasing procedure in Algorithm~\ref{alg: rank increase} to increase the rank. Otherwise, in view of~\eqref{eq: Tucker tangent cone} and Fig.~\ref{fig: core tensor in tangent cone}, we observe that\[\tangent_{{\tensX}^{(t)}}\!\tensM_{\underline{\vecr}^{(t)}}+\normal_{\leq\vecl^{(t)}}({\tensX}^{(t)})\subsetneq\tangent_{{\tensX}^{(t)}}\!\tensM_{\leq\vecr}.\]
Therefore, we check the restart criterion 
\begin{equation}
    \label{eq: restart criterion}
    \varepsilon_2\|\nabla f({\tensX}^{(t)})\|_\frob\geq\|{\tensT}^{(t)}\|_\frob.
\end{equation}
If the criterion holds, we resort to line search along \revise{$\approj_{\tangent_{\tensX^{(t)}}\!\tensM_{\leq\underline{\vecr}^{(t)}}}(-\nabla f({\tensX}^{(t)}))$ by lines~\ref{alg: GRAP restart}--\ref{line: 4 in GRAP} in Algorithm~\ref{alg: GRAP}}.

\begin{remark}
    In practice, one can always improve the approximation error by increasing rank since the search space is enlarged (e.g., \cite[\S 4.9]{steinlechner2016riemannian}). However, applying GRAP, rfGRAP or Riemannian conjugate gradient method with a fixed (large) rank parameter $\vecr$ can result in severe overfitting (see., e.g., section~\ref{subsec: synthetic} and~\cite[\S 4.3]{kressner2014low}). The proposed rank-increasing procedure increases the rank of an iterate only when the search direction $\tensN_{\leq\vecl^{(t)}}^{(t)}$ is dominant in the sense of~\eqref{eq: increase accept}. Additionally, the rank-increasing procedure enjoys theoretical guarantees; see Theorem~\ref{thm: TRAM}.     
\end{remark}

\subsection{Convergence results}\label{subsec: TRAM convergence}
Let $\{\tensX^{(t)}\}_{t\geq 0}$ be an infinite sequences generated by Algorithm~\ref{alg: TRAM}. Note that in view of Algorithms~\ref{alg: RGD}--\ref{alg: rank decrease},  $\tensX^{(t+1)}$ satisfies
\[f(\tensX^{(t+1)})\leq f(\tilde{\tensX}^{(t+1)})\leq f(\tensX^{(t)}),\]
i.e., $\{f(\tensX^{(t)})\}_{t\geq 0}$ is nonincreasing. Subsequently, we prove the following global convergence of TRAM. 
\begin{lemma}\label{prop: TRAM 1}
    Let $\{\tensX^{(t)}\}_{t\geq 0}$ be an infinite sequence generated by Algorithm~\ref{alg: TRAM}. Assume that $f$ is bounded below \revise{by $f^*$}. Then, it holds that 
    \[\revise{\liminf_{t\to\infty}\|\grad f(\tensX^{(t)})\|_\frob=\liminf_{t\to\infty}\|\proj_{\tangent_{{\tensX}^{(t)}}\!\tensM_{\underline{\vecr}^{(t)}}}(\nabla f({\tensX}^{(t)}))\|_\frob=0.}\]
\end{lemma}
\begin{proof}
    See Appendix~\ref{app: TRAM 1}.
    \qed
\end{proof}

By using Lemma~\ref{prop: TRAM 1}, we can prove a stronger result as follows.
\begin{theorem}\label{thm: TRAM}
    Let $\{\tensX^{(t)}\}_{t\geq 0}$ be an infinite sequence generated by Algorithm~\ref{alg: TRAM}. Assume that $f$ is bounded below \revise{by $f^*$}. Then, it holds that 
    \[\liminf_{t\to\infty}\|\proj_{\tangent_{{\tensX}^{(t)}}\tensM_{\leq\vecr}}(-\nabla f({\tensX}^{(t)}))\|_\frob=0.\]
\end{theorem}
\begin{proof}
    Let $\{\tilde{\tensX}^{(t)}\}_{t\geq 0}$ be an infinite sequence generated by Algorithm~\ref{alg: TRAM}. \revise{Recall ${\tensT}^{(t)}=\grad f({\tensX}^{(t)})$.} It follows from Lemma~\ref{prop: TRAM 1} that there exists a subsequence $\{\tensX^{(t_j)}\}_{j\geq 0}$, such that $\|\tensT^{(t_j)}\|_\frob\leq\varepsilon_R^{(t_j)}$ and $\lim_{j\to\infty}\|\tensT^{(t_j)}\|_\frob=0$. Assume that $\|\nabla f({\tensX}^{(t_j)})\|_\frob\geq\varepsilon_0$ holds for all $j\geq 0$ and some $\varepsilon_0>0$. Otherwise, \revise{the result is straightforward}.

    \revise{If line~\ref{line: 7} in Algorithm~\ref{alg: TRAM} is executed infinitely, there exists a subsequence $\{\tensX^{(t_{j_l})}\}_{l\geq 0}$ of $\{\tensX^{(t_j)}\}_{j\geq 0}$, such that $\ranktc(\tensX^{(t_{j_l})})=\vecr$. Therefore, 
    \[\lim_{l\to\infty}\|\proj_{\tangent_{{\tensX}^{(t_{j_l})}}\!\tensM_{\leq\vecr}}(-\nabla f({\tensX}^{(t_{j_l})}))\|_\frob
    =\lim_{l\to\infty}\|\tensT^{(t_{j_l})}\|_\frob
    =0.\]
    Otherwise, } since $\|\tensT^{(t_j)}\|_\frob$ converges to $0$ and \revise{$\|\nabla f({\tensX}^{(t_j)})\|_\frob\geq\varepsilon_0$, it follows from~\eqref{eq: restart criterion} that} the restart in line~\ref{line: 13} will be continuously executed for sufficiently large $j$, it follows from the backtracking line search in line~\ref{line: 13} that
    \begin{equation*}
        \begin{aligned}
            f({\tensX}^{(t_j)})-f(\tilde{\tensX}^{(t_j+1)})&\geq s_{\min}a\,\|\revise{\approj_{\tangent_{\tensX^{(t_j)}}\!\tensM_{\leq\vecr}}\!(-\nabla f({\tensX}^{(t_j)}))}\|_\frob^2\\
            &\geq s_{\min}a\,\revise{\omega^2\|\proj_{\tangent_{\tensX^{(t_j)}}\!\tensM_{\leq\vecr}}\!(-\nabla f({\tensX}^{(t_j)}))\|_\frob^2.}
        \end{aligned}
    \end{equation*}
    Consequently, 
    \[\lim_{j\to\infty}\|\proj_{\tangent_{{\tensX}^{(t_{j})}}\!\tensM_{\leq\vecr}}(-\nabla f({\tensX}^{(t_{j})}))\|_\frob
    =0.\]
    \qed
\end{proof}

\tikzstyle{stateTransition}=[->, thick]
\tikzstyle{decision} = [diamond, draw, align=center, font=\scriptsize, aspect=2]
\begin{figure}[htbp]
    \centering
    \begin{tikzpicture}
        \node[draw, rectangle, font=\scriptsize, rounded corners=1pt, thick, align=center] (Ini) at (0,0) {Initial guess \\  $(\tensX^{(0)},\underline{\vecr}^{(0)})$};

        \node[draw, rectangle, font=\scriptsize, rounded corners=1pt, thick, align=center] (Fix) at ($(Ini)+(0,-40pt)$) {Line search on $\tensM_{\underline{\vecr}^{(t)}}$\\  $({\tensX}^{(t)},\underline{\vecr}^{(t)})$};

        \draw[stateTransition] (Ini) -- (Fix);
        \node[decision, inner sep=0ex] (Dec1) at ($(Fix)+(0,-50pt)$) {Rank deficiency?};
        \draw[stateTransition] (Fix) -- (Dec1);

        \node[decision, inner sep=-1ex] (Dec2) at ($(Dec1)+(0,-60pt)$) {Rank increase?\\ $\|\tensN_{\leq\vecl^{(t)}}^{(t)}\|_\frob\geq\varepsilon_1\|\tensT^{(t)}\|_\frob$\\ \ };

        \draw[stateTransition] (Dec1) -- (Dec2);
        \node[right] at ($(Dec1)-(0,25pt)$) {\scriptsize No};

        \node[draw, rectangle, font=\scriptsize, rounded corners=1pt, thick, align=center] (Update) at ($(Fix)+(110pt,0)$) { Parameters update \\ $\varepsilon_R^{(t+1)}=\rho_R^{}\varepsilon_R^{(t)}$ \\  $(\tensX^{(t)},\vecr^{(t)})\to(\tilde{\tensX}^{(t+1)},\underline{\vecr}^{(t+1)})$};
        \draw[stateTransition] (Update) -- (Fix);

        \draw[stateTransition] (Dec2) -| (Update);
        \node[above] at ($(Dec2)+(80pt,0)$) {\scriptsize No};

        \node[draw, rectangle, font=\scriptsize, rounded corners=1pt, thick, align=center] (Increase) at ($(Dec2)+(0,-75pt)$) {Rank increasing \vspace{2mm}\\ $\tensX^{(t)}+s^{(t)}\tensN_{\leq\vecl^{(t)}}^{(t)}=\tilde{\tensX}^{(t+1)}$ \vspace{-2mm} \\ \includegraphics[width=250pt]{RankIncrease.pdf}};
        \draw[stateTransition] (Dec2) -- (Increase);
        \node[right] at ($(Dec2)-(0,30pt)$) {\scriptsize Yes};

        \node[draw, rectangle, font=\scriptsize, rounded corners=1pt, thick, align=center] (Decrease) at ($(Dec1)+(-110pt,0)$) {Rank decreasing\\ $(\tensX^{(t)},\Delta)\to(\tilde{\tensX}^{(t+1)},\underline{\vecr}^{(t+1)})$\\ \includegraphics[width=80pt]{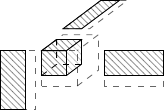}};
        \coordinate (aux) at ($(Decrease)+(0,50pt)$); 
        \draw[stateTransition] (Dec1) -- (Decrease);
        \node[above] at ($(Dec1)-(45pt,0)$) {\scriptsize Yes};
        \draw[stateTransition] (Decrease) -- (aux);
        \draw[stateTransition] (aux) -- (Fix);

        \coordinate (aux) at ($(Increase)-(170pt,-20pt)$);
        \draw[stateTransition] (Increase) -| (aux) |- (Fix);

    \end{tikzpicture}
    \caption{A flowchart of the practical Tucker rank-adaptive method}
    \label{fig: TRAM}
\end{figure}

\subsection{Practical implementation details of TRAM}\label{subsec: practical TRAM}
In practice, the TRAM method is implemented by following the flowchart in Fig.~\ref{fig: TRAM}. \revise{For the rank-increasing procedure,} we notice that applying restart in Algorithm~\ref{alg: TRAM} can be computationally disadvantageous, \revise{as it increases the rank to $\vecr$ in the most likely cases}. This scenario arises when $\tensN_{\leq\vecl^{(t)}}^{(t)}$ is rejected by~\eqref{eq: increase accept}. Instead, we opt to tighten the stopping criteria by setting $\varepsilon_R^{(t+1)}=\rho_R^{}\varepsilon_R^{(t)}$ and proceed with the RGD method again. For the rank-decreasing procedure, let $\{\tensY^{(i)}\}_{i\geq 0}$ be a sequence generated by RGD with $\tensY^{(0)}=\tilde{\tensX}^{(t)}$. In practice, the ratio $\sigma_{\min}(\matY_{(k)}^{(i)})/\sigma_{\max}(\matY_{(k)}^{(i)})$ is computed for every point $\tensY^{(i)}=\tensG^{(i)}\times_{k=1}^d\matu_k^{(i)}\in\tensM_{\underline{\vecr}^{(t)}}$ in Algorithm~\ref{alg: RGD} to detect the rank deficiency. We observe that 
\[\matY_{(k)}^{(i)}=\matu_{k}^{(i)}\matG_{(k)}^{(i)}(\matv_{k}^{(i)})^\T=\matu_{k}^{(i)}\breve{\matu}_{k}^{(i)}\breve{\Sigma}(\breve{\matv}_{k}^{(i)})^\T(\matv_{k}^{(i)})^\T\]
is a SVD of $\matY_{(k)}^{(i)}$, where $\breve{\matu}_{k}^{(i)}\breve{\Sigma}(\breve{\matv}_{k}^{(i)})^\T$ is the SVD of $\matG_{(k)}^{(i)}$. Therefore, it holds that
\[\frac{\sigma_{\min}(\matY_{(k)}^{(i)})}{\sigma_{\max}(\matY_{(k)}^{(i)})}=\frac{\sigma_{\min}(\matG_{(k)}^{(i)})}{\sigma_{\max}(\matG_{(k)}^{(i)})}.\]
We can benefit from it and avoid the explicit large-size construction of $\tensY^{(i)}$ to carry out the rank detection by employing a small-size $\tensG^{(i)}$. Additionally, the condition $f({\tensX}^{(t)})\geq f(\proj_{\leq\hat{\vecr}}^{\mathrm{HO}}({\tensX}^{(t)}))$ in Algorithm~\ref{alg: RGD} will never be checked and thus the rank is indeed decreased.

\section{Numerical experiments}\label{sec: experiments}
In this section, we test the performance of the proposed GRAP (Algorithm~\ref{alg: GRAP}), rfGRAP (Algorithm~\ref{alg: rfGRAP}), TRAM (Algorithm~\ref{alg: TRAM}) and other existing methods on the tensor completion problem. Specifically, given a partially observed tensor $\tensA\in\mathbb{R}^{n_1\times n_2\times\cdots\times n_d}$ on an index set $\Omega\subseteq[n_1]\times[n_2]\times\cdots\times[n_d]$. The goal of Tucker tensor completion is to recover the tensor $\tensA$ from its entries on $\Omega$ based on the low-rank Tucker decomposition. The optimization problem can be formulated on the Tucker tensor variety~$\tensM_{\leq\vecr}$, i.e., 
\begin{equation*}
    \begin{aligned}
        \min\ \ & \frac12\|\proj_\Omega(\tensX)-\proj_\Omega(\tensA)\|_\frob^2\\
        \subjectto\ \ & \quad \quad \tensX\in\tensM_{\leq\vecr},
    \end{aligned}
\end{equation*}
where $\proj_\Omega$ is the projection operator onto $\Omega$, i.e, $\proj_\Omega(\tensX)(i_1,\dots,i_d)=\tensX(i_1,\dots,i_d)$ if~$(i_1,\dots,i_d)\in\Omega$, otherwise $\proj_\Omega(\tensX)(i_1,\dots,i_d)=0$ for $\tensX\in\mathbb{R}^{n_1\times\cdots\times n_d}$. The \emph{sampling rate} is denoted by $p:=|\Omega|/(n_1n_2\cdots n_d)$.

\subsection{Implementation details}
First, we introduce all the default settings and implementation details. In general, the tensor-related implementation of proposed methods is based on the {Tensor-Toolbox v3.4\footnote{Tensor-Toolbox v3.4: \url{http://www.tensortoolbox.org/}}}. All experiments are performed on a workstation with two Intel(R) Xeon(R) Processors Gold 6330 (at 2.00GHz$\times$28, 42M Cache) and 512GB of RAM running Matlab R2019b under Ubuntu 22.04.3. The codes of proposed methods are available at~\url{https://github.com/JimmyPeng1998}.

\paragraph{Computing projections}
Given a tensor $\tensT\in\mathbb{R}^{n_1\times n_2\times\cdots\times n_d}$ and $\tensX=\tensG\times_{k=1}^d\matu_k$ with $\ranktc(\tensX)=\underline{\vecr}$, the proposed methods involve the projections onto the tangent cone~$\tangent_\tensX\!\tensM_{\leq\vecr}$ and Tucker tensor varieties $\tensM_{\vecr}$. We provide the computational details of two projections, $\approj_{\tangent_\tensX\!\tensM_{\leq\vecr}}(\tensT)$ and $\proj_{\leq\vecr}^{\mathrm{HO}}(\tensT)$. In practice, we never manipulate a large full tensor $\tensX$ with $n_1n_2\cdots n_d$ number of parameters in $\approj_{\tangent_\tensX\!\tensM_{\leq\vecr}}(\tensT)$ and~$\proj_{\leq\vecr}^{\mathrm{HO}}(\tensT)$ but core tensor and unfolding matrices.

The approximate projection $\approj_{\tangent_\tensX\!\tensM_{\leq\vecr}}(\tensT)$ in~\eqref{eq: an approximate projection} involves choosing appropriate matrices $\tilde{\matu}_{k,1}\in\St(r_k-\underline{r}_k,n_k)$ with $\tilde{\matu}_{k,1}^\T\matu_k^{}=0$ for $k\in[d]$. 
\revise{We generate a random matrix $\matM_{k,1}\in\mathbb{R}^{n_k\times(r_k-\underline{r}_k)}$ whose elements are i.i.d. samples from the normal distribution $N(0,1)$, and $\tilde{\matu}_{k,1}$ is chosen by the last $(r_k-\underline{r}_k)$ columns of the Q-factor of the matrix~$[\matu_k\ \matM_{k,1}]\in\mathbb{R}^{n_k\times r_k}$.}
Subsequently, the approximate projection onto $\tangent_\tensX\!\tensM_{\leq\vecr}$ can be computed by Algorithm~\ref{alg: orth proj onto tangent cone}. The approach of choosing $\tilde{\matu}_{k,1}$ is also adopted to selecting $\{\matu_{k,1}\}_{k=1}^d$ in Algorithm~\ref{alg: rank increase}. \revise{Since the angle condition is computationally intractable, we do not verify the angle condition in practice.} Additionally, the orthogonal projection onto the tangent space is computed by GeomCG toolbox\footnote{GeomCG toolbox: \url{https://www.epfl.ch/labs/anchp/index-html/software/geomcg/}.}. Note that if $\ranktc(\tensX^{(t)})=\vecr$ in GRAP method, the projection onto the tangent cone is also computed by GeomCG toolbox for fair comparison since $\tangent_{\tensX^{(t)}}\!\tensM_{\leq\vecr}=\tangent_{\tensX^{(t)}}\!\tensM_{\vecr}$.

For the projection $\proj_{\leq\vecr}^{\mathrm{HO}}(\tensT)$ in~\eqref{eq: HOSVD}, in view of Algorithm~\ref{alg: GRAP}, we consider $\tensT$ being in the form of $\tensT=\tensX+\tensV$ with $\tensV\in\tangent_{\tensX}\!\tensM_{\leq\vecr}$. We observe from~\eqref{eq: Tucker tangent cone} that
\[\tensT=\tensX+\tensV\in\bigotimes_{k=1}^d\left(\Span(\matu_k)+\Span(\matu_{k,1})+\Span(\matu_{k,2}\matR_{k,2})\right)\subseteq\tensM_{\leq(\vecr+\underline{\vecr})}.\]
Therefore, the tensor $\tensT$ admits a Tucker decomposition $\tensT=\tilde{\tensG}\times_{k=1}^d\tilde{\matu}_k\in\tensM_{\tilde{\vecr}}$ with some $\tilde{\vecr}\leq\vecr+\underline{\vecr}$. Instead of implementing HOSVD directly to the full tensor~$\tensT\in\mathbb{R}^{n_1\times\cdots\times n_d}$, we exploit its low-rank structure and apply HOSVD to the core tensor~$\tilde{\tensG}\in\mathbb{R}^{\tilde{r}_1\times\cdots\times \tilde{r}_d}$ of $\tensT$, which is much smaller. Specifically, denote the {rank-$\vecr$} HOSVD of $\tilde{\tensG}$ by~$(\tilde{\tensG}\times_{k=1}^d\hat{\matu}_k^\T)\times_{k=1}^d\hat{\matu}_k$, where $\hat{\matu}_k\in\St(r_k,\tilde{r}_k)$ is the leading~${r}_k$ singular vectors of $\tilde{\matG}_{(k)}$. Therefore, it holds that
\[\proj_{\leq\vecr}^{\mathrm{HO}}(\tensT)=((\tilde{\tensG}\times_{k=1}^d\hat{\matu}_k^\T)\times_{k=1}^d\hat{\matu}_k)\times_{k=1}^d\tilde{\matu}_k=(\tilde{\tensG}\times_{k=1}^d\hat{\matu}_k^\T)\times_{k=1}^d(\tilde{\matu}_k\hat{\matu}_k).\]
Note that $\tilde{\matu}_k\hat{\matu}_k\in\St({r}_k,n_k)$ since $(\tilde{\matu}_k\hat{\matu}_k)^\T(\tilde{\matu}_k\hat{\matu}_k)=\matI_{{r}_k}$. This technique is also adopted to the rank-decreasing procedure and the retraction in the Riemannian gradient descent method (Algorithm~\ref{alg: RGD}) on $\tensM_{\underline{\vecr}}$.

\paragraph{Exact line search on tangent cone}
Similar to the optimization on fixed-rank manifold of Tucker tensors~\cite{kressner2014low}, given a point $\tensX^{(t)}$ and a descent direction~$\tensV^{(t)}\in\tangent_{\tensX^{(t)}}\!\tensM_{\leq\vecr}$, the solution of the optimization problem 
\[s_0^{(t)}=\argmin_{s\geq 0}\|\proj_\Omega(\tensX^{(t)}+s\tensV^{(t)})-\proj_\Omega\!\tensA\|_\frob^2\] 
enjoys a closed-form
\[s_0^{(t)}=\frac{\langle\proj_\Omega\!\tensV^{(t)},\proj_\Omega(\tensA-\tensX^{(t)})\rangle}{\langle\proj_\Omega\!\tensV^{(t)},\proj_\Omega\!\tensV^{(t)}\rangle}\geq 0.\]
The computation of $\proj_\Omega\!\tensV^{(t)}$ is implemented in a \texttt{MEX} function.
We adopt $s_0^{(t)}$ as an initial stepsize of Armijo backtracking line search in~\eqref{eq: Armijo}.

\paragraph{Compared methods}
For Tucker-based methods, we compare the proposed methods with a Riemannian conjugate gradient method (GeomCG)~\cite{kressner2014low}, and a Riemannian conjugate gradient method on quotient manifold under a preconditioned metric\footnote{Available at: \url{https://bamdevmishra.in/codes/tensorcompletion/}.} (Tucker-RCG)~\cite{kasai2016low} for optimization on fixed-rank manifold. 

We also compare the proposed methods with other candidates based on different tensor formats. For CP decomposition, we choose the graph-based alternating minimization method\footnote{Available at: \url{https://gitlab.com/ricky7guanyu/tensor-completion-with-regularization-term}.} by Guan et~al.~\cite{guan2020alternating}, denoted by CP-AltMin. We consider the Riemannian conjugate gradient method\footnote{TTeMPS toolbox: \url{https://www.epfl.ch/labs/anchp/index-html/software/ttemps/}.} in~\cite{steinlechner2016riemannian} for tensor train completion, denoted by TT-RCG. For tensor completion in tensor ring decomposition, we consider the Riemannian gradient descent method (TR-RGD)\footnote{LRTCTR toolbox: \url{https://github.com/JimmyPeng1998/LRTCTR}} under a preconditioned metric proposed by Gao et~al.~\cite{gao2024riemannian}.

\paragraph{Stopping criteria}
The performance of all methods is evaluated by the training and test errors
\[\varepsilon_{\Omega}(\tensX):=\frac{\|\proj_\Omega(\tensX)-\proj_\Omega(\tensA)\|_\frob}{\|\proj_\Omega(\tensA)\|_\frob}\quad\text{and}\quad\varepsilon_{\Gamma}(\tensX):=\frac{\|\proj_\Gamma(\tensX)-\proj_\Gamma(\tensA)\|_\frob}{\|\proj_\Gamma(\tensA)\|_\frob},\]
where $\Gamma$ is a test set different from the training set $\Omega$. We terminate the methods if: 1) the training error $\varepsilon_{\Omega}(\tensX^{(t)})<10^{-12}$; 2) the relative change of the training error $(\varepsilon_{\Omega}(\tensX^{(t)})-\varepsilon_{\Omega}(\tensX^{(t-1)}))/\varepsilon_{\Omega}(\tensX^{(t-1)})<10^{-8}$; 3) maximum iteration number is reached; 4) time budget is exceeded. 

\paragraph{Default settings of proposed methods}
The default settings of the proposed methods are reported below. We set $\rho_R=0.5$, $\varepsilon_R^{(0)}=0.1$, rank-decreasing parameters $\Delta=0.01$ and $\rho_1=0.5$, and rank-increasing parameters $\vecl=(1,1,\dots,1)$ and $\varepsilon_1=0.01$ in TRAM. The backtracking parameters are set to be $\rho=0.5$, $a=10^{-4}$ and $s_{\min}=10^{-10}$. Additionally, the maximum iteration number of fixed-rank line search in the TRAM method is set to be $5$.

\subsection{Experiments on synthetic data}\label{subsec: synthetic}
We test the recovery performance of Tucker-based methods on synthetic data. Given $\vecr^*=(r_1^*,r_2^*,\dots,r_d^*)$, we consider a synthetic low-rank tensor $\tensA$ generated by 
\[\tensA=\tensG^*\times_{k=1}^d\matu_k^*,\]
where the entries of $\tensG^*\in\mathbb{R}^{r_1^*\times r_2^*\times\cdots\times r_d^*}$ and $\matu_k^*\in\mathbb{R}^{n_k^{}\times r_k^*}$ are sampled from the normal distribution $N(0,1)$. Then, $\matu_k^*$ is orthogonalized by the QR decomposition. We set $d=3$, $n_1=n_2=n_3=400$, the size of test set $|\Gamma|=pn_1n_2n_3$, and $r_1^*=r_2^*=r_3^*=6$. The initial guess $\tensX^{(0)}$ is generated in a same fashion with given rank $\underline{\vecr}^{(0)}$. A method is terminated if the training error $\varepsilon_\Omega(\tensX^{(t)})\leq 10^{-12}$ or it exceeds the time budget~$200$s.

\paragraph{Test with true rank}
First, we examine the performance of Tucker-based methods with true rank, i.e., $\vecr=\vecr^*=(6,6,6)$. To ensure a fair comparison, we compare the proposed methods with GeomCG and Tucker-RCG with initial guess $\tensX^{(0)}\in\tensM_{\vecr}$. Figure~\ref{fig: synthetic, true rank} reports the test error of Tucker-based methods with sampling rate $p=0.01,0.05$. First, we observe that GRAP and TRAM methods are comparable to GeomCG and Tucker-RCG. Second, rfGRAP method requires more iterations than other candidates, since it only adopts partial information about the tangent cone to avoid retraction. 

\begin{figure}[htbp]
    \centering
    \subfigure{\includegraphics[width=0.48\textwidth]{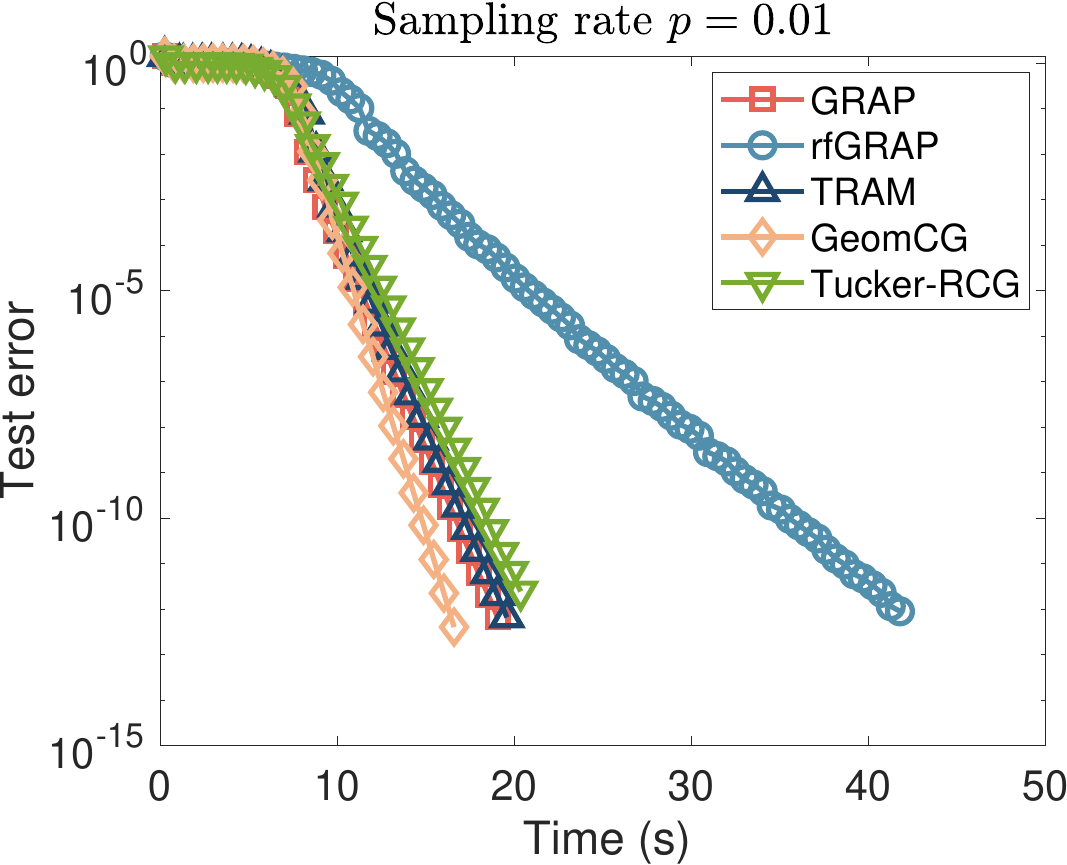}}
    \subfigure{\includegraphics[width=0.48\textwidth]{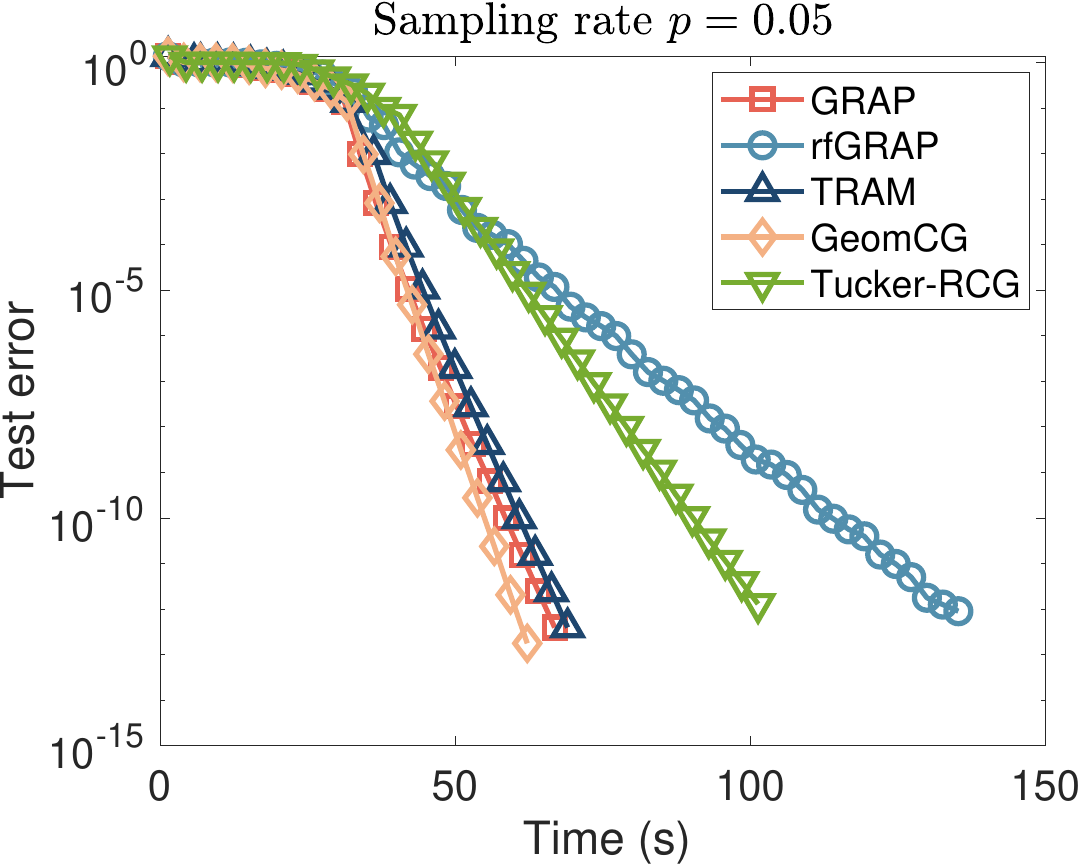}}
    \caption{The recovery performance under sampling rate $p=0.01,0.05$}
    \label{fig: synthetic, true rank}
\end{figure}

\paragraph{Test with under-estimated initial rank}
In contrast with the Riemannian methods on $\tensM_{\vecr}$, the proposed methods can adopt any initial guess $\tensX^{(0)}\in\tensM_{\leq\vecr}$. Therefore, we compare the proposed methods under different initial ranks $\vecr^{(0)}=(r^{(0)},r^{(0)},r^{(0)})$ for $r^{(0)}=1$ and $5$. The sampling rate is chosen as $p=0.05$. Note that we still run GeomCG and Tucker-RCG on~$\tensM_{\vecr^{(0)}}$. The test error is reported in Fig.~\ref{fig: underestimated initial rank}. We observe from Fig.~\ref{fig: underestimated initial rank} that the proposed GRAP and rfGRAP methods have favorably comparable performance than TRAM. A rank-increasing procedure is required to find the true rank $\vecr^*$ in the TRAM method. In addition, the proposed TRAM method can successfully find the true rank $\vecr^*$. However, since $\vecr^{(0)}<\vecr^*$, the GeomCG and Tucker-RCG methods can only obtain a poor low-rank approximation of the data tensor $\tensA$. Therefore, the rank-increasing procedure does allow us to search in a larger space with higher accuracy.

\begin{figure}[htbp]
    \centering
    \includegraphics[width=\textwidth]{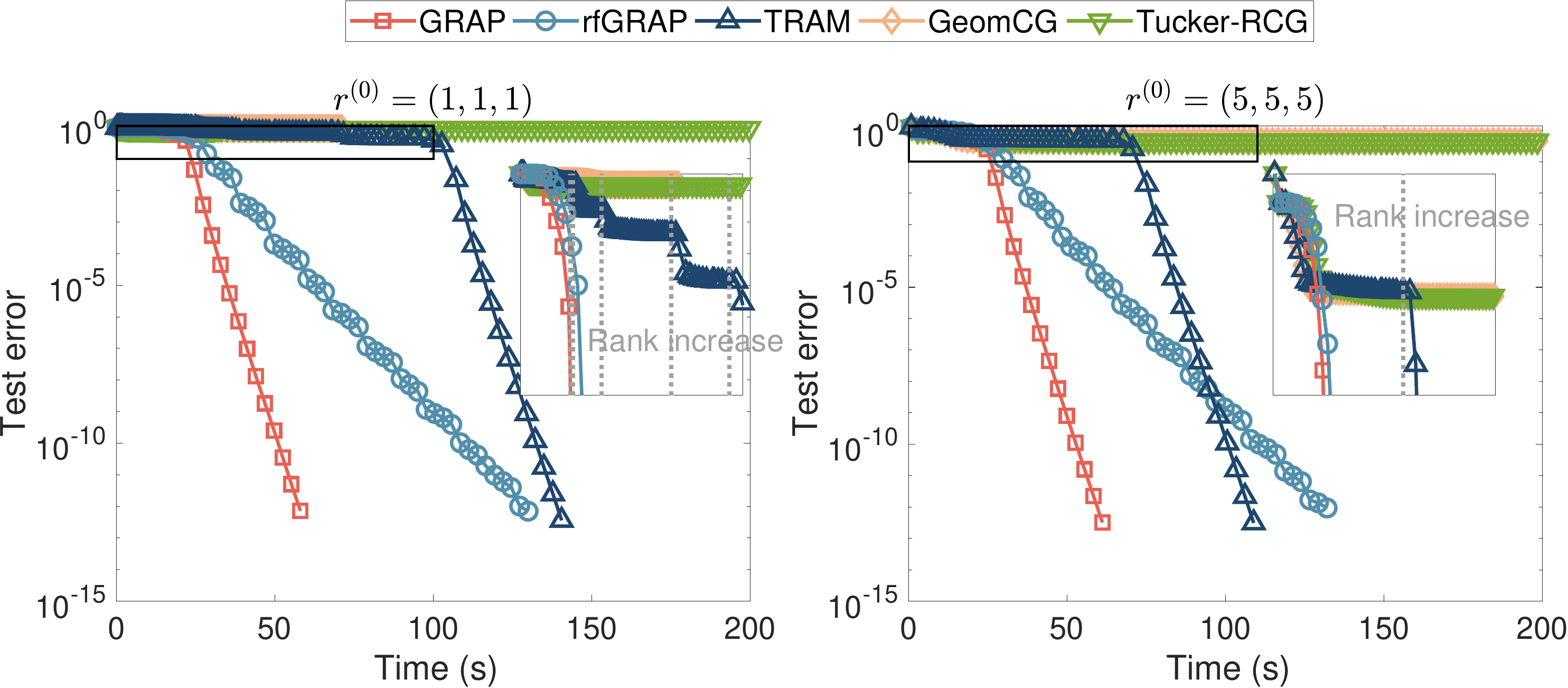}
    \caption{Test error under different initial ranks $\vecr^{(0)}=(1,1,1)$ and $\vecr^{(0)}=(5,5,5)$}
    \label{fig: underestimated initial rank}
\end{figure}

\paragraph{Test with over-estimated rank}
We test the performance of Tucker-based methods under a set of over-estimated ranks $\vecr=(r,r,r)$ with $r=7,8,9,10,11,12>r^*=6$. We set the sampling rate $p=0.01$. To ensure a fair comparision to GeomCG and Tucker-RCG(Q), the initial guess $\tensX^{(0)}$ is generated from $\tensM_{\vecr}$. The numerical results are reported in Figs.~\ref{fig: biased results} and~\ref{fig: biased singular values}. First, we observe from Fig.~\ref{fig: biased results} that the proposed TRAM method converges while other candidates fail to recover the data tensor due to the over-estimated rank parameter. Second, the right figure in Fig.~\ref{fig: biased results} suggests that TRAM successfully recovers the true Tucker rank of the data tensor $\tensA$ under all selections of rank parameter. Therefore, the TRAM performs better than the other candidates. Additionally, Figure~\ref{fig: biased singular values} provides history of the singular values of the unfolding matrices $\matX_{(1)}^{(t)}$, $\matX_{(2)}^{(t)}$, and $\matX_{(3)}^{(t)}$ in TRAM for $\vecr=(8,8,8)$. The proposed TRAM method indeed detects the disparity between the leading six singular values and the subsequent two singular values. The rank-decreasing procedure is activated to reduce the rank parameter to the true rank $\vecr^*$.

\begin{figure}[htbp]
    \centering
    \subfigure{\includegraphics[width=0.48\textwidth]{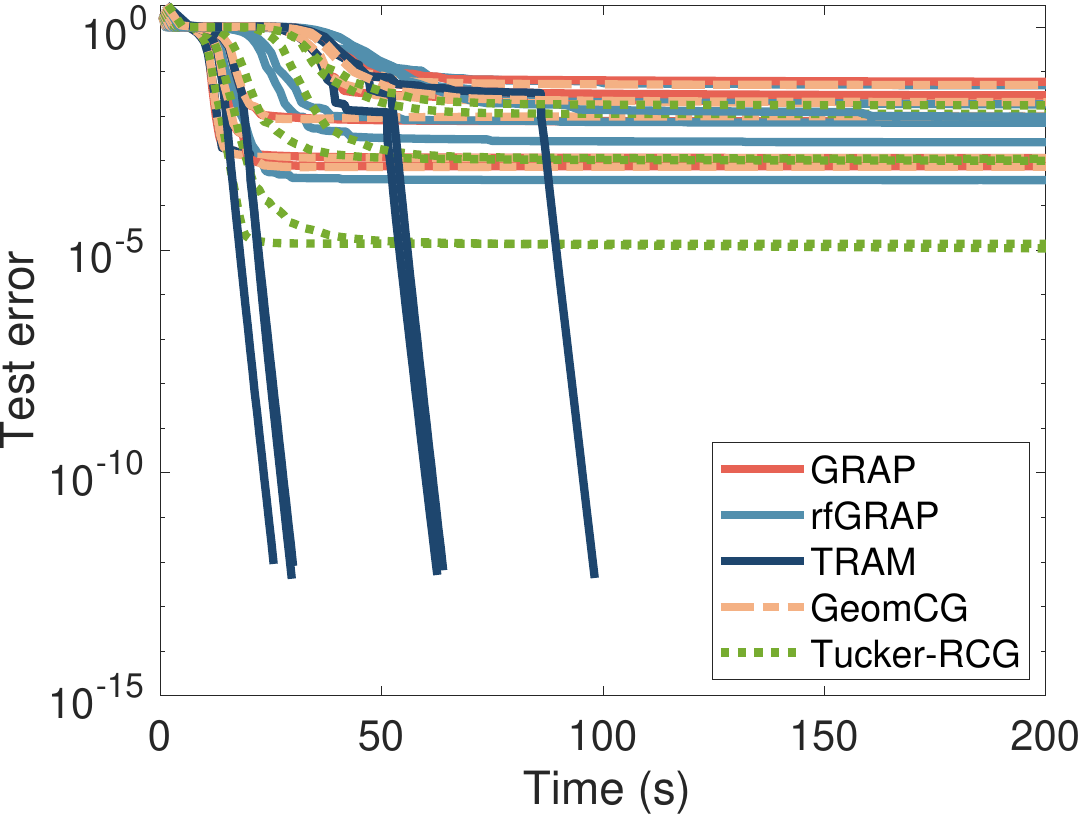}}\
    \subfigure{\includegraphics[width=0.465\textwidth]{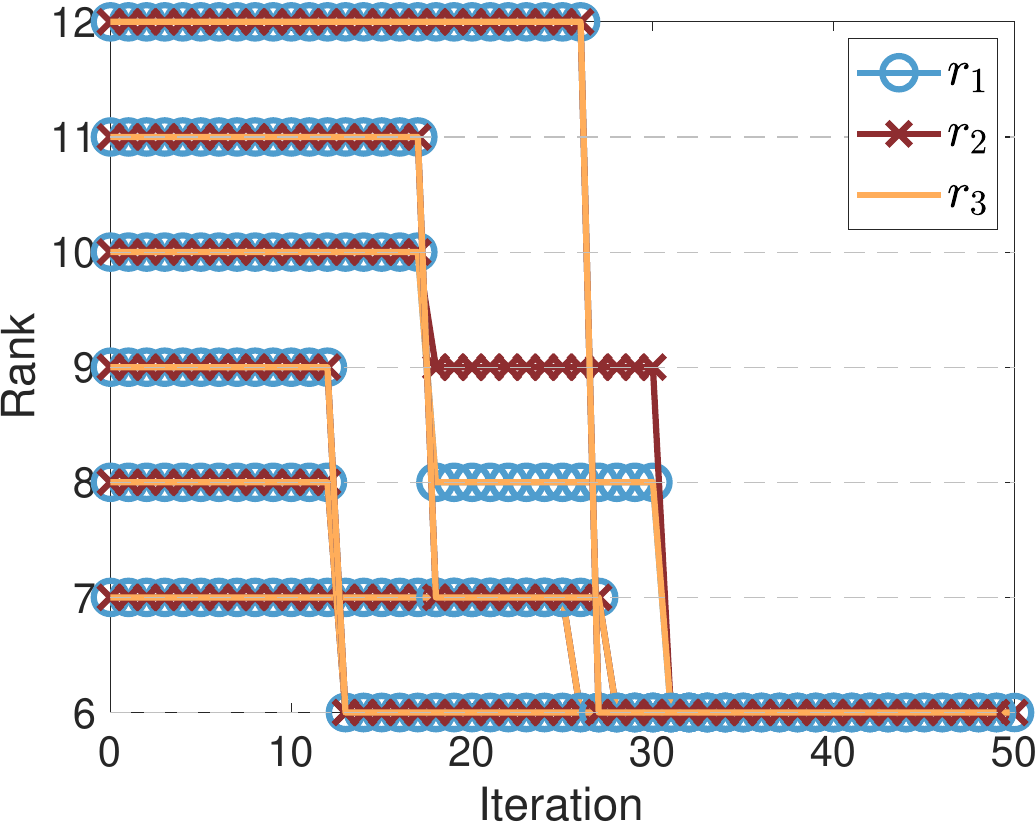}}
    \caption{Numerical results on synthetic dataset under over-estimated rank parameter of Tucker-based methods. Left: test error. Right: rank update of TRAM}
    \label{fig: biased results}
\end{figure}

\begin{figure}[htbp]
    \centering
    \includegraphics[width=\textwidth]{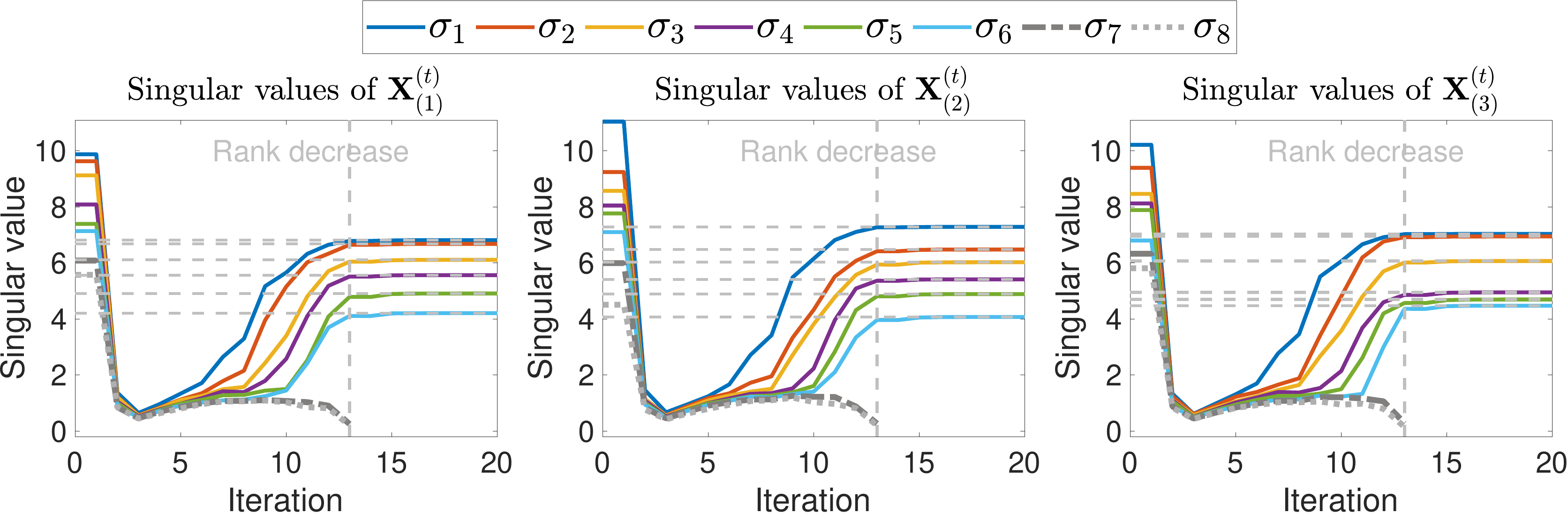}
    \caption{The history of singular values of unfolding matrices $\matX_{(1)}^{(t)}$, $\matX_{(2)}^{(t)}$, and $\matX_{(3)}^{(t)}$ for $\vecr=(8,8,8)$ in TRAM}
    \label{fig: biased singular values}
\end{figure}

\revise{
    Additionally, in order to verify the effect of the rank-increasing procedure, we compare the proposed methods under initial rank $\vecr^{(0)}=(1,1,1)$ and a set of over-estimated ranks $\vecr=(r,r,r)$ with $r=7,8,9,10,11,12>r^*=6$. Figure~\ref{fig: over rank under init} reports the test error and the history of rank update in TRAM method. We observe that only the TRAM method can find the true rank due to the rank-increasing procedure. 

    \begin{figure}[htbp]
        \centering
        \subfigure{\includegraphics[width=0.48\textwidth]{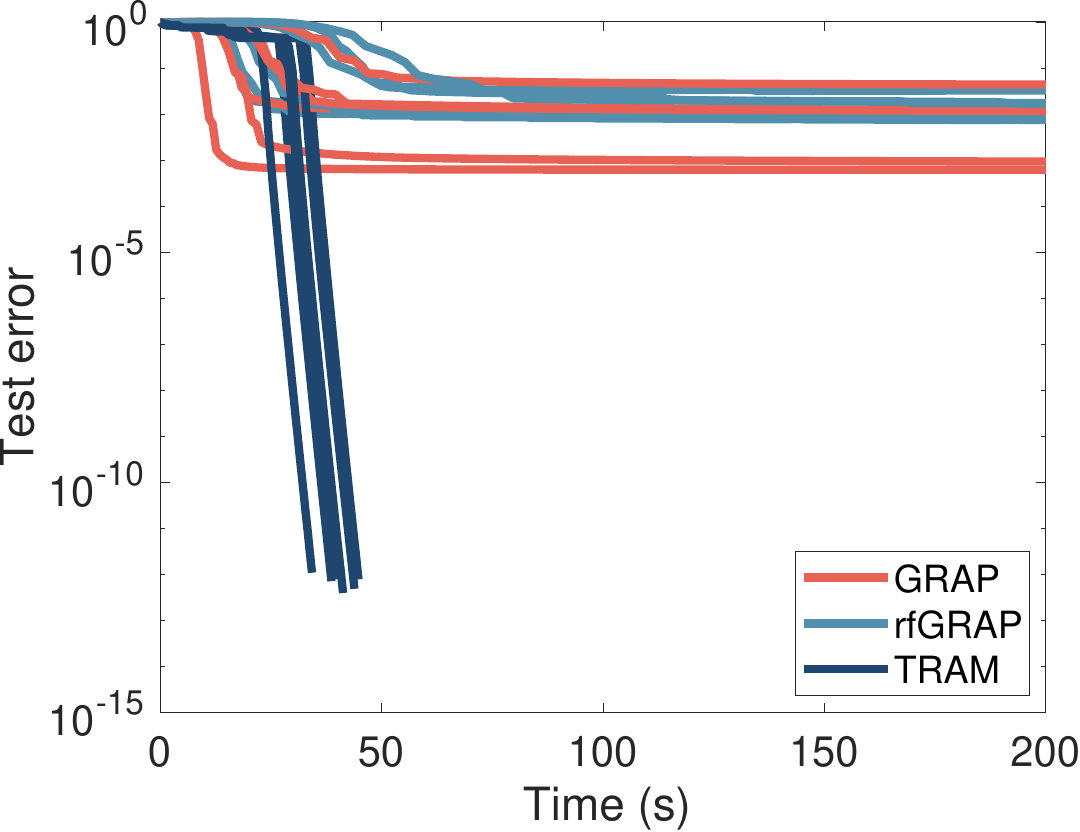}}\
        \subfigure{\includegraphics[width=0.465\textwidth]{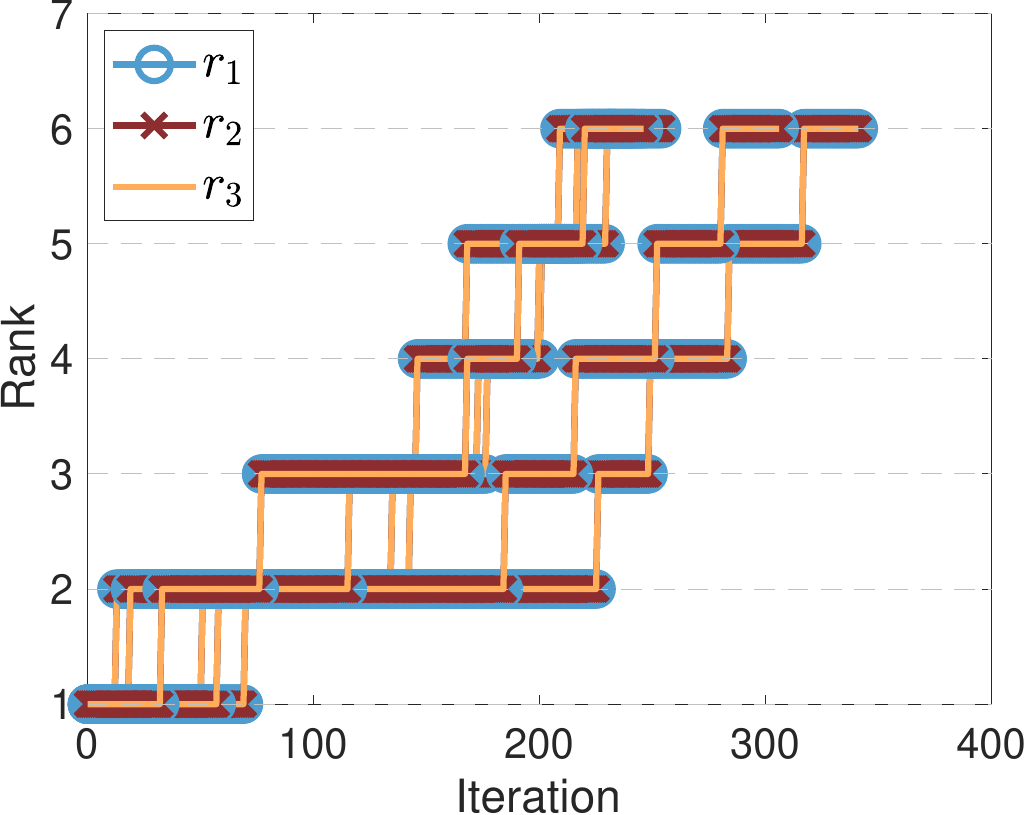}}
        \caption{\revise{Numerical results on synthetic dataset under over-estimated rank parameters and under-estimated initial rank $\vecr^{(0)}=(1,1,1)$. Left: test error. Right: rank update of TRAM}}
        \label{fig: over rank under init}
    \end{figure}

    In summary, both the rank-decreasing and rank-increasing procedure are vital for finding an appropriate rank parameter if the true rank is not available.
}

\subsection{Experiments on hyperspectral images}
In this experiment, we test the performance of proposed methods and other candidates on hyperspectral images, which is formulated as a third order tensor~$\tensA\in\mathbb{R}^{n_1\times n_2\times n_3}$. Mode three of $\tensA$ represents the reflectance level under $n_3$ wavelength values of light. Mode one and two represents the reflectance level of light under different wavelengths. We select the ``Ribeira Hotel Image'' (Ribeira\footnote{Image source: hsi\_32.mat from \url{https://figshare.manchester.ac.uk/articles/dataset/Fifty_hyperspectral_reflectance_images_of_outdoor_scenes/14877285}.}) with size~$249\times 329\times 33$ from {``50 reduced hyperspectral reflectance images''} by Foster~\cite{foster2022colour}, and ``220 Band AVIRIS Hyperspectral Image'' (AVIRIS\footnote{Available at \url{https://purr.purdue.edu/publications/1947/1}.}) with size~$145\times 145\times 220$. Figure~\ref{fig: HSI true} shows the twenty-fourth frame of two hyperspectral images. 

\begin{figure}[htbp]
    \centering
    \subfigure{\includegraphics[height=0.2\textheight]{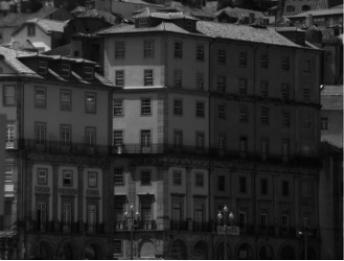}}\
    \subfigure{\includegraphics[height=0.2\textheight]{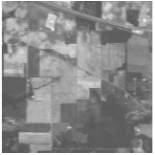}}
    \caption{The twenty-fourth frame of two images. Left: ``Ribeira''. Right: ``AVIRIS''}
    \label{fig: HSI true}
\end{figure}

We evaluate the recovery performance of image completion by the peak signal-to-noise ratio (PSNR) defined by
\[\mathrm{PSNR}:=10\log_{10}\left(n_1n_2n_3\frac{\max(\tensA)}{\|\tensX-\tensA\|_\frob^2}\right),\]
where $\max(\tensA)$ denotes the largest element of $\tensA$. Additionally, the relative error
\[\mathrm{relerr}:=\frac{\|\tensX-\tensA\|_\frob}{\|\tensA\|_\frob}\]
is also reported. The sampling rate is $p=0.1$. We test the Tucker-based methods under the rank parameter $\vecr=(r,r,r)$ with $r=5,10,15,\dots,30$. To ensure a fair comparison, a method is terminated if it reaches the maximum iteration number~250, which is the same as~\cite[\S 5]{kasai2016low}.

\begin{figure}[htbp]
    \centering
    \subfigure{\includegraphics[width=0.48\textwidth]{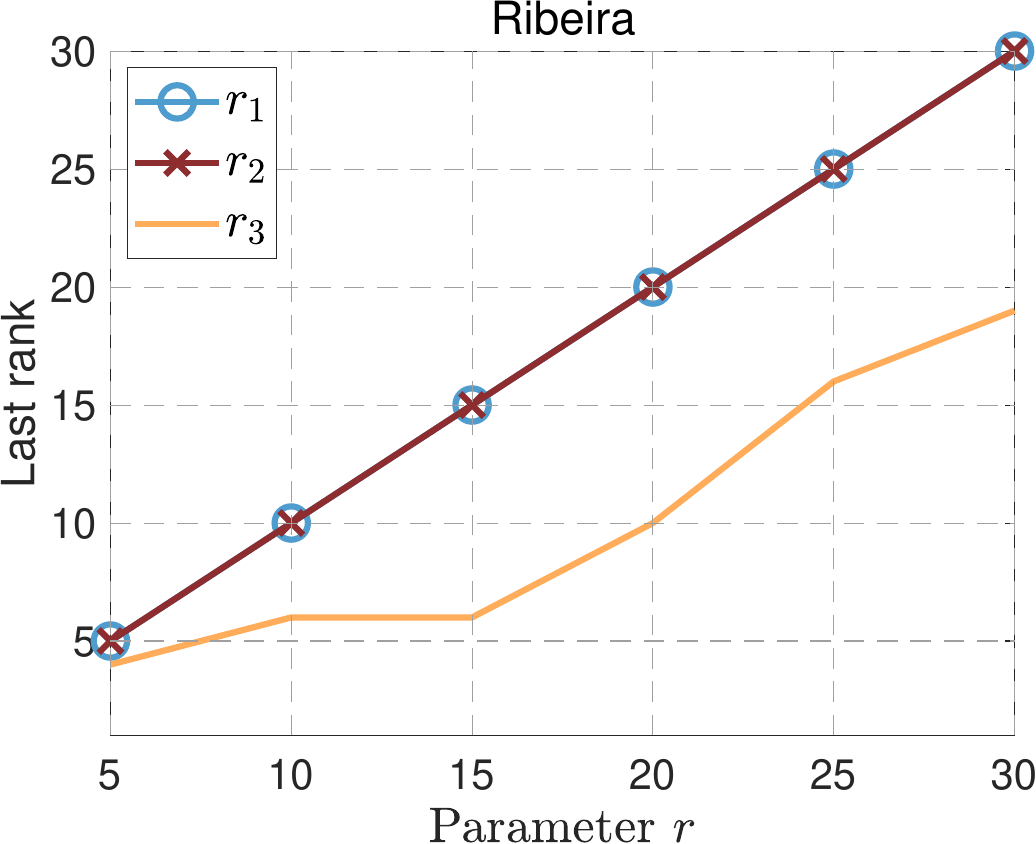}}\ 
    \subfigure{\includegraphics[width=0.48\textwidth]{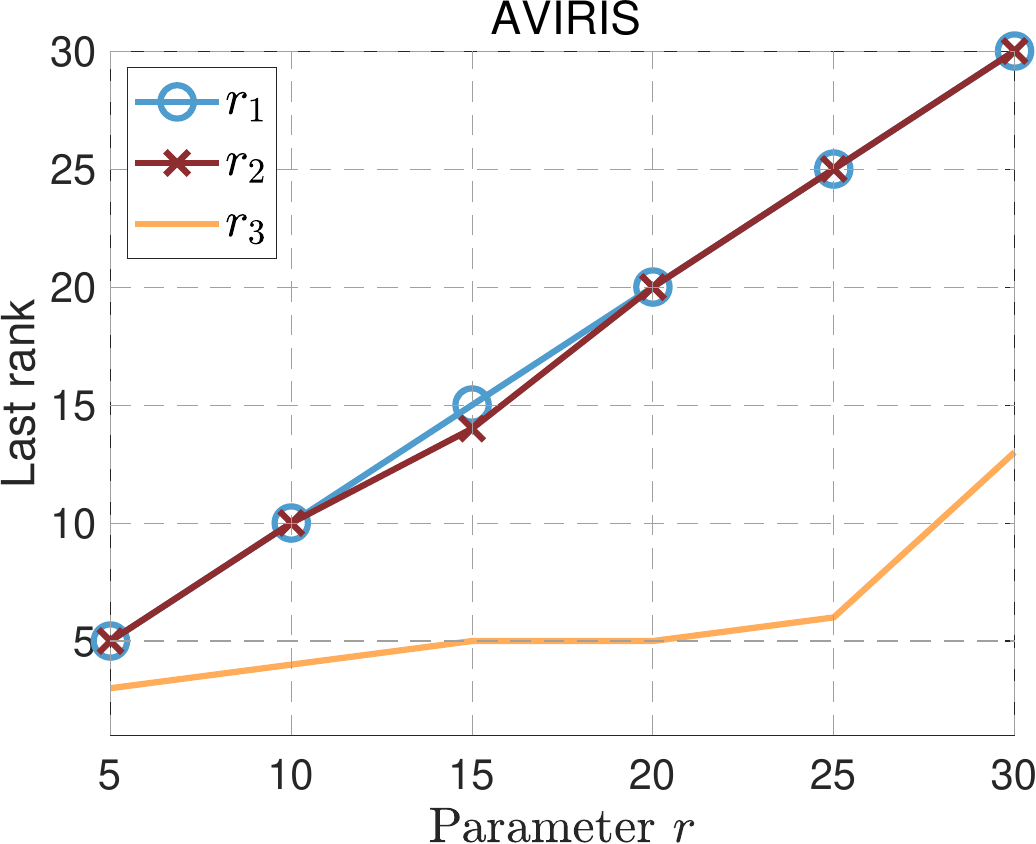}}
    \caption{The last rank obtained from TRAM for ``Ribeira'' and ``AVIRIS'' images under different parameters $\vecr=(r,r,r)$}
    \label{fig: HSI}
\end{figure}

\begin{table}[htbp]
    \centering
    \caption{Relative error and PSNR on ``Ribeira'' and ``AVIRIS'' image} 
    \label{tab: HSI}
    \begin{tabular}{ccrrrrr}
        \toprule 
        Tucker rank $\vecr$ & \multirow{2}*{Results} & \multicolumn{1}{c}{GRAP} & \multicolumn{1}{c}{rfGRAP} & \multicolumn{1}{c}{TRAM} & \multicolumn{1}{c}{GeomCG} & \multicolumn{1}{c}{Tucker-RCG}\\
        \cmidrule{3-7}
        $(r_1,r_2,r_3)$ & & \multicolumn{5}{c}{``Ribeira''}\\
        \midrule
        \multirow{2}*{$(5,5,5)$} & PSNR & \bf 24.9351 & 24.9325 & \bf 24.9351 & \bf 24.9351 & 24.9350 \\ 
        & relerr & \bf 0.2984 & 0.2985 & \bf 0.2984 & \bf 0.2984 & \bf 0.2984 \\ 
        \multirow{2}*{$(10,10,10)$} & PSNR & 26.8481 & 26.8482 &  \bf 26.8648 & 26.8483 & 26.8482 \\ 
        & relerr & 0.2394 & 0.2394 & \bf 0.2389 & 0.2394 & 0.2394 \\ 
        \multirow{2}*{$(15,15,15)$} & PSNR & 28.3451 & 28.3450 & \bf 28.4127 & 28.3451 & 28.3451 \\ 
        & relerr & 0.2015 & 0.2015 & \bf 0.1999 & 0.2015 & 0.2015 \\ 
        \multirow{2}*{$(20,20,20)$} & PSNR & 29.3908 & 29.3934 & \bf 29.5197 & 29.3917 & 29.3924 \\ 
        & relerr & 0.1786 & 0.1786 & \bf 0.1760 & 0.1786 & 0.1786 \\ 
        \multirow{2}*{$(25,25,25)$} & PSNR & 30.2324 & 30.1852 & \bf 30.3897 & 30.2315 & 30.2332 \\ 
        & relerr & 0.1621 & 0.1630 & \bf 0.1592 & 0.1622 & 0.1621 \\
        \multirow{2}*{$(30,30,30)$} & PSNR & 30.7088 & 30.7182 & \bf 30.9921 & 30.7579 & 30.7566 \\ 
        & relerr & 0.1535 & 0.1533 & \bf 0.1486 & 0.1526 & 0.1527 \\ 
        \midrule
         & & \multicolumn{5}{c}{``AVIRIS''}\\
        \midrule
        \multirow{2}*{$(5,5,5)$} & PSNR & \bf 31.7181 & \bf 31.7181 & 31.6955 & \bf 31.7181 & \bf 31.7181 \\ 
        & relerr & \bf 0.0835 & \bf 0.0835 & 0.0837 & \bf 0.0835 & \bf 0.0835 \\ 
        \multirow{2}*{$(10,10,10)$} & PSNR & 33.7393 & 33.7393 & \bf 33.7517 & 33.7393 & 33.7394 \\ 
        & relerr & 0.0661 & 0.0661 & \bf 0.0660 & 0.0661 & 0.0661 \\
        \multirow{2}*{$(15,15,15)$} & PSNR & 35.1308 & 35.1157 & \bf 35.1427 & 35.1144 & 35.1251 \\ 
        & relerr & 0.0564 & 0.0564 & \bf 0.0563 & 0.0565 & 0.0564 \\ 
        \multirow{2}*{$(20,20,20)$} & PSNR & 36.1776 & 36.1777 & \bf 36.5438 & 36.1781 & 36.1780 \\ 
        & relerr & 0.0500 & 0.0500 & \bf 0.0479 & 0.0500 & 0.0500 \\ 
        \multirow{2}*{$(25,25,25)$} & PSNR & 36.6010 & 36.6430 & \bf 37.5433 & 36.6142 & 36.6002 \\ 
        & relerr & 0.0476 & 0.0473 & \bf 0.0427 & 0.0475 & 0.0476 \\ 
        \multirow{2}*{$(30,30,30)$} & PSNR & 36.3106 & 36.4263 & \bf 37.4879 & 36.1278 & 36.1505 \\ 
        & relerr & 0.0492 & 0.0485 & \bf 0.0430 & 0.0502 & 0.0501 \\ 
        \bottomrule
    \end{tabular}
\end{table}

Figure~\ref{fig: HSI} and Table~\ref{tab: HSI} illustrate the recovery results of Tucker-based methods. We observe from Fig.~\ref{fig: HSI} that the low-rank structure along mode three is detected by TRAM, i.e., there exists similarity among different wavelength values of light in the image tensor $\tensA$. It is worth noting that the last rank obtained from TRAM under $\vecr=(15,15,15)$ in ``Ribeira'' image is $(15,15,6)$, which coincides with the rank selection in~\cite[\S 4.3.1]{kressner2014low}. Moreover, Table~\ref{tab: HSI} reports a quantified recovery result. The proposed GRAP and rfGRAP are comparable to GeomCG and Tucker-RCG. Specifically, the proposed TRAM method reaches the highest PSNR and the lowest relative error under most rank parameters.

\subsection{Experiments on ``MovieLens 1M'' dataset}
We consider tensor completion on the real-world dataset {``MovieLens 1M"\footnote{Available at \url{https://grouplens.org/datasets/movielens/1m/}.}}, which consists of $1000209$ movie ratings from 6040 users on 3952 movies from September 19th, 1997 to April 22nd, 1998. By choosing one week as a period, these movie ratings are formulated as a third-order tensor $\tensA$ of size $6040\times 3952\times 150$. We randomly select $80\%$ of the known ratings as a training set $\Omega$ and the rest $20\%$ ratings are test set $\Gamma$. The rank parameter is set to be $\vecr=(r,r,r)$ with $r=1,2,\dots,15$. In addition, we not only compare the performance of the proposed methods to other Tucker-based methods, but also to other methods including CP-AltMin, TT-RCG, and TR-RGD. To ensure a close number of parameters in different tensor decompositions, we choose the CP rank $9$, tensor train rank $(1,4,4,1)$, and $(3,3,3)$ in tensor ring completion. The initial guess $\tensX^{(0)}$ for the proposed methods is generated in the same fashion as section~\ref{subsec: synthetic} by Tucker decomposition. Then, the initial guesses for other methods are transformed by CP-ALS~\cite[Fig. 3.3]{kolda2009tensor}, TT-SVD~\cite[Theorem 2.1]{oseledets2011tensor}, and TR-SVD~\cite[Algorithm 1]{zhao2016tensor} from $\tensX^{(0)}$. Note that initial guesses under different tensor formats have a comparable number of parameters. A method is terminated if it exceeds the time budget of $3000$s.

Figure~\ref{fig: ML1M results tot} demonstrates the numerical results on the ``MovieLens 1M" dataset. We observe that: 1) the proposed methods are favorably comparable to GeomCG and Tucker-RCG with lower test error under different rank parameters $\vecr$; 2) The test error of the TRAM method is less sensitive when $\vecr$ increases, while the test error of the other candidates begins increasing; 3) Figure~\ref{fig: ML1M results tot}(right) presents the last rank obtained from TRAM. The proposed TRAM method indeed adaptively finds an appropriate rank $\vecr^{(t)}$ and reveals the low-rank structure of the categories of movies in mode two of the ``MovieLens 1M" data tensor $\tensA$. 

\begin{figure}[htbp]
    \centering
    \subfigure{\includegraphics[width=0.488\textwidth]{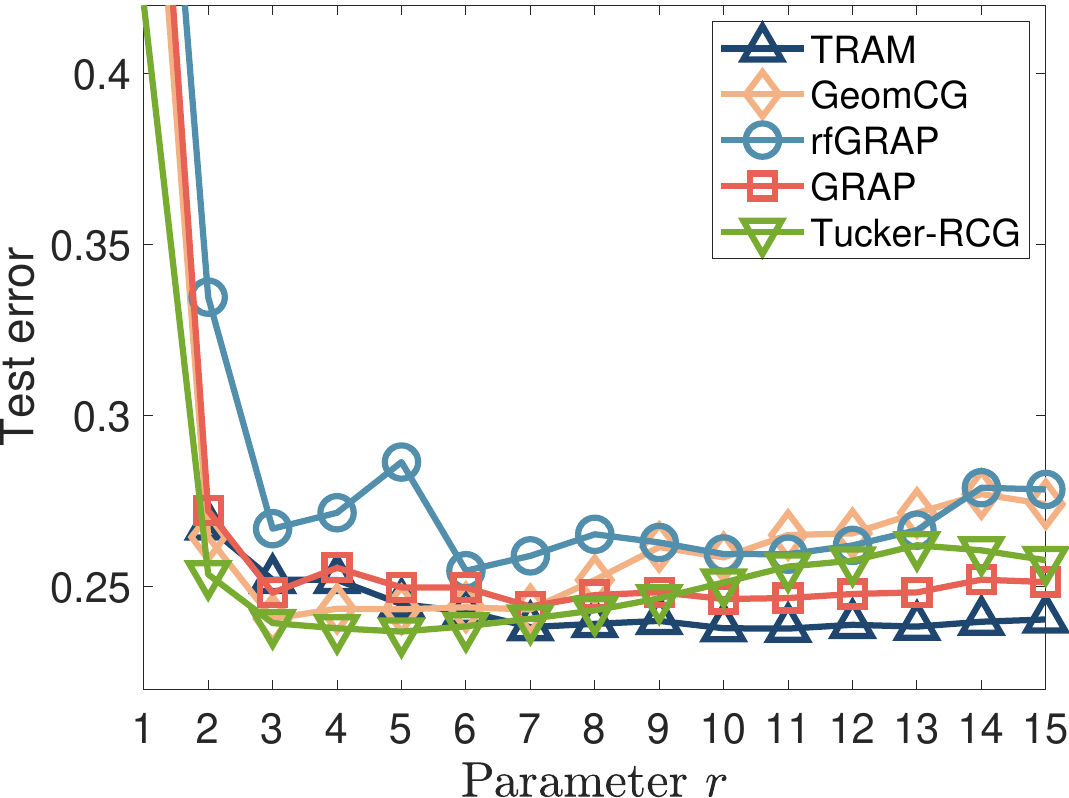}}\ 
    \subfigure{\includegraphics[width=0.468\textwidth]{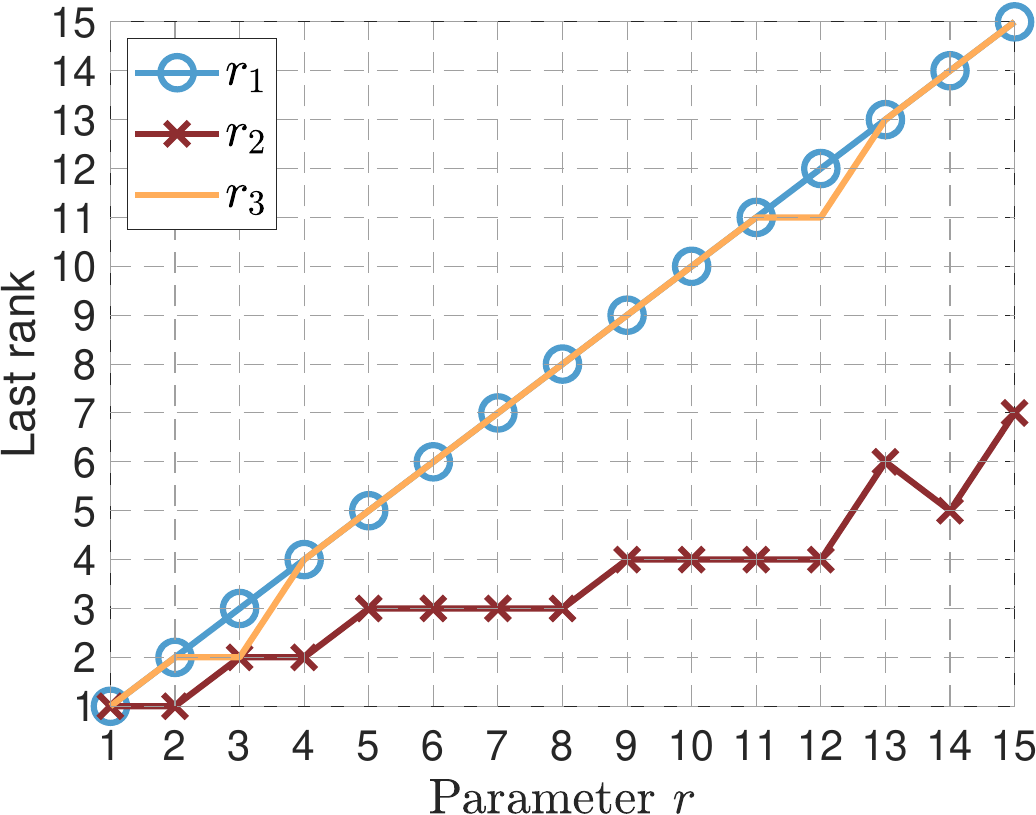}} 
    \caption{Test error and the last rank obtained from TRAM under different rank parameters $\vecr=(r,r,r)$. Left: test error. Right: last rank of TRAM}
    \label{fig: ML1M results tot}
\end{figure}

Morevover, we compare the Tucker-based methods with other methods under rank parameter $\vecr=(9,9,9)$. We observe from Fig.~\ref{fig: ML1M results 5}(left) that the proposed methods are favorably comparable to other candidates, and the TRAM method performs better than the other candidates. Figure~\ref{fig: ML1M results 5}(right) demonstrates that under the rank parameter $\vecr=(9,9,9)$, the parameter $\vecr^{(t)}$ is reduced to $(9,4,9)$. This reduction signifies the identification of four distinct categories of movies within the ``MovieLens 1M" by TRAM.

\begin{figure}[htbp]
    \centering
    \subfigure{\includegraphics[width=0.488\textwidth]{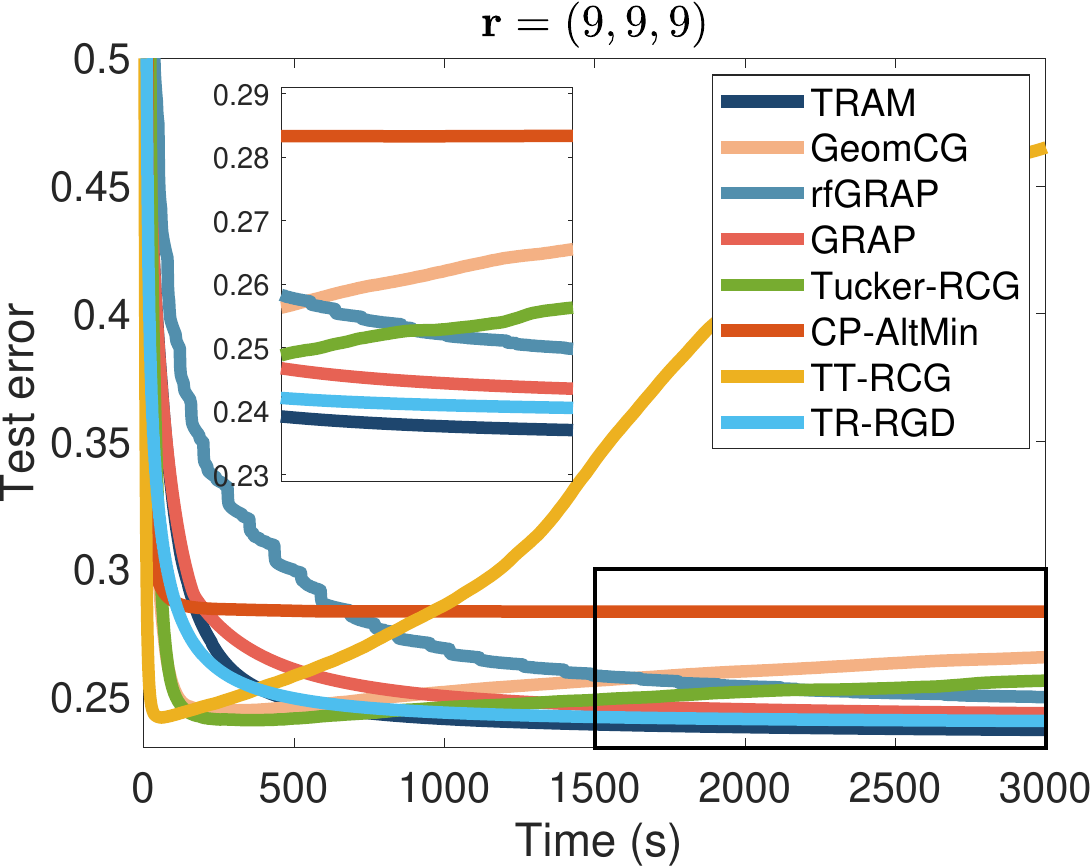}} 
    \subfigure{\includegraphics[width=0.47\textwidth]{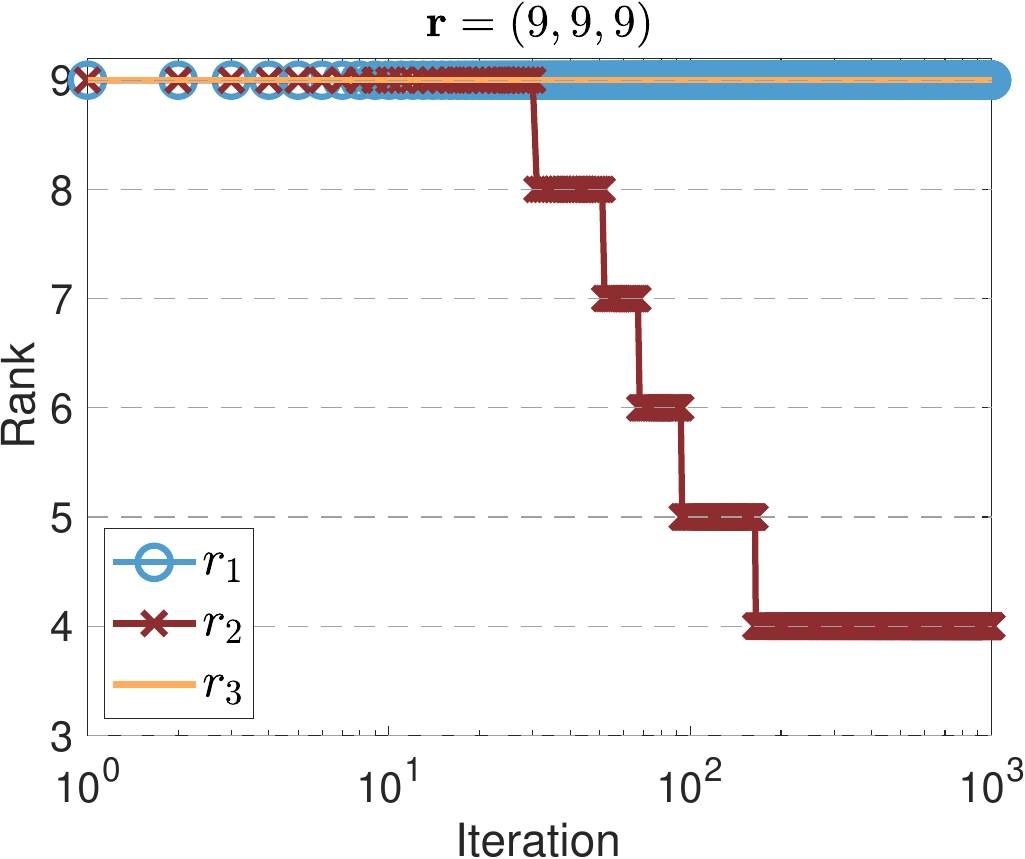}} 
    \caption{Numerical results on ``MovieLens 1M" dataset under rank parameter $\vecr=(9,9,9)$. Left: test error. Right: rank update in iterations}
    \label{fig: ML1M results 5}
\end{figure}

\section{Conclusion and perspectives}\label{sec: conclusion}
In this paper, we have conducted a thorough investigation of the geometry of Tucker tensor varieties and have developed novel geometric and rank-adaptive methods for optimization on Tucker tensor varieties. We observe that the geometry of Tucker tensor varieties is closely connected but much more intricate to the matrix varieties. All of the results can elegantly boil down to the known ones in matrix varieties via geometric illustration figures. Furthermore, the heart of optimization on Tucker tensor varieties is the metric projection. By leveraging the established geometry, we have proposed approximate projections to circumvent the explicit computation of metric projections. Surprisingly, we have observed the retraction-free search directions by using partial information of the tangent cone. Numerical experiments on tensor completion suggest that the proposed methods perform better than existing state-of-the-art methods across various rank parameter selections. In general, when a reliable rank parameter estimation is available, we recommend the GRAP and rfGRAP methods to bypass the lengthy process of rank parameter selection. Conversely, in scenarios where the rank parameter is uncertain, the TRAM method is advised since it adaptively identifies an appropriate rank. To the best of our knowledge, this study represents the first endeavor to explore optimization on Tucker tensor varieties.

In the future, we intend to adopt the proposed methods for optimization on Tucker tensor varieties to other applications, e.g., dynamic low-rank tensor approximation and low-rank solution of high-dimensional partial differential equations. In addition, it is interesting but challenging to design \emph{apocalypse-free} methods on tensor varieties that is guaranteed to converge to a stationary point.





\section*{Acknowledgements}
We would like to thank the editor and two anonymous reviewers for insightful comments. We acknowledge Guillaume Olikier for helpful discussions on the retraction and GRAP method.

\section*{Declaration}
The authors declare that the data supporting the findings of this study are available within the paper. The authors have no competing interests to declare that are relevant to the content of this article.

\appendix

\section{Proof of Proposition~\ref{prop: retraction}}\label{app: retraction}
\begin{proof}
    \revise{It suffices to prove} $\lim_{t\to 0^+}(\retr_\tensX^\mathrm{HO}(t\tensV)-\tensX-t\tensV)/t=0$. Since $\tensV\in\tangent_{\tensX}\!\tensM_{\leq\vecr}$, it follows from~\cite[Proposition 2]{o2004limits} that there exists an analytic arc $\gamma:[0,\epsilon)\to\tensM_{\leq\vecr}$ such that $\gamma(0)=\tensX$ and $\dot{\gamma}(0)=\tensV$.
    
    Moreover, since $\proj_{\leq\vecr}(\tensX+t\tensV)$ is the metric projection of $\tensX+t\tensV$ onto $\tensM_{\leq\vecr}$, it holds that 
    \[\|\tensX+t\tensV-\proj_{\leq\vecr}(\tensX+t\tensV)\|_\frob\leq\|\tensX+t\tensV-\gamma(t)\|_\frob.\]
    By using the quasi-optimality~\eqref{eq: quasi-optimal HOSVD}, we have 
    \[\|\tensX+t\tensV-\retr^\mathrm{HO}_\tensX(t\tensV)\|_\frob\leq\sqrt{d}\|\tensX+t\tensV-\proj_{\leq\vecr}(\tensX+t\tensV)\|_\frob\leq\sqrt{d}\|\tensX+t\tensV-\gamma(t)\|_\frob=o(t),\]
    and thus $\retr_\tensX^\mathrm{HO}$ is a retraction mapping. \qed
\end{proof}

\section{A compact parametrization of the Tucker tangent cone}\label{app: compact Tucker tangent cone}
\revise{
We provide a compact parametrization of an element in $\tangent_\tensX\!\tensM_{\leq\vecr}$ by taking advantage of the compact parametrization~\eqref{eq: new representation of matrix 01} and the singular value decomposition $\matx_{(k)}=\matu_k\matG_{(k)}\matv_k^\T=\matu_k^{}\tilde{\matu}_k^{}\tilde{\Sigma}_{k}^{}\tilde{\matv}_k^\T\matv_k^\T$ in the proof of Theorem~\ref{thm: tangent cone of Tucker}. Note that the parameters $\tensC$ and $\matu_{k,1}$ in~\eqref{eq: Tucker tangent cone 01} have smaller sizes than those in~\eqref{eq: Tucker tangent cone}.

\begin{corollary}\label{coro: compact tangent cone of Tucker}
    Given a Tucker tensor $\tensX=\tensG\times_1\matu_1\cdots\times_d\matu_d\in\mathbb{R}^{n_1\times\cdots\times n_d}$ with $\ranktc(\tensX)=\underline{\vecr}\leq\vecr$, any $\tensV$ in the tangent cone of $\tensM_{\leq\vecr}$ at $\tensX$ can be expressed by 
    \begin{equation}
        \label{eq: Tucker tangent cone 01}
        \begin{aligned}
            \tensV=\tensC\times_{k=1}^d\begin{bmatrix}
                \matu_k & \matu_{k,1}
            \end{bmatrix}+\sum_{k=1}^d\tensG\times_k(\matu_{k,2}\matR_{k,2})\times_{j\neq k}\matu_j,
        \end{aligned}
    \end{equation}
    where $\tensC\in\mathbb{R}^{(\underline{r}_1+\ell_1)\times\cdots\times (\underline{r}_d+\ell_d)}$, $\matR_{k,2}\in\mathbb{R}^{(n_k-\underline{r}_k-\ell_k)\times \underline{r}_k}$, $\matu_{k,1}\in\St(\ell_k,n_k)$ and $\matu_{k,2}\in\St(n_k-\underline{r}_k-\ell_k,n_k)$ are arbitrary that satisfy $[\matu_k\ \matu_{k,1}\ \matu_{k,2}]\in\mathcal{O}(n_k)$ for $k\in[d]$, and $\vecl=(\ell_1,\ell_2,\dots,\ell_d)$ satisfies 
    \[\ell_k=\rank(\proj_{\matu_k}^\perp\!\matv_{(k)}^{}\!\proj_{\matv_{k}\tilde{\matv}_k}^\perp)=\rank([0\ \matu_{k,1}]\matC_{(k)}([\matu_j\ \matu_{j,1}]^{\otimes j\neq k})^\top\proj_{\matv_{k}\tilde{\matv}_k}^\perp)\]
    with $\matv_{k}=(\matu_j)^{\otimes j\neq k}$ and the right singular vectors $\tilde{\matv}_k\in\St(\underline{r}_k,\underline{r}_{-k})$ of $\matG_{(k)}$. Furthermore, the representation~\eqref{eq: Tucker tangent cone 01} is unique in the sense of the right orthogonal group actions on $\matu_{k,1}$ and $\matu_{k,2}$ with $k\in[d]$.
\end{corollary}
\begin{proof}
    We can obtain the parametrization~\eqref{eq: Tucker tangent cone 01} in a same fashion by substituting the parametrization in Fig.~\ref{fig: new representation of matrix tangent cone} by the compact surrogate~\eqref{eq: new representation of matrix 01} with  $\ell_k=\rank(\proj_{\matu_k}^\perp\!\matv_{(k)}^{}\!\proj_{\matv_{k}\tilde{\matv}_k}^\perp)$ in the proof of Theorem~\ref{thm: tangent cone of Tucker}.

    Furthermore, we aim to show that $\Span(\matu_{k,1})=\Span(\proj_{\matu_k}^\perp\!\matv_{(k)}^{}\!\proj_{\matv_{k}\tilde{\matv}_k}^\perp)$. To this end, we have
    \begin{equation*}
        \begin{aligned}
            \proj_{\matu_k}^\perp\!\matv_{(k)}\!\proj_{\matv_{k}\tilde{\matv}_k}^\perp&=\proj_{\matu_k}^\perp\!\Big([\matu_k\ \matu_{k,1}]\matC_{(k)}([\matu_j\ \matu_{j,1}]^{\otimes j\neq k})^\top+\matu_{k,2}\matR_{k,2}\matG_{(k)}\matv_{k}^\T\Big)\!\proj_{\matv_{k}\tilde{\matv}_k}^\perp\\
            &=[0\ \matu_{k,1}]\matC_{(k)}([\matu_j\ \matu_{j,1}]^{\otimes j\neq k})^\top\proj_{\matv_{k}\tilde{\matv}_k}^\perp+\matu_{k,2}^{}\matR_{k,2}^{}\tilde{\matu}_k^{}\tilde{\Sigma}_{k}^{}\tilde{\matv}_k^\T\matv_k^\T\proj_{\matv_{k}\tilde{\matv}_k}^\perp\\
            &=\matu_{k,1}\underline{\matC}_{(k)}([\matu_j\ \matu_{j,1}]^{\otimes j\neq k})^\top\proj_{\matv_{k}\tilde{\matv}_k}^\perp,
        \end{aligned}
    \end{equation*}
    where $\matG_{(k)}=\tilde{\matu}_k^{}\tilde{\Sigma}_{k}^{}\tilde{\matv}_k^\T$ is the SVD of the unfolding matrix $\matG_{(k)}$, and $\underline{\matC}_{(k)}\in\mathbb{R}^{\ell_k \times (\prod_{j\neq k}(\underline{r}_j+\ell_j))}$ consists of the last $\ell_k$ rows of $\matC_{(k)}$. Since $\matu_{k,1}\in\St(\ell_k,n_k)$ and $\ell_k=\rank(\proj_{\matu_k}^\perp\!\matv_{(k)}^{}\!\proj_{\matv_{k}\tilde{\matv}_k}^\perp)$, we obtain that $\Span(\matu_{k,1})=\Span(\proj_{\matu_k}^\perp\!\matv_{(k)}^{}\!\proj_{\matv_{k}\tilde{\matv}_k}^\perp)$. 

    Consequently, the representation~\eqref{eq: Tucker tangent cone 01} is unique in the sense of the right orthogonal group actions on $\matu_{k,1}$ and $\matu_{k,2}$ with $k\in[d]$.
    \qed 
\end{proof}

It is worth noting that since $r_{-k}=r$ for $d=2$ and $k=1,2$, it holds that $\tilde{\matv}_k\in\mathcal{O}(r)$ and thus $\proj_{\matv_{k}\tilde{\matv}_k}^\perp=\proj_{\matv_{k}}^\perp$. Therefore, the compact parametrization~\eqref{eq: Tucker tangent cone 01} coincides with the compact one in matrix case~\eqref{eq: new representation of matrix 01}.
}

\section{Proof of Proposition~\ref{prop: apocalypse of Tucker}}\label{app: apocalypse of Tucker}
\begin{proof}
    $\tensX^*$ admits the Tucker decomposition $\tensG^*\times_1\matu_1^*\cdots\times_d\matu_d^*$, where $\mathcal{G}^*\in\mathbb{R}^{\underline{r}_1\times \underline{r}_2\times\cdots\times \underline{r}_d}$, $\matU_k^*\in\St(\underline{r}_k,n_k)$ for $k\in[d]$. Consider the $k_0$-th unfolding $\matX^*_{(k_0)}=\matU_{k_0}^*\matG_{(k_0)}^*(\matV_{k_0}^*)^\T\in\mathbb{R}^{n_{k_0}\times n_{-k_0}}$, where $\matV_{k_0}^*:=(\matU_j^*)^{\otimes j\neq k_0}$. Since $\underline{r}_{k_0}<n_{k_0}$ and $\underline{r}_{-k_0}<n_{-k_0}$, \revise{there exist matrices} $\vecu\in\mathbb{R}^{n_{k_0}}\setminus\{0\}$ and $\vecv\in\mathbb{R}^{n_{-k_0}}\setminus\{0\}$, such that $(\matU_{k_0}^*)^\T\vecu=0$ and $(\matV_{k_0}^*)^\T\vecv=0$. \revise{We aim to construct a sequence $\{\tensX^{(t)}\}\subseteq\tensM_{\leq\vecr}$ and a function $f$ such that $\tensX^{(t)}$ converges to $\tensX^*$ and $\|\proj_{\tangent_{\tensX^{(t)}}\!\tensM_{\leq\vecr}}(-\nabla f(\tensX^{(t)}))\|_\frob$ converges to $0$, but $\tensX^*$ is not a stationary point of $f$.}

    \revise{First, we consider} the sequence $\tensX^{(t)}\in\tensM_{\leq\vecr}$ defined by $\tensX^{(t)}=\tensG^{(t)}\times_{k=1}^d\begin{bmatrix}
        \matu_k^* & \matu_{k,1}^{}
    \end{bmatrix}$, where $\matu_{k,1}\in\St(r_k-\underline{r}_{k},n_k)$ with $\matu_{k,1}^\T\matu_k^*=0$, $\mathcal{G}^{(t)}(i_1,\dots,i_d):=\mathcal{G}^*(i_1,\dots,i_d)$ if $(i_1,\dots,i_d)\leq\underline{\vecr}$, $\mathcal{G}^{(t)}(i_1,\dots,i_d):=\frac{1}{t}\bar{\mathcal{G}}(i_1-\underline{r}_1,\dots,i_d-\underline{r}_d)$ if $(i_1,\dots,i_d)>\underline{\vecr}$, $\mathcal{G}^{(t)}(i_1,\dots,i_d):=0$ otherwise; and $\bar{\mathcal{G}}\in\mathbb{R}^{(r_1-\underline{r}_1)\times (r_2-\underline{r}_2)\times \cdots\times (r_d-\underline{r}_d)}$ satisfying $\mathrm{rank}_\mathrm{tc}(\bar{\mathcal{G}})=\vecr-\underline{\vecr}$. Then, it holds that $\mathrm{rank}_\mathrm{tc}(\tensX^{(t)})=\vecr$ and $\tensX^{(t)}$ converges to $\tensX^*$. 

    \revise{Subsequently, we construct} the function $f(\tensX)=\vecu^\T\matX_{(k_0)}^*\vecv$. Since $\tensX^*$ is rank-deficient but the gradient $\nabla f(\tensX)=\ten_{(k_0)}(\vecu\vecv^\T)\neq 0$, it follows from \revise{Proposition~\ref{prop: stationary}} that $\tensX^*$ is not a stationary point. Moreover, for all $\tensX^{(t)}\in\tensM_{\vecr}$, we have
    \begin{equation*}
        \begin{aligned}
            \proj_{\tangent_{\tensX^{(t)}}\!\tensM_{\leq\vecr}}(-\nabla f(\tensX^{(t)}))=&\proj_{\tangent_{\tensX^{(t)}}\!\tensM_{\vecr}}(-\nabla f(\tensX^{(t)}))\\
            =&-\left(\ten_{(k_0)}(\vecu\vecv^\T)\times_{k=1}^d (\matU_k^*)^\T\right)\times_{k=1}^d \matU_k^*\\
            &-\sum_{k=1}^d\mathcal{G}^{(t)}\times_k\left(\proj^\perp_{\matU_k^*}(\mathrm{ten}_{(k_0)}(\vecu\vecv^\T)\times_{j\neq k}(\matU_j^*)^\T)_{(k)}(\mathcal{G}^{(t)}_{(k)})^\dagger\right)\times_{j\neq k}\matU_j^*\\
            =&-\mathcal{G}^{(t)}\times_{k_0}\left(\proj^\perp_{\matU_{k_0}^*}\vecu\vecv^\T\matv_{k_0}^*(\mathcal{G}^{(t)}_{(k_0)})^\dagger\right)\times_{j\neq {k_0}}\matU_j^*\\
            =&\,0,
        \end{aligned}
    \end{equation*}
    where we use \eqref{eq: orth proj onto tangent space Tucker} and $\mathrm{ten}_{(k_0)}(\vecu\vecv^\T)\times_{j\neq k}(\matU_j^*)^\T=0$ for $k\neq k_0$. 

    Hence, the triplet $(\tensX^*,\{\tensX^{(t)}\},f)$ is an apocalypse. \qed
\end{proof}

\section{Proof of Lemma~\ref{prop: TRAM 1}}\label{app: TRAM 1}
\begin{proof}

    Let $\{\tilde{\tensX}^{(t)}\}_{t\geq 0}$ be an infinite sequence generated by Algorithm~\ref{alg: TRAM}. First, we claim that there are infinite many $t$ such that $\|{\tensT}^{(t)}\|_\frob\leq\varepsilon_R^{(t)}$, \revise{where ${\tensT}^{(t)}=\grad f({\tensX}^{(t)})$.} Otherwise, according to the stopping criteria of Algorithm~\ref{alg: RGD}, it always returns rank-deficient points for sufficiently large $t$. Hence, the rank-decreasing procedure will be executed continuously, which leads to a contradiction to the fact that rank is finite. Additionally, we \revise{assume} that the update of parameters in lines~\ref{line: 7} and~\ref{line: 15} is executed finitely. If not, $\varepsilon_R^{(t+1)}=\rho_R^{}\varepsilon_R^{(t)}$ \revise{will be} executed infinitely and thus $\lim_{t\to\infty}\varepsilon_R^{(t)}=0$, implying $\liminf_{t\to\infty}\|\proj_{\tangent_{{\tensX}^{(t)}}\!\tensM_{\underline{\vecr}^{(t)}}}(-\nabla f({\tensX}^{(t)}))\|_\frob=0$. 

    Consequently, there exists a subsequence $\{\tensX^{(t_j)}\}_{j\geq 0}$ satisfying $\|\tensT^{(t_j)}\|_\frob\leq\varepsilon_R^{(t_j)}$. We aim to prove that $\|\tensT^{(t_j)}\|_\frob$ converges to $0$. Since $f$ is bounded \revise{from} below, it holds that
    \[0\leq\lim_{j\to\infty} f({\tensX}^{(t_j)})-f(\tilde{\tensX}^{(t_j+1)})\leq\lim_{j\to\infty} f({\tensX}^{(t_j)})-f({\tensX}^{(t_j+1)})=0.\]
    Subsequently, we proceed to discuss two scenarios (rank-increasing procedure and restart) regarding the update of $\tensX^{(t_j)}$ for sufficiently large $j$: 1) if the rank-increasing procedure in line~\ref{line: 11} is executed, it follows from the backtracking line search in Algorithm~\ref{alg: rank increase} and~\eqref{eq: increase accept} that
            \begin{equation*}
                \begin{aligned}
                    f({\tensX}^{(t_j)})-f(\tilde{\tensX}^{(t_j+1)})&\geq s_{\min}a\left\langle\tensN_{\leq\vecl^{(t_j)}}^{(t_j)},-\nabla f({\tensX}^{(t_j)})\right\rangle\\
                    &=s_{\min}a\,\|\tensN_{\leq\vecl^{(t_j)}}^{(t_j)}\|_\frob^2\\
                    &\geq s_{\min}a\,\varepsilon_1^2\|\tensT^{(t_j)}\|_\frob^2;
                \end{aligned}
            \end{equation*} 
            2) if the restart in line~\ref{line: 13} is executed, the search direction \revise{$\approj_{\tangent_{{\tensX}^{(t_j)}}\!\tensM_{\leq\vecr}}(-\nabla f({\tensX}^{(t_j)}))$} is adopted. Therefore, it holds that
            
                \[f({\tensX}^{(t_j)})-f(\tilde{\tensX}^{(t_j+1)})\geq s_{\min}a\,\|\revise{\approj_{\tangent_{{\tensX}^{(t_j)}}\!\tensM_{\leq\vecr}}(-\nabla f({\tensX}^{(t_j)}))}\|_\frob^2\geq s_{\min}a\,\|\tensT^{(t_j)}\|_\frob^2.\]
                Note that the last inequality comes from $\tangent_{{\tensX}^{(t_j)}}\!\tensM_{\underline{\vecr}^{(t_j)}}\subseteq\tangent_{{\tensX}^{(t_j)}}\!\tensM_{\leq\vecr}$.
            
    In summary, we have 
    \[\|\tensT^{(t_j)}\|_\frob^2\leq\frac{f({\tensX}^{(t_j)})-f(\tilde{\tensX}^{(t_j+1)})}{s_{\min}a\min\{\varepsilon_1^2,1\}}\]
    for sufficiently large $j$ and thus it converges to $0$. \qed
\end{proof}

\printbibliography


%
%

\end{document}